\documentclass[11pt,a4paper]{amsart}

\usepackage[utf8]{inputenc} 
\usepackage{hyperref}       
\usepackage{url}            
\usepackage{amsfonts,amsmath,amsthm}       
\usepackage{amsaddr}
\usepackage{tikz}
\usetikzlibrary{calc,shapes,snakes,patterns,matrix,backgrounds,trees,arrows}
\usetikzlibrary{arrows.meta}
\usepackage{pgfplots}
\usepackage{capt-of}
\usepackage{mathtools}
\usepackage{algorithm}
\usepackage{algpseudocode}
\usepackage{xparse}
\usepackage{listings}  
\usepackage[numbered,framed]{mcode}
\usepackage{xcolor}
\usepackage{textcomp}
\usepackage[margin=0.75in]{geometry}
\usepackage{tabularx}
\usepackage{placeins}

\DeclareMathOperator*{\old}{old}
\DeclareMathOperator{\conv}{conv}

\DeclareMathOperator{\adm}{\textsc{Admissible}}
\DeclareMathOperator{\mar}{\textsc{Mark}}
\DeclareMathOperator{\clos}{\textsc{Closure}}
\DeclareMathOperator{\upd}{\textsc{Update Mesh}}
\DeclareMathOperator{\rec}{\textsc{Recoarsen}}
\DeclareMathOperator{\reg}{\textsc{Regularize}}
\DeclareMathOperator{\coar}{\textsc{Coarsen}}
\DeclareMathOperator{\re}{\textsc{Refine}}
\newcommand{\TR}{\textsf{T-R}}
\newcommand{\TRG}{\textsf{T-RG}}
\newcommand{\TRGB}{\textsf{T-RGB}}
\newcommand{\TNVB}{\textsf{T-NVB}}
\newcommand{\QR}{\textsf{Q-R}}
\newcommand{\QRG}{\textsf{Q-RG}}
\newcommand{\QRB}{\textsf{Q-RB}}

\definecolor{pastelred}{rgb}{1.0, 0.41, 0.38}
\definecolor{lightgreen}{rgb}{0.56, 0.93, 0.56}
\definecolor{lightblue}{rgb}{0.53, 0.81, 0.98}
\definecolor{markred}{rgb}{0.8, 0.0, 0.0}
\definecolor{lemon}{rgb}{1.0, 1.0, 0.4}
\definecolor{rednodes}{rgb}{0.7, 0.11, 0.11}
\def\matlab#1{{\small\mcode{#1}}}

\newtheorem*{definition}{Definition}
\newtheorem{theorem}{Theorem}

\newtheorem{bem}{Remark}
\newtheorem{lemma}{Lemma}

\tikzset{
  schraffiert/.style={pattern=north west lines,pattern color=#1},
  schraffiert/.default=black
}
\tikzset{cross/.style={cross out, draw=black, minimum size=2*(#1-\pgflinewidth), inner sep=0pt, outer sep=0pt},
cross/.default={3pt}}

\title[\tiny Local Coarsening Algorithms on Adaptively Refined Meshes in 2D ]{Local Coarsening Algorithms on \\Adaptively Refined Meshes in 2D and \\ their Efficient Implementation in MATLAB}
\author{Stefan A.~Funken \and Anja Schmidt}
\address{Ulm University, Institute for Numerical Mathematics}
\email{stefan.funken@uni-ulm.de, anja.schmidt@uni-ulm.de}
\keywords{Coarsening, Meshes, Grids, Refinement, Adaptive Finite Element Method}
\subjclass[2010]{65M50}

\begin{document}
\maketitle

\begin{abstract}
Adaptive meshing includes local refinement as well as coarsening of meshes. Typically, coarsening algorithms are based on an explicit refinement history. In this work, we deal with local coarsening algorithms that build on the refinement strategies for triangular and quadrilateral meshes implemented in the \texttt{ameshref} package (Funken and Schmidt 2018, 2019). The \texttt{ameshref} package is a \textsc{Matlab}-toolbox for research and teaching purposes which offers the user a certain flexibility in the \textsc{Refine} step of an adaptive finite element method but can also be used in other contexts like computer graphics. This toolbox is now be extended by the coarsening option. In \texttt{ameshref}, no explicit information about the refinement process is stored, but is instead implicit in the data structure. In this work, we present coarsening algorithms that use easy-to-verify criteria to coarsen adaptively generated meshes by exploiting the data structure. Thereby, the desired properties are guaranteed and computational efficiency is maintained. A \textsc{Matlab} implementation and some numerical examples are discussed in this work and are included in full in the toolbox \texttt{ameshcoars} (Funken and Schmidt 2020). 
\end{abstract}

\section{Introduction and Outline}

Adaptive meshing is an important component in various research areas. It is widely used in the context of solving partial differential equations where solutions have local singularities \cite{nochetto2009theory}. In stationary problems this usually involves refining meshes locally. In time-dependent problems, singularities, interfaces or forces may change in time. To capture a moving singularity in the adaptive mesh, degrees of freedom need to be released that were generated in earlier time steps. To this end, it is common to deploy coarsening algorithms to maintain the adaptive efficiency \cite{bartels,alberta}. In addition, coarsening routines can also be used to generate a sequence of coarse and fine meshes in the context of multigrid techniques \cite{multigrid,gooch}. In computer graphics, adaptive meshing algorithms come into play in, e.g., the subdivision surface method \cite{book,daniels,muller} that is in turn, e.g., used in character animation \cite{pixar}.

In \cite{funkenschmidt}, strategies are presented and an efficient implementation is made available in the \texttt{ameshref}-package \cite{web}. This toolbox was designed for teaching and research purposes. It is easy to understand, simple to use, adaptable and reliable. Compared to other existing work, it brings more flexibility to the $\re$ step of an adaptive method.

Local mesh coarsening is the counterpart of local refinement and involves the deletion of nodes. Remeshing the grid after the deletion of nodes while preserving the mesh quality is not an obvious task. To this end, there exist different approaches to coarsening, e.g., edge collapsing \cite{bankxu,gooch,daniels} or making use of the refinement history \cite{bartels,chenzhang,kossa,alberta}. 

Coarsening based on the refinement history typically uses an explicit history tree to invert the refinement \cite{kossa,alberta}. This requires refinement history updates in each step that needs to be stored additionally. Various works have shown that it is not necessary to save the refinement history explicitly to coarsen elements to their corresponding parent element. Instead, relevant information is implicitly contained in the data structure and can be extracted within linear complexity. For example, Chen and Zhang proposed a concept on how to identify admissible-to-coarsening nodes without explicitly storing the history for the newest vertex bisection \cite{chenzhang}. Based on the determination of these admissible-to-coarsening nodes, elements can be coarsened to their parent element. Similar attempts were successful for bisections in any dimensions \cite{bartels}, for the red-green-blue refinement strategy in two dimensions in \cite{RGB} and for quadrilateral meshes using a red refinement in \cite{huang}. To the best of our knowledge, this has not been done for other refinement strategies in two dimensions. For this reason, we bridge the gap and present easy-to-verify criteria to adaptively coarsen meshes generated by the refinement strategies implemented in the \texttt{ameshref}-package \cite{ameshref}. We show that all relevant information is contained in the data structures built within this package and thus a coarsening implementation within linear complexity is possible. We provide a new \textsc{Matlab} toolbox called \texttt{ameshcoars} that implements coarsening efficiently by use of vectorization and built-in functions. The toolbox \texttt{ameshcoars} is created in such a way that it can be fully used in interplay with the \texttt{ameshref} package. These packages are designed for teaching and research and add more flexibility to the $\re$/$\coar$ step of an adaptive method.

\begin{figure}
\begin{center}
\begin{minipage}[t]{0.24\textwidth}
\hspace*{6mm}\begin{tikzpicture}[rotate=90]
\node at (2.5,1) {\TR};
\coordinate (1) at (0,0);
\coordinate (2) at (2,0);
\coordinate (3) at (2,2);
\coordinate (4) at (0,2);
\draw (1) -- (2);
\draw (2) -- (3);
\draw (3) -- (4);
\draw (4) -- (1);
\coordinate (5) at ($(1)!0.5!(2)$);
\coordinate (6) at ($(2)!0.5!(3)$);
\coordinate (7) at ($(3)!0.5!(4)$);
\coordinate (8) at ($(4)!0.5!(1)$);
\coordinate (9) at ($(5)!0.5!(7)$);
\draw (5) -- (7);
\draw (6) -- (8);
\draw (5) -- (6);
\draw (7) -- (8);
\draw (1) -- (3);
\coordinate (10) at ($(1)!0.5!(5)$);
\coordinate (11) at ($(5)!0.5!(9)$);
\coordinate (12) at ($(9)!0.5!(8)$);
\coordinate (13) at ($(8)!0.5!(1)$);
\draw (10) -- (11);
\draw (11) -- (13);
\draw (10) -- (12);
\draw (12) -- (13);
\coordinate (1) at (0,-1);
\coordinate (2) at (0.5,-1);
\coordinate (3) at (0,-0.5);
\draw (1) -- (2);
\draw (2) -- (3);
\draw (3) -- (1);
\coordinate (1) at (0.7,-1);
\coordinate (2) at (1.2,-1);
\coordinate (3) at (0.7,-0.5);
\draw (1) -- (2) ;
\draw (2) -- (3);
\draw (3) -- (1);
\coordinate (5) at ($(1)!0.5!(2)$);
\coordinate (6) at ($(2)!0.5!(3)$);
\coordinate (7) at ($(3)!0.5!(1)$);
\fill (5) circle (1pt);
\fill (6) circle (1pt);
\fill (7) circle (1pt);
\draw (5) -- (6);
\draw (6) -- (7);
\draw (7) -- (5);
\coordinate (A) at (0,-1.7);
\fill[white] (A) circle (0.01pt); 
\end{tikzpicture}
\end{minipage}
\begin{minipage}[t]{0.24\textwidth}
\hspace*{6mm}\begin{tikzpicture}[rotate=90]
\node at (2.5,1) {\TRG};
\coordinate (1) at (0,0);
\coordinate (2) at (2,0);
\coordinate (3) at (2,2);
\coordinate (4) at (0,2);
\draw (1) -- (2);
\draw (2) -- (3);
\draw (3) -- (4);
\draw (4) -- (1);
\coordinate (5) at ($(1)!0.5!(2)$);
\coordinate (6) at ($(2)!0.5!(3)$);
\coordinate (7) at ($(3)!0.5!(4)$);
\coordinate (8) at ($(4)!0.5!(1)$);
\coordinate (9) at ($(5)!0.5!(7)$);
\draw (5) -- (7);
\draw (6) -- (8);
\draw (5) -- (6);
\draw (7) -- (8);
\draw (1) -- (3);
\coordinate (10) at ($(1)!0.5!(5)$);
\coordinate (11) at ($(5)!0.5!(9)$);
\coordinate (12) at ($(9)!0.5!(8)$);
\coordinate (13) at ($(8)!0.5!(1)$);
\draw (10) -- (11);
\draw (11) -- (13);
\draw (10) -- (12);
\draw (12) -- (13);
\draw (12) -- (7);
\draw (11) --(6);
\coordinate (1) at (0,-1);
\coordinate (2) at (0.5,-1);
\coordinate (3) at (0,-0.5);
\draw (1) -- (2);
\draw (2) -- (3);
\draw (3) -- (1);
\coordinate (1) at (0.7,-1);
\coordinate (2) at (1.2,-1);
\coordinate (3) at (0.7,-0.5);
\draw (1) -- (2) ;
\draw (2) -- (3);
\draw (3) -- (1);
\coordinate (5) at ($(1)!0.5!(2)$);
\coordinate (6) at ($(2)!0.5!(3)$);
\coordinate (7) at ($(3)!0.5!(1)$);
\fill (5) circle (1pt);
\fill (6) circle (1pt);
\fill (7) circle (1pt);
\draw (5) -- (6);
\draw (6) -- (7);
\draw (7) -- (5);
\coordinate (1) at (1.4,-1);
\coordinate (2) at (1.9,-1);
\coordinate (3) at (1.4,-0.5);
\draw (1) -- (2);
\draw (2) -- (3);
\draw (3) -- (1);
\coordinate (6) at ($(2)!0.5!(3)$);
\fill (6) circle (1pt);
\draw (6) -- (1);
\coordinate (A) at (0,-1.7);
\fill[white] (A) circle (0.01pt); 
\end{tikzpicture}
\end{minipage}
\begin{minipage}[t]{0.24\textwidth}
\hspace*{6mm}\begin{tikzpicture}[rotate=90]
\node at (2.5,1) {\TRGB};
\coordinate (1) at (0,0);
\coordinate (2) at (2,0);
\coordinate (3) at (2,2);
\coordinate (4) at (0,2);
\draw (1) -- (2);
\draw (2) -- (3);
\draw (3) -- (4);
\draw (4) -- (1);
\coordinate (5) at ($(1)!0.5!(2)$);
\coordinate (6) at ($(2)!0.5!(3)$);
\coordinate (7) at ($(3)!0.5!(4)$);
\coordinate (8) at ($(4)!0.5!(1)$);
\coordinate (9) at ($(5)!0.5!(7)$);
\draw (5) -- (7);
\draw (6) -- (8);
\draw (5) -- (6);
\draw (7) -- (8);
\draw (1) -- (3);
\coordinate (10) at ($(1)!0.5!(5)$);
\coordinate (11) at ($(5)!0.5!(9)$);
\coordinate (12) at ($(9)!0.5!(8)$);
\coordinate (13) at ($(8)!0.5!(1)$);
\draw (10) -- (11);
\draw (11) -- (13);
\draw (10) -- (12);
\draw (12) -- (13);
\coordinate (14) at ($(5)!0.5!(6)$);
\coordinate (15) at ($(8)!0.5!(7)$);
\draw (11) --(14);
\draw (9) --(14);
\draw (2) --(14);
\draw (12) --(15);
\draw(9) --(15);
\draw (4) --(15);
\coordinate (1) at (0,-1);
\coordinate (2) at (0.5,-1);
\coordinate (3) at (0,-0.5);
\draw (1) -- (2);
\draw (2) -- (3);
\draw (3) -- (1);
\coordinate (1) at (0.7,-1);
\coordinate (2) at (1.2,-1);
\coordinate (3) at (0.7,-0.5);
\draw (1) -- (2) ;
\draw (2) -- (3);
\draw (3) -- (1);
\coordinate (5) at ($(1)!0.5!(2)$);
\coordinate (6) at ($(2)!0.5!(3)$);
\coordinate (7) at ($(3)!0.5!(1)$);
\fill (5) circle (1pt);
\fill (6) circle (1pt);
\fill (7) circle (1pt);
\draw (5) -- (6);
\draw (6) -- (7);
\draw (7) -- (5);
\coordinate (1) at (1.4,-1);
\coordinate (2) at (1.9,-1);
\coordinate (3) at (1.4,-0.5);
\draw (1) -- (2);
\draw (2) -- (3);
\draw (3) -- (1);
\coordinate (6) at ($(2)!0.5!(3)$);
\coordinate (7) at ($(3)!0.5!(1)$);
\fill (6) circle (1pt);
\fill (7) circle (1pt);
\draw (6) -- (7);
\draw (6) -- (1);
\coordinate (1) at (2.1,-1);
\coordinate (2) at (2.6,-1);
\coordinate (3) at (2.1,-0.5);
\draw (1) -- (2);
\draw (2) -- (3);
\draw (3) -- (1);
\coordinate (6) at ($(2)!0.5!(3)$);
\fill (6) circle (1pt);
\draw (6) -- (1);
\end{tikzpicture}
\end{minipage}
\begin{minipage}[t]{0.24\textwidth}
\hspace*{6mm}\begin{tikzpicture}[rotate=90]
\node at (2.5,1) {\TNVB};
\coordinate (1) at (0,0);
\coordinate (2) at (2,0);
\coordinate (3) at (2,2);
\coordinate (4) at (0,2);
\draw (1) -- (2);
\draw (2) -- (3);
\draw (3) -- (4);
\draw (4) -- (1);
\coordinate (5) at ($(1)!0.5!(2)$);
\coordinate (6) at ($(2)!0.5!(3)$);
\coordinate (7) at ($(3)!0.5!(4)$);
\coordinate (8) at ($(4)!0.5!(1)$);
\coordinate (9) at ($(5)!0.5!(7)$);
\draw (5) --(7);
\draw (6) -- (8);
\draw (1) -- (3);
\draw (7) -- (8);
\draw (5) -- (6);
\draw (5) -- (8);
\coordinate (10) at ($(1)!0.5!(5)$);
\coordinate (11) at ($(5)!0.5!(9)$);
\coordinate (12) at ($(9)!0.5!(8)$);
\coordinate (13) at ($(8)!0.5!(1)$);
\draw (11) -- (13);
\draw (10) -- (12);
\coordinate (14) at ($(5)!0.5!(6)$);
\coordinate (15) at ($(7)!0.5!(8)$);
\draw (11) -- (14);
\draw (14) -- (15);
\draw (15) -- (12);
\draw (14) -- (2);
\draw (15) -- (4);
\coordinate (1) at (0,-1);
\coordinate (2) at (0.5,-1);
\coordinate (3) at (0,-0.5);
\draw (1) -- (2);
\draw (2) -- (3);
\draw (3) -- (1);
\coordinate (1) at (0.7,-1);
\coordinate (2) at (1.2,-1);
\coordinate (3) at (0.7,-0.5);
\draw (1) -- (2) ;
\draw (2) -- (3);
\draw (3) -- (1);
\coordinate (5) at ($(1)!0.5!(2)$);
\coordinate (6) at ($(2)!0.5!(3)$);
\coordinate (7) at ($(3)!0.5!(1)$);
\fill (5) circle (1pt);
\fill (6) circle (1pt);
\fill (7) circle (1pt);
\draw (5) -- (6);
\draw (6) -- (1);
\draw (7) -- (6);
\coordinate (1) at (1.4,-1);
\coordinate (2) at (1.9,-1);
\coordinate (3) at (1.4,-0.5);
\draw (1) -- (2);
\draw (2) -- (3);
\draw (3) -- (1);
\coordinate (6) at ($(2)!0.5!(3)$);
\coordinate (7) at ($(3)!0.5!(1)$);
\fill (6) circle (1pt);
\fill (7) circle (1pt);
\draw (6) -- (7);
\draw (6) -- (1);
\coordinate (1) at (2.1,-1);
\coordinate (2) at (2.6,-1);
\coordinate (3) at (2.1,-0.5);
\draw (1) -- (2);
\draw (2) -- (3);
\draw (3) -- (1);
\coordinate (6) at ($(2)!0.5!(3)$);
\fill (6) circle (1pt);
\draw (6) -- (1);
\end{tikzpicture}
\end{minipage}\\
\vspace*{5mm}
\begin{minipage}[t]{0.25\textwidth}
\end{minipage}\hfill
\begin{minipage}[t]{0.16\textwidth}
\hspace*{6mm}\begin{tikzpicture}[rotate=90]
\node at (2.5,1) {\QR};
\coordinate (1) at (0,0);
\coordinate (2) at (2,0);
\coordinate (3) at (2,2);
\coordinate (4) at (0,2);
\draw (1) -- (2);
\draw (2) -- (3);
\draw (3) -- (4);
\draw (4) -- (1);
\coordinate (5) at ($(1)!0.5!(2)$);
\coordinate (6) at ($(2)!0.5!(3)$);
\coordinate (7) at ($(3)!0.5!(4)$);
\coordinate (8) at ($(4)!0.5!(1)$);
\coordinate (9) at ($(5)!0.5!(7)$);
\draw (5) -- (7);
\draw (6) -- (8);
\coordinate (10) at ($(1)!0.5!(5)$);
\coordinate (11) at ($(5)!0.5!(9)$);
\coordinate (12) at ($(9)!0.5!(8)$);
\coordinate (13) at ($(8)!0.5!(1)$);
\draw (10) -- (12);
\draw (11) -- (13);
\coordinate (1) at (0,-1);
\coordinate (2) at (0.5,-1);
\coordinate (4) at (0,-0.5);
\coordinate (3) at (0.5,-0.5);
\draw (1) -- (2);
\draw (2) -- (3);
\draw (3) -- (4);
\draw (4) -- (1);
\coordinate (1) at (0.7,-1);
\coordinate (2) at (1.2,-1);
\coordinate (4) at (0.7,-0.5);
\coordinate (3) at (1.2,-0.5);
\draw (1) -- (2) ;
\draw (2) -- (3);
\draw (3) -- (4);
\draw (4) -- (1);
\coordinate (5) at ($(1)!0.5!(2)$);
\coordinate (6) at ($(2)!0.5!(3)$);
\coordinate (7) at ($(3)!0.5!(4)$);
\coordinate (8) at ($(4)!0.5!(1)$);
\fill (5) circle (1pt);
\fill (6) circle (1pt);
\fill (7) circle (1pt);
\fill (8) circle (1pt);
\draw (5) -- (7);
\draw (6) -- (8);
\coordinate (A) at (0,-1.7);
\fill[white] (A) circle (0.01pt); 
\end{tikzpicture}
\end{minipage}\hfill
\begin{minipage}[t]{0.16\textwidth}
\hspace*{-1mm}\begin{tikzpicture}[rotate=90]
\node at (2.5,1) {\QRG};
\coordinate (1) at (0,0);
\coordinate (2) at (2,0);
\coordinate (3) at (2,2);
\coordinate (4) at (0,2);
\draw (1) -- (2);
\draw (2) -- (3);
\draw (3) -- (4);
\draw (4) -- (1);
\coordinate (5) at ($(1)!0.5!(2)$);
\coordinate (6) at ($(2)!0.5!(3)$);
\coordinate (7) at ($(3)!0.5!(4)$);
\coordinate (8) at ($(4)!0.5!(1)$);
\coordinate (9) at ($(5)!0.5!(7)$);
\draw (5) -- (7);
\draw (6) -- (8);
\coordinate (10) at ($(1)!0.5!(5)$);
\coordinate (11) at ($(5)!0.5!(9)$);
\coordinate (12) at ($(9)!0.5!(8)$);
\coordinate (13) at ($(8)!0.5!(1)$);
\draw (10) -- (12);
\draw (11) -- (13);
\draw (2) --(11);
\draw (6) --(11);
\draw (7) --(12);
\draw (4) --(12);
\coordinate (1) at (0,-1);
\coordinate (2) at (0.5,-1);
\coordinate (4) at (0,-0.5);
\coordinate (3) at (0.5,-0.5);
\draw (1) -- (2);
\draw (2) -- (3);
\draw (3) -- (4);
\draw (4) -- (1);
\coordinate (1) at (0.7,-1);
\coordinate (2) at (1.2,-1);
\coordinate (4) at (0.7,-0.5);
\coordinate (3) at (1.2,-0.5);
\draw (1) -- (2) ;
\draw (2) -- (3);
\draw (3) -- (4);
\draw (4) -- (1);
\coordinate (5) at ($(1)!0.5!(2)$);
\coordinate (6) at ($(2)!0.5!(3)$);
\coordinate (7) at ($(3)!0.5!(4)$);
\coordinate (8) at ($(4)!0.5!(1)$);
\fill (5) circle (1pt);
\fill (6) circle (1pt);
\fill (7) circle (1pt);
\fill (8) circle (1pt);
\draw (5) -- (7);
\draw (6) -- (8);
\coordinate (1) at (1.4,-1);
\coordinate (2) at (1.9,-1);
\coordinate (3) at (1.9,-0.5);
\coordinate (4) at (1.4,-0.5);
\draw (1) -- (2) ;
\draw (2) -- (3);
\draw (3) -- (4);
\draw (4) -- (1);
\coordinate (5) at ($(4)!0.5!(1)$);
\fill (5) circle (1pt);
\draw (5) -- (2);
\draw (5) -- (3);
\coordinate (1) at (2.1,-1);
\coordinate (2) at (2.6,-1);
\coordinate (3) at (2.6,-0.5);
\coordinate (4) at (2.1,-0.5);
\draw (1) -- (2) ;
\draw (2) -- (3);
\draw (3) -- (4);
\draw (4) -- (1);
\coordinate (5) at ($(2)!0.5!(3)$);
\coordinate (6) at ($(4)!0.5!(1)$);
\fill (5) circle (1pt);
\fill (6) circle (1pt);
\draw (5) --(6);
\coordinate (1) at (2.8,-1);
\coordinate (2) at (3.3,-1);
\coordinate (3) at (3.3,-0.5);
\coordinate (4) at (2.8,-0.5);
\draw (1) -- (2) ;
\draw (2) -- (3);
\draw (3) -- (4);
\draw (4) -- (1);
\coordinate (5) at ($(1)!0.5!(2)$);
\coordinate (6) at ($(4)!0.5!(1)$);
\fill (5) circle (1pt); 
\fill (6) circle (1pt);
\draw (5) -- (6);
\draw (5) -- (3);
\draw (6) -- (3);
\end{tikzpicture}
\end{minipage}\hfill
\begin{minipage}[t]{0.16\textwidth}
\hspace*{-8mm}\begin{tikzpicture}[rotate=90]
\node at (2.5,1) {\QRB};
\coordinate (1) at (0,0);
\coordinate (2) at (2,0);
\coordinate (3) at (2,2);
\coordinate (4) at (0,2);
\draw (1) -- (2);
\draw (2) -- (3);
\draw (3) -- (4);
\draw (4) -- (1);
\coordinate (5) at ($(1)!0.5!(2)$);
\coordinate (6) at ($(2)!0.5!(3)$);
\coordinate (7) at ($(3)!0.5!(4)$);
\coordinate (8) at ($(4)!0.5!(1)$);
\coordinate (9) at ($(5)!0.5!(7)$);
\draw (5) -- (7);
\draw (6) -- (8);
\coordinate (10) at ($(1)!0.5!(5)$);
\coordinate (11) at ($(5)!0.5!(9)$);
\coordinate (12) at ($(9)!0.5!(8)$);
\coordinate (13) at ($(8)!0.5!(1)$);
\draw (10) -- (12);
\draw (11) -- (13);
\coordinate (14) at ($(5)!0.5!(6)$);
\coordinate (15) at ($(7)!0.5!(8)$);
\coordinate (16) at ($(5)!0.5!(2)$);
\coordinate (17) at ($(4)!0.5!(8)$);
\draw (14) --(11);
\draw (16) --(14);
\draw(14) --(6);
\draw (17) --(15);
\draw (12) --(15);
\draw (15) --(7);
\coordinate (1) at (0,-1);
\coordinate (2) at (0.5,-1);
\coordinate (4) at (0,-0.5);
\coordinate (3) at (0.5,-0.5);
\draw (1) -- (2);
\draw (2) -- (3);
\draw (3) -- (4);
\draw (4) -- (1);
\coordinate (1) at (0.7,-1);
\coordinate (2) at (1.2,-1);
\coordinate (4) at (0.7,-0.5);
\coordinate (3) at (1.2,-0.5);
\draw (1) -- (2) ;
\draw (2) -- (3);
\draw (3) -- (4);
\draw (4) -- (1);
\coordinate (5) at ($(1)!0.5!(2)$);
\coordinate (6) at ($(2)!0.5!(3)$);
\coordinate (7) at ($(3)!0.5!(4)$);
\coordinate (8) at ($(4)!0.5!(1)$);
\fill (5) circle (1pt);
\fill (6) circle (1pt);
\fill (7) circle (1pt);
\fill (8) circle (1pt);
\draw (5) -- (7);
\draw (6) -- (8);
\coordinate (1) at (1.4,-1);
\coordinate (2) at (1.9,-1);
\coordinate (3) at (1.9,-0.5);
\coordinate (4) at (1.4,-0.5);
\draw (1) -- (2) ;
\draw (2) -- (3);
\draw (3) -- (4);
\draw (4) -- (1);
\coordinate (5) at ($(1)!0.5!(2)$);
\coordinate (6) at ($(2)!0.5!(3)$);
\coordinate (A) at ($(3)!0.5!(1)$);
\fill (5) circle (1pt);
\fill (6) circle (1pt);
\draw (5) -- (A);
\draw (6) -- (A);
\draw (A) -- (4);
\coordinate (A) at (0,-1.7);
\fill[white] (A) circle (0.01pt); 
\end{tikzpicture}
\end{minipage}\hfill
\begin{minipage}[t]{0.25\textwidth}
\end{minipage}\\
\flushright
\begin{tikzpicture}
\end{tikzpicture}

\caption{Overview of refinement strategies from the \texttt{ameshref}-toolbox \cite{web,funkenschmidt,ameshref} considered for coarsening.}
\label{fig:mesh}
\end{center}
\end{figure}

In particular, we present adaptive coarsening strategies for meshes generated by the following refinement procedures: For triangular meshes, we consider the \emph{newest vertex bisection} ({\TNVB}) first mentioned in \cite{sewell}, the \emph{red-green-blue} ({\TRGB}) refinement method using reference edges presented in \cite{carstensen2004}, the \emph{red-green} ({\TRG}) refinement strategy proposed in \cite{banksherman1981}, and a \emph{red} ({\TR}) refinement strategy that naturally emerges when  hanging nodes are allowed. For quadrilateral meshes, a \emph{red} ({\QR}) refinement strategy is considered that is a refinement by quadrisection \cite{verfuerth}. This procedure allows hanging nodes in the mesh. To eliminate these hanging nodes, additional patterns are needed. To this end, \cite{bankshermanweiser} proposes a \textsc{Closure} step via the \emph{red-green} ({\QRG}) method. A different approach to eliminate hanging nodes in the mesh was inspired by \cite{kobbelt} and is referred to as \emph{red-blue} ({\QRB}) refinement method. An overview of these refinement strategies is shown in Figure~\ref{fig:mesh}. In each case a mesh is depicted that is refined in the right bottom corner. The patterns next to the meshes are used to eliminate hanging nodes. Note that, for {{\TNVB}} and {{\TRGB}}, reference edges play a crucial role in the refinement method but they are omitted in this illustration.

Our coarsening algorithms follow the framework presented in Algorithm~\ref{alg:algo1} for {\TNVB}, {\TRGB}, {\TR} and {\QR}. For {\TRG}, {\QRG} and {\QRB}, Algorithm~\ref{alg:algo2} is used that calls Algorithm~\ref{alg:algo1}.

\begin{algorithm}[h!]
\caption{$\coar$ framework for \TNVB, \TRGB, {\TR} and \QR}\label{alg:algo1}
\begin{algorithmic}[1]
\State \textbf{Input:} An adaptively generated triangulation $\mathcal{T}$ by \TNVB, \TRGB,{\TR} or {\QR} where $\mathcal{T}_\mathrm{mark}$ is the set of marked elements for coarsening.
\State \textbf{Output:} A coarsened mesh $\hat{\mathcal{T}}$.
\Procedure{CoarsenR/Rgb/Nvb}{$\mathcal{T},\mathcal{T}_\mathrm{mark}$}
\State $\mathcal{A} \gets \adm(\mathcal{T})$ \Comment{Find admissible set $\mathcal{A}$ of elements or nodes for coarsening.}
   \State $\mathcal{A} \gets \mar(\mathcal{A},\mathcal{T}_\mathrm{mark})$ \Comment{Update $\mathcal{A}$ according to the set of marked elements or nodes $\mathcal{T}_\mathrm{mark}$.}
   \State $\mathcal{A} \gets \clos(\mathcal{T},\mathcal{A})$ \Comment{Update $\mathcal{A}$ according to neighboring information.}
    \State $\hat{\mathcal{T}} \gets \upd(\mathcal{T},\mathcal{A})$  \Comment{Coarsen mesh $\mathcal{T}$ according to the set $\mathcal{A}$.}
         \State \textbf{return} $\hat{\mathcal{T}}$
\EndProcedure
\end{algorithmic}
\end{algorithm}

\begin{algorithm}[h!]
\caption{$\coar$ framework for \TRG, {\QRG} and \QRB}\label{alg:algo2}
\begin{algorithmic}[1]
\State \textbf{Input:} An adaptively generated triangulation $\mathcal{T}$ by \TRG, {\QRG} or {\QRB} where $\mathcal{T}_\mathrm{mark}$ is the set of marked elements for coarsening.
\State \textbf{Output:} A coarsened mesh $\hat{\mathcal{T}}$.
\Procedure{CoarsenRg/Rb}{$\mathcal{T},\mathcal{T}_\mathrm{mark}$}
\State $(\mathcal{T}^R,\mathcal{T}_\mathrm{mark}^R) \gets \rec(\mathcal{T},\mathcal{T}_\mathrm{mark})$ \Comment{Recoarsen green and blue patterns to receive $\mathcal{T}^R$.}
   \State $\mathcal{T}^R \gets \textsc{CoarsenR}(\mathcal{T}^R,\mathcal{T}_\mathrm{mark}^R)$ \Comment{Operate coarsening on $\mathcal{T}^R$ according to $\mathcal{T}_\mathrm{mark}^R$ (Alg.~\ref{alg:algo1}).}
   \State $\hat{\mathcal{T}} \gets \reg(\mathcal{T}^R)$ \Comment{Regularize $\mathcal{T}^R$ by adding green and blue patterns.}
         \State \textbf{return} $\hat{\mathcal{T}}$
\EndProcedure
\end{algorithmic}
\end{algorithm}

The rest of this paper is structured as follows. First, we introduce some notation in Section ~\ref{sect:prelim}. In Section~\ref{sect:refine}, we also present the refinement strategies considered in this paper and outline their properties. We describe the data structures we use to define a triangulation and explain how the refined elements are stored within these data structures. Then, in Section~\ref{sect:ideas}, we discuss our coarsening algorithms, namely Algorithm~\ref{alg:algo1} and Algorithm~\ref{alg:algo2}, focusing on the individual steps $\adm, \mar, \clos$ and $\upd$ and their differences depending on the refinement method. We also show that the presented algorithms achieve the desired results. In Section~\ref{sect:implement} we convert these ideas into code. Thereby we focus on the efficient implementation in \textsc{Matlab}. Most of this efficiency is due to the easy-to-verify criteria - but in addition we use vectorization and built-in functions to support us in efficiency. Furthermore, an overview of the \texttt{ameshcoars} toolbox is given. The work concludes with some numerical experiments in Section~\ref{sect:experiments}. To gain an insight into the abstract level, we recommend reading sections~\ref{sect:prelim} to \ref{sect:ideas}. To use the tool it is also useful to read Section~\ref{sect:overview} as it gives an overview of the included functions and their use in context.
\section{Notation and Preliminaries}\label{sect:prelim}

We consider a polygonal domain $\Omega$ in $\mathbb{R}^2$. An element $T \subset \mathbb{R}^2$ has quadrangular or triangular shape and by convention the edges are included in $T$. A \emph{triangulation} $\mathcal{T}$ is a finite set of elements $T$ with positive area $|T|>0$, the union of all elements in $\mathcal{T}$ covers the closure $\overline{\Omega}$ and for two elements $T,K \in \mathcal{T}$ with $T \neq K$ holds $\mathring{T}\cap\mathring{K}=\emptyset$ where $\mathring{T}$ is the element $T$ without edges. In the following we use the terms \emph{mesh} and \emph{grid} synonymously to triangulation.

We denote the set of edges of a triangulation $\mathcal{T}$ with $\mathcal{E}$ and the set of nodes/vertices with $\mathcal{N}$. Accordingly, $\mathcal{E}(T) \coloneqq \left\{e \in \mathcal{E}~|~e \text{ is edge of } T\right\}$ is the set of edges of an element $T \in \mathcal{T}$ and the set of nodes $\mathcal{N}(T) \coloneqq \left\{v \in \mathcal{N}~|~v \text{ is vertex of } T\right\}$, respectively. 

We call $v \in \mathcal{N}$ a \emph{hanging node} if for some element $T \in \mathcal{T}$ holds $v \in \partial T\setminus \mathcal{N}(T)$. A \emph{$1$-irregular triangulation} is a triangulation with at most one hanging node per edge. A \emph{conforming or regular triangulation} $\mathcal{T}$ of $\Omega$ is a triangulation without hanging nodes, i.e., if additionally for all $T,K \in \mathcal{T}$ with $T\neq K$ holds that $T\cap K$ is the empty set, a common node or a common edge. 

A family of triangulations $\left\{\mathcal{T}_h\right\}$ into triangles with $h = \max\limits_{T \in \mathcal{T}} h_T >0$ is called \emph{shape regular} if there exists a constant $c\geq 1$ such that 
\[ \frac{h_T}{\rho_T} \leq c \quad \text{ for all } T \in \mathcal{T}_h \text{ and all } h>0, \]
where $h_T$ denotes the diameter of $T \in \mathcal{T}_h$ and $\rho_T$ is the diameter of the largest ball inscribed in $T$ \cite{ciarlet}.
This condition is equivalent to the condition, that the minimal angle of all triangles $T \in \mathcal{T}_h$ is bounded away from zero.

A family of triangulations $\left\{\mathcal{T}_h\right\}$ into quadrilaterals with $h = \max\limits_{T \in \mathcal{T}} h_T >0$ is called \emph{shape regular} if \begin{itemize}
\item there exists a constant $c_1\geq 1$ such that 
\[ \frac{\overline{h}_T}{\underline{h}_T} \leq c_1 \quad \text{ for all } T \in \mathcal{T}_h \text{ and all } h>0, \]
where $\overline{h}_T$ and $\underline{h}_T$ denote the longest and shortest edge of the quadrilateral $T$ respectively and 
\item there exists a constant $c_2>0$ such that
\[ |\cos\varphi \,| \leq c_2 < 1 \quad \text{ for all inner angles } \varphi \text{ of } T\] 
\end{itemize} 
holds \cite{ciarletraviart}. The latter property ensures that a quadrilateral does not degenerate into a triangle. Other weaker or equivalent formulations can be found in \cite{brandts,brandts2008equivalence} for triangular meshes and in \cite{acosta,chou,ming,ming2} for quadrilateral meshes. The shape regularity of a family of triangulations is a key property in the analysis of finite element methods \cite{babuvska,braess,verfuerth}.

Let $\{ \mathcal{T}_i\}_{i=1,\ldots,n}$ be a sequence of triangulations where $\mathcal{T}_{i+1}$ results from refining $\mathcal{T}_i$ with a certain strategy. This sequence is called a \emph{nested} triangulation if $\mathcal{T}_1 \subset \mathcal{T}_2 \subset \mathcal{T}_3 \subset \cdots \subset \mathcal{T}_n$ holds for all $n \in \mathbb{N}$.

\section{Mesh Refinement Strategies}\label{sect:refine}

In this work, we are concerned with coarsening strategies of triangulations that were generated through different refinement strategies. Our goal is to find adaptive coarsening routines for these refinement strategies without explicitly knowing the refinement history. To this end, we make use of implicit information in the data structure and are thus able to find a way back in the refinement history. In the course of adaptive refinement, we ensure certain properties that we also want to maintain during coarsening. We will look at the properties in more detail when discussing the refinement process. 

All of the ideas used within this work rely on the refinement strategies as well as on the data structures used in the implementation of these refinement strategies. For this reason, we address this topic in detail in this section.

Let us first begin with the different refinement strategies. In Figure~\ref{fig:mesh}, an overview of the implemented refinement strategies in the \texttt{ameshref}-package is given that are considered for coarsening in this work. We can cluster these methods into three different categories: nested, regular refinements; nested, irregular refinements; and non-nested, regular refinements. This categorization is helpful for coarsening as similar approaches can be used within each category. Let us present each of this category in more detail.

\subsection{Nested and Regular Refinement}

In this category, we have the newest vertex bisection {\TNVB} \cite{sewell} and the red-green-blue {\TRGB} refinement as presented in \cite{carstensen2004}. The refinement patterns are shown in Figure~\ref{fig:NVBRGB}. As the name newest vertex bisection suggests, the angle at the newest vertex is bisected. In our presentation, we do not highlight the newest vertex but the edge lying opposite to the newest vertex. In the further course of this work, we call this edge a \emph{reference edge}.

\begin{figure}
\begin{center}
\begin{minipage}{0.19\linewidth}
\hspace*{6mm}\begin{tikzpicture}
[ interface/.style={
        postaction={draw,decorate,decoration={border,angle=45,
                    amplitude=0.13cm,segment length=1mm}}},
    ]
\coordinate (1) at (-1.2,0);
\coordinate (2) at (1.2,0);
\coordinate (3) at (0,1.5);
\draw[fill=pastelred] (1)--(2) -- (3)--(1);
\draw [interface, thick](1) --(2);
\fill (1) circle (2pt); 
\fill (2) circle (2pt); 
\fill (3) circle (2pt); 
\node at (0,-0.5) {none};
\end{tikzpicture}
\end{minipage}
\begin{minipage}{0.19\linewidth}
\hspace*{6mm}\begin{tikzpicture}[ interface/.style={
        postaction={draw,decorate,decoration={border,angle=45,
                    amplitude=0.13cm,segment length=1mm}}},
    ]
\coordinate (1) at (-1.2,0);
\coordinate (2) at (1.2,0);
\coordinate (3) at (0,1.5);
\draw[fill=lightgreen] (1) --(2)--(3)--(1);
\draw[interface,thick] (2) -- (3);
\draw[interface,thick] (3) -- (1);
\fill (1) circle (2pt); 
\fill (2) circle (2pt); 
\fill (3) circle (2pt); 
\fill ($(1)!0.5!(2)$) circle (2pt);
\draw ($(1)!0.5!(2)$) --(3);
\node at (0,-0.5) {green};
\end{tikzpicture}
\end{minipage}
\begin{minipage}{0.19\linewidth}
\hspace*{6mm}\begin{tikzpicture}[ interface/.style={
        postaction={draw,decorate,decoration={border,angle=45,
                    amplitude=0.13cm,segment length=1mm}}},
    ]
\coordinate (1) at (-1.2,0);
\coordinate (2) at (1.2,0);
\coordinate (3) at (0,1.5);
\draw[fill=lightblue] (1)--(2) -- (3)--(1);
\draw[interface,thick] (2) -- (3);
\fill (1) circle (2pt); 
\fill (2) circle (2pt); 
\fill (3) circle (2pt); 
\fill ($(1)!0.5!(2)$) circle (2pt);
\fill ($(1)!0.5!(3)$) circle (2pt);
\draw ($(1)!0.5!(3)$) -- ($(1)!0.5!(2)$);
\draw[interface,thick] ($(1)!0.5!(2)$)--(3);
\draw[interface,thick] (1)-- ($(1)!0.5!(2)$);
\node at (0,-0.5) {blue$_\ell$};
\end{tikzpicture}
\end{minipage}
\begin{minipage}{0.19\linewidth}
\hspace*{6mm}\begin{tikzpicture}[ interface/.style={
        postaction={draw,decorate,decoration={border,angle=45,
                    amplitude=0.13cm,segment length=1mm}}},
    ]
\coordinate (1) at (-1.2,0);
\coordinate (2) at (1.2,0);
\coordinate (3) at (0,1.5);
\draw[fill=lightblue] (1)--(2) -- (3)--(1);
\draw[interface,thick] (3) -- (1);
\fill (1) circle (2pt); 
\fill (2) circle (2pt); 
\fill (3) circle (2pt); 
\fill ($(1)!0.5!(2)$) circle (2pt);
\fill ($(2)!0.5!(3)$) circle (2pt);
\draw ($(2)!0.5!(3)$) -- ($(1)!0.5!(2)$);
\draw[interface,thick] (3)--($(1)!0.5!(2)$);
\draw[interface,thick] ($(1)!0.5!(2)$) --(2);
\node at (0,-0.5) {blue$_r$};
\end{tikzpicture}
\end{minipage}
\begin{minipage}{0.19\linewidth}
\begin{minipage}{\textwidth}
\hspace*{6mm}\begin{tikzpicture}[ interface/.style={
        postaction={draw,decorate,decoration={border,angle=45,
                    amplitude=0.13cm,segment length=1mm}}},
    ]
\coordinate (1) at (-1.2,0);
\coordinate (2) at (1.2,0);
\coordinate (3) at (0,1.5);
\draw[fill=pastelred] (1)--(2) -- (3)--(1);
\draw[interface,thick] (1) --(2);
\fill (1) circle (2pt); 
\fill (2) circle (2pt); 
\fill (3) circle (2pt); 
\fill ($(1)!0.5!(2)$) circle (2pt);
\fill ($(2)!0.5!(3)$) circle (2pt);
\fill ($(1)!0.5!(3)$) circle (2pt);
\draw ($(1)!0.5!(2)$) -- ($(2)!0.5!(3)$);
\draw[interface,thick] ($(1)!0.5!(2)$) -- (3);
\draw[interface,thick] (3)--($(1)!0.5!(2)$);
\draw ($(1)!0.5!(2)$) -- ($(1)!0.5!(3)$);
\node at (0,-0.5) {bisec$(3)$};
\end{tikzpicture}
\end{minipage}\\
\begin{minipage}{\textwidth}
\hspace*{6mm}\begin{tikzpicture}[ interface/.style={
        postaction={draw,decorate,decoration={border,angle=45,
                    amplitude=0.13cm,segment length=1mm}}},
    ]
\coordinate (1) at (-1.2,0);
\coordinate (2) at (1.2,0);
\coordinate (3) at (0,1.5);
\draw[fill=pastelred] (1)--(2) -- (3)--(1);
\draw[interface,thick] (1) --(2);
\fill (1) circle (2pt); 
\fill (2) circle (2pt); 
\fill (3) circle (2pt); 
\fill ($(1)!0.5!(2)$) circle (2pt);
\fill ($(2)!0.5!(3)$) circle (2pt);
\fill ($(1)!0.5!(3)$) circle (2pt);
\draw ($(1)!0.5!(2)$) -- ($(2)!0.5!(3)$)--($(1)!0.5!(3)$)--($(1)!0.5!(2)$);
\draw[interface,thick] ($(2)!0.5!(3)$)--($(1)!0.5!(3)$);
\draw[interface,thick] ($(1)!0.5!(3)$)--($(2)!0.5!(3)$);
\node at (0,-0.5) {red};
\end{tikzpicture}
\end{minipage}
\end{minipage}
\end{center}
\caption{Refinement patterns for newest vertex bisection and red-green-blue refinement. Both methods differ only by the last pattern. {\TNVB} uses a bisec$(3)$-operation where {\TRGB} uses a red refinement. The other patterns are used in both methods. Reference edges are highlighted by hatched lines. The vertex opposite to the reference edge is considered the newest vertex. It becomes clear that {\TNVB} always splits the angle at the newest vertex. {\TRGB} is an adaption to {\TNVB} with the difference that if all three edges are bisected, a red pattern is used.}
\label{fig:NVBRGB}
\end{figure}
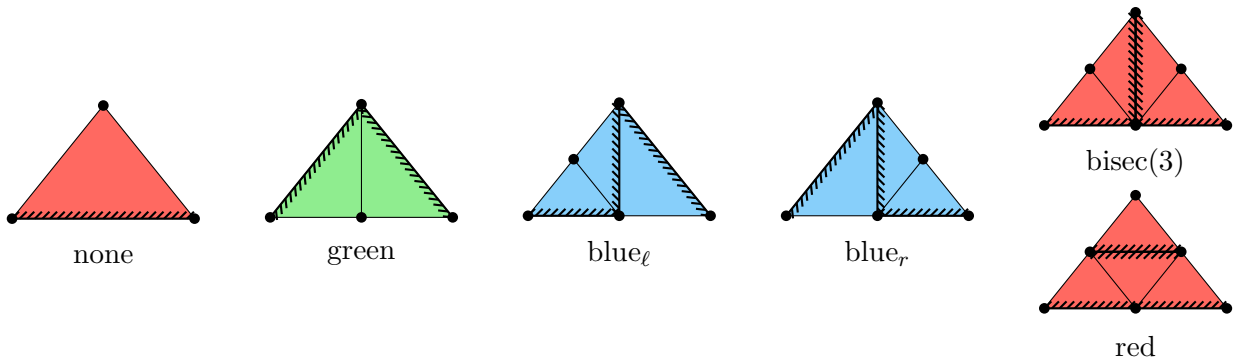

Halving the angle at the newest vertex corresponds to the bisection of the reference edge. This means that if an edge of an element is marked at least the reference edge of this element needs to be marked, too. {\TRGB} differs from {\TNVB} only in one pattern -- when all edges are bisected. Instead of using a bisec(3)-operation a red refinement pattern is used. The possible refinement patterns for {\TNVB} and {\TRGB} are shown in Figure~\ref{fig:NVBRGB}. The reference edges ensure that in a refinement process the successors of a triangle fall into at most four similarity classes, the minimal angle is thus bounded by a constant not depending on the mesh size. Thus, the shape regularity is ensured for all families $\{\mathcal{T}_h\}$ obtained by adaptive refinement via {\TNVB} or {\TRGB}.

To the best of our knowledge, no triangulations into quadrilaterals are known that fall into this category.

\subsection{Nested and Irregular Refinement}

In this category, we have the red (\textsf{R}) refinement strategy for triangular and quadrilateral meshes. 
\begin{definition}[Red Refinement]
A triangle $T$ is split into four subtriangles by joining the edges' midpoints with each other.
A quadrilateral $T$ is subdivided into four subquadrilaterals by connecting the midpoints of opposite edges with each other. We call this operation for triangles or quadrilaterals, respectively, a \emph{red refinement}, see Figure~\ref{fig:red}.
\end{definition}

\begin{figure}
\centering
\hspace*{4mm}
\begin{minipage}{0.24\textwidth}
\begin{tikzpicture}
\coordinate (1) at (-1.5,0);
\coordinate (2) at (1.5,0);
\coordinate (3) at (0,2);
\node at (0,-0.5) {none};
\draw[fill=pastelred] (1)--(2)--(3)--(1);
\fill (1) circle (2pt);
\fill (2) circle (2pt);
\fill (3) circle (2pt);
\end{tikzpicture}
\end{minipage}
\begin{minipage}{0.24\textwidth}
\begin{tikzpicture}
\coordinate (1) at (-1.5,0);
\coordinate (2) at (1.5,0);
\coordinate (3) at (0,2);
\coordinate (4) at (0,0);
\coordinate (5) at ($(2)!0.5!(3)$);
\coordinate (6) at ($(3)!0.5!(1)$);
\node at (0,-0.5) {red};
\draw[fill=pastelred] (1)--(2)--(3)--(1);
\draw (4) -- (5)-- (6)--(4);
\fill (1) circle (2pt);
\fill (2) circle (2pt);
\fill (3) circle (2pt);
\fill (4) circle (2pt);
\fill (5) circle (2pt);
\fill (6) circle (2pt);
\end{tikzpicture}
\end{minipage}\hspace*{4mm}
\begin{minipage}{0.22\textwidth}
\begin{tikzpicture}
\coordinate (1) at (-1,-1);
\coordinate (2) at (1,-1);
\coordinate (3) at (1,1);
\coordinate (4) at (-1,1);
\draw[fill=pastelred] (1)--(2)--(3)--(4)--(1);
\fill (1) circle (2pt);
\fill (2) circle (2pt);
\fill (3) circle (2pt);
\fill (4) circle (2pt);
\node at (0,-1.5) {none};
\end{tikzpicture}
\end{minipage}
\begin{minipage}{0.22\textwidth}
\begin{tikzpicture}
\coordinate (1) at (-1,-1);
\coordinate (2) at (1,-1);
\coordinate (3) at (1,1);
\coordinate (4) at (-1,1);
\coordinate (5) at (0,-1);
\coordinate (6) at (1,0);
\coordinate (7) at (0,1);
\coordinate (8) at (-1,0);
\draw[fill=pastelred] (1)--(2)--(3)--(4)--(1);
\draw (5)--(7);
\draw (6)--(8);
\fill (1) circle (2pt);
\fill (2) circle (2pt);
\fill (3) circle (2pt);
\fill (4) circle (2pt);
\fill (5) circle (2pt);
\fill (6) circle (2pt);
\fill (7) circle (2pt);
\fill (8) circle (2pt);
\fill ($(5)!0.5!(7)$) circle (2pt);
\node at (0,-1.5) {red};
\end{tikzpicture}
\end{minipage}
\caption{A triangle and quadrilateral with its red refinement.}
\label{fig:red}
\end{figure}
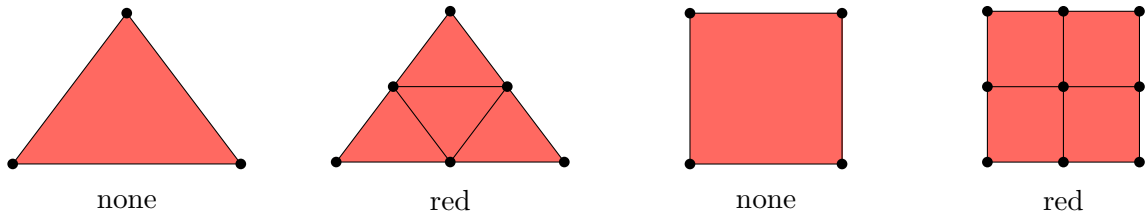 

If a red refinement is used adaptively, hanging nodes arise naturally. If hanging nodes are to be avoided, but only red refinements are allowed, the result is automatically an evenly refined mesh. For adaptive procedures, hanging nodes can thus not be avoided by red-refining only. The red refinement is a nested refinement which is irregular. There are several reasons why the number of hanging nodes per edge should be limited to one.

One reason lies in the implementation of adaptive mesh refinement. Simpler data structures can be used to track hanging nodes as there can either be one hanging node per edge or none. Further, in the context of adaptive finite element methods, the associated finite element space is spanned by basis functions whose support does not exceed a small number of elements. This has favorable consequences in the assembly of matrices as it keeps the computational cost low and ensures that the matrix is sparse. 

In summary, this rule is a sort of "regularization" of the mesh, without eliminating all hanging nodes. A more thorough analysis of the $1$-irregularity can be found in \cite{bankshermanweiser}. For this reason, we follow the $1$-Irregular Rule \emph{"Refine any unrefined element that has more than one hanging node on an edge."} established in \cite{bankshermanweiser}.

The shape regularity for families of triangulations into triangles obtained via the red refinement is readily apparent. A red refinement of a triangle creates four similar subtriangles. The angle for successors does thus not change during refinement. For quadrilaterals, this is not directly apparent. Zhao et al.\ proved the shape regularity for meshes obtained by enneasection instead of red refinement \cite{maozhaoshi}. I.e., a quadrilateral is not split into four subquadrilaterals but into nine by trisecting each edge and connecting opposite lying nodes with each other. The ideas of the proofs for the shape regularity of convex quadrilateral meshes obtained by enneasection carry over for the red refinement. For a family of convex quadrilateral triangulations $\{\mathcal{T}_h\}$ the shape regularity is thus satisfied.

Further, for quadrilaterals, we would like to follow the $3$-Neighbor Rule from \cite{bankshermanweiser} that says \emph{"Refine any element with three neighbors that have been red refined."} For triangles, a $2$-Neighbor Rule can be used accordingly. An implementation with and without the $2$-Neighbor Rule are given in the \texttt{ameshref}-package. The standard version (\texttt{TrefineR.m}) is without this rule.

\subsection{Non-Nested and Regular Refinement}

In this category, we have the red-green (\textsf{RG}) refinement for triangular and quadrilateral elements and the red-blue (\textsf{RB}) refinement for quadrilateral elements. In the previous section, we already discussed red refinement. Now, we want to regularize the triangulation by allowing further refinement to eliminate the hanging nodes. 
\begin{definition}[Green Refinement]
A triangle $T$ is \emph{green}-refined if it is divided into two subtriangles by bisecting one edge of the triangle and connecting this midpoint to the vertex opposite to this edge. A quadrilateral $T$ is \emph{green}-refined, if it is subdivided into
\begin{itemize}
\item three triangles by connecting the midpoint of one edge with the vertices of the opposite edge;
\item four triangles by connecting the midpoints of two adjacent edges and the common vertex of the other two edges with each other; or
\item two quadrilaterals by connecting the midpoints of two opposite edges.
\end{itemize}
\end{definition}

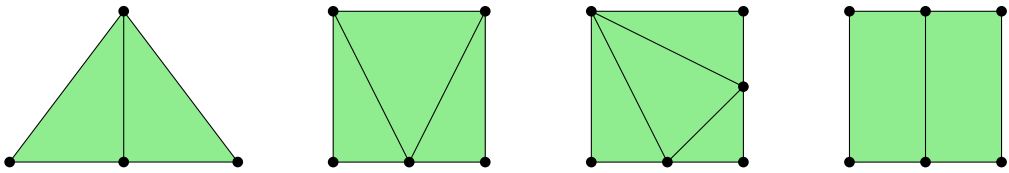
\begin{figure}
\centering
\begin{minipage}{0.24\textwidth}
\begin{tikzpicture}
\coordinate (1) at (-1.5,0);
\coordinate (2) at (1.5,0);
\coordinate (3) at (0,2);
\coordinate (4) at (0,0);
\draw[fill=lightgreen] (1)--(2)--(3)--(1);
\draw (4) -- (3);
\fill (1) circle (2pt);
\fill (2) circle (2pt);
\fill (3) circle (2pt);
\fill (4) circle (2pt);
\end{tikzpicture}
\end{minipage}
\begin{minipage}{0.19\textwidth}
\begin{tikzpicture}
\coordinate (1) at (-1,-1);
\coordinate (2) at (1,-1);
\coordinate (3) at (1,1);
\coordinate (4) at (-1,1);
\draw[fill=lightgreen] (1)--(2)--(3)--(4)--(1);
\draw (4)--($(1)!0.5!(2)$)--(3);
\fill (1) circle (2pt);
\fill (2) circle (2pt);
\fill (3) circle (2pt);
\fill (4) circle (2pt);
\fill ($(1)!0.5!(2)$) circle (2pt);
\end{tikzpicture}
\end{minipage}
\begin{minipage}{0.19\textwidth}
\begin{tikzpicture}
\coordinate (1) at (-1,-1);
\coordinate (2) at (1,-1);
\coordinate (3) at (1,1);
\coordinate (4) at (-1,1);
\draw[fill=lightgreen] (1)--(2)--(3)--(4)--(1);
\draw (4)--($(1)!0.5!(2)$)--($(2)!0.5!(3)$)--(4);
\fill (1) circle (2pt);
\fill (2) circle (2pt);
\fill (3) circle (2pt);
\fill (4) circle (2pt);
\fill ($(1)!0.5!(2)$) circle (2pt);
\fill ($(2)!0.5!(3)$) circle (2pt);
\end{tikzpicture}
\end{minipage}
\begin{minipage}{0.19\textwidth}
\begin{tikzpicture}
\coordinate (1) at (-1,-1);
\coordinate (2) at (1,-1);
\coordinate (3) at (1,1);
\coordinate (4) at (-1,1);
\draw[fill=lightgreen] (1)--(2)--(3)--(4)--(1);
\draw ($(1)!0.5!(2)$)--($(3)!0.5!(4)$);
\fill (1) circle (2pt);
\fill (2) circle (2pt);
\fill (3) circle (2pt);
\fill (4) circle (2pt);
\fill ($(1)!0.5!(2)$) circle (2pt);
\fill ($(3)!0.5!(4)$) circle (2pt);
\end{tikzpicture}
\end{minipage}
\caption{Green refinements of a triangle and quadrilateral.}
\label{fig:green}
\end{figure}

Figure~\ref{fig:green} illustrates the different green patterns for triangles and quadrilaterals. In the red-green refinement, the green patterns are now applied like a jigsaw puzzle to eliminate all the hanging nodes that are created by red refinement, cf.~Figure~\ref{fig:regularize_green}, in compliance with the $1$-Irregular Rule and $d$-Neighbor Rule with $d=2$ for triangles and $d=3$ for quadrilaterals. The red-green refinement procedure thus reads: \begin{enumerate} \item recoarsen green patterns, \item red-refine irregular mesh in accordance with the $1$-Irregular and $d$-Neighbor Rule, and \item regularize mesh with green patterns. \end{enumerate}
This procedure can be repeated until the triangulation fulfils the desired properties.
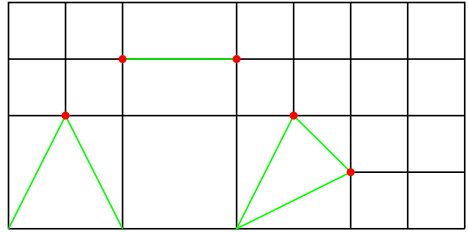
\begin{figure}
\centering
\scalebox{1.5}{
\begin{tikzpicture}[rotate=180]
\draw (0,2)--(0,0)--(4,0)--(4,2);
\draw (0,1.5)--(1,1.5);
\draw (0,2)--(4,2);
\draw (0,1)--(4,1);
\draw (0,0.5)--(4,0.5);
\draw (1,2)--(1,0);
\draw (2,2)--(2,0);
\draw (3,2)--(3,0);
\draw (0.5,2)--(0.5,0);
\draw (1.5,1)--(1.5,0);
\draw (3.5,1)--(3.5,0);
\draw[green] (1,1.5)--(1.5,1)--(2,2)--(1,1.5);
\draw[green] (2,0.5)--(3,0.5);
\draw [green] (3,2)--(3.5,1)--(4,2);
\fill[red] (1,1.5) circle (1pt);
\fill[red] (1.5,1) circle (1pt);
\fill[red] (2,0.5) circle (1pt);
\fill[red] (3,0.5) circle (1pt);
\fill[red] (3.5,1) circle (1pt);
\end{tikzpicture}}
\caption{Regularization of an irregular mesh of quadrilaterals by creating green patterns. The red nodes represent the hanging nodes of the irregular mesh. The green patterns are arranged accordingly to obtain a regular triangulation.} 
\label{fig:regularize_green}
\end{figure}

These refinement strategies do not create a nested sequence of triangulations during the refinement process. The shape regularity for a family of triangulations $\{\mathcal{T}_h\}$ obtained by the red-green refinement is ensured. Red-refinements preserve the shape regularity and, as green patterns are removed before further refinement, the angles can not become too small. The red-green refinement of quadrilaterals leads to a mixed triangulation with triangles and quadrilaterals. The red-blue refinement strategy generates an adaptive mesh with quadrilaterals only. In contrast, new nodes need to be added by a blue pattern; green pattern can be used without additional nodes. 

\begin{definition}[Blue Refinement]
A quadrilateral $T$ is \emph{blue}-refined if it is split into three quadrilaterals by joining the midpoint of $T$ with the two midpoints of adjacent edges and with the node that is opposite to the node that is shared by those adjacent edges, cf.~Figure~\ref{fig:blue}.
\end{definition}

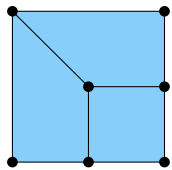
\begin{figure}
\begin{center}
\begin{minipage}[t]{0.19\textwidth}
\begin{tikzpicture}
\coordinate (1) at (-1,-1);
\coordinate (2) at (1,-1);
\coordinate (3) at (1,1);
\coordinate (4) at (-1,1);
\coordinate (5) at (0,-1);
\coordinate (6) at (1,0);
\coordinate (7) at (0,1);
\coordinate (8) at (-1,0);
\coordinate (9) at (0,0);
\draw[fill=lightblue] (1)--(2)--(3)--(4)--(1);
\draw (5)--(9);
\draw (6)--(9);
\draw (9)--(4);
\fill (1) circle (2pt);
\fill (2) circle (2pt);
\fill (3) circle (2pt);
\fill (4) circle (2pt);
\fill (5) circle (2pt);
\fill (6) circle (2pt);
\fill (9) circle (2pt);
\end{tikzpicture}
\end{minipage}
\end{center}
\caption{Blue refinement of a quadrilateral.} 
\label{fig:blue}
\end{figure}

The red-blue refinement is similar to the red-green refinement. However, as we only have one blue refinement pattern, we cannot fit this blue pattern to eliminate all hanging nodes. By closing one hanging node another might be introduced. Or, for two opposite lying hanging nodes, the only fitting pattern is the red one. Thus, an additional \textsc{Closure} step might be needed to close the mesh properly. Further remarks are to be made for red-blue refinement, to ensure that the mesh refinement strategy defines a unique way of refining the elements. A thorough discussion of this topic is presented in \cite{ameshref,funkenschmidt} and is omitted here, as for the coarsening step this is not essential.
The family of triangulations $\{\mathcal{T}_h\}$ does satisfy shape regularity. Blue patterns are not refined further and thus the shape regularity of the red-refinements is essential.

\subsection{Data Structure and Storage of Refined Elements}\label{sect:datastructure}

For our coarsening routines, we present the key ideas used to go back in the refinement history and coarsen adaptively marked elements. A refinement creates so-called \emph{child} elements. In coarsening, we want to coarsen these child elements back to their \emph{parent} elements. We will use this terminology throughout this work. The ideas used for {\TNVB} and {\TRGB} are not as obvious and some effort has to be made to determine an easy-to-verify criterion for admissible nodes. An algorithmic perspective on coarsening for {\TNVB} is made in \cite{chenzhang} and this idea is implemented in accordance with our refinement procedure in \cite{p1afem}. The same has been done in \cite{RGB} for {\TRGB}. Both implementations are included in the \texttt{ameshcoars}-package but their discussion is only marginally addressed in this work, as it is presented thoroughly in the mentioned works.
For the other listed refinement strategies the ideas to coarsen elements is quite clear. The core of the work here lies in the efficient implementation with the lack of an explicit refinement history. For this purpose, we concentrate on the data structures of these triangulations and the storage of refined elements within this data structure.

\subsubsection{Data Structure}

A triangulation $\mathcal{T}$ consists of a finite number of elements $T$, i.e.,we can denote a triangulation as a set of elements $\mathcal{T}=\left\{T_1,\ldots,T_m\right\}$. The elements $T_\ell$ for $\ell=1,\ldots m$ are further described for quadrilateral elements by $ T_\ell= \conv(v_i,v_j,v_k,v_l)$ with vertices $v_i,v_j,v_k,v_l$ that are located within a defined coordinate system. In case of triangular elements, we have $ T_\ell= \conv(v_i,v_j,v_k)$ with vertices $v_i,v_j,v_k$, respectively. The set of all vertices also called the set of nodes $\mathcal{N}=\left\{v_1,\ldots,v_n\right\}$, is stored in an $n \times 2$ array \matlab{coordinates} where the $\ell$-th node $v_\ell = (x_\ell,y_\ell) \in \mathbb{R}^2$ with its x- and y-coordinates is represented by
\[ \text{\matlab{coordinates($\ell$,:)= $[x_\ell \quad y_\ell]$}}. \]
The triangulation $\mathcal{T}=\left\{T_1,\ldots,T_m\right\}$ is stored in an $m \times 4$ array for quadrilaterals and in an $m  \times 3$ array for triangles. If there is no ambiguity, we denote this array by \matlab{elements}. If triangular und quadrilateral elements are present at the same time, we make the distincton \matlab{elements3} for triangular and \matlab{elements4} for quadrilateral elements. The $\ell$-th element $ T_\ell= \conv(v_i,v_j,v_k,v_l)$ is then stored in
\[ \matlab{elements4($\ell$,:) = $[i\quad j\quad k \quad l]$} \] 
with ordering of the vertices in a mathematical positive sense and for $ T_\ell= \conv(v_i,v_j,v_k)$ in
\[ \matlab{elements3($\ell$,:) $= [i\quad j\quad k]$}, \]
respectively. Additional boundary information can be added, where the $\ell$-th edge $E_\ell=\conv(v_i,v_j)$ corresponds to
\[\matlab{boundary($\ell$,:)$ = [i\quad j]$}.\]
Irregular edges are edges on which a hanging node lies. These edges are stored in a separate array \matlab{irregular} for which holds 
\[ \matlab{irregular($\ell$,:)$= [i \quad j \quad k]$}\]
is the $\ell$-th edge $E_\ell = \conv(v_i, v_j)$ with hanging node $v_k$. Another way to interpret an irregular edge is to see it as a \emph{virtual element}. This means that the three nodes form a triangle $V=\conv(v_i,v_j,v_k)$ with area $|V|=0$ and are therefore called virtual. We call an element $K \in \mathcal{T}$ a \emph{direct neighbor} of $T \in \mathcal{T}$ if $T$ and $K$ have a common edge. We call an element $K \in \mathcal{T}$ an \emph{indirect neighbor} of $T \in \mathcal{T}$ if they are connected by a virtual element, i.e. $ T\cap K$ is either an edge of $T$ or an edge of $K$, but not both at the same time.

The data structures presented so far are minimal in the sense that they describe the mesh fully but without any further calculations they do not provide information about the neighborhood. However, this information will be of great importance in the following. Therefore, auxiliary functions are used to generate neighborhood information that can be easily extracted later. The creation of the additional data structures is an effort that is performed once each time, but which ensures an efficient extraction of information during the implementation.

With these auxiliary functions we number the edges, save in \matlab{egde2nodes} which nodes belong to an edge and then define the elements in a mathematically positive sense by these edge numbers in \matlab{element3edges} for triangles and in \matlab{element4edges} for quadrilaterals. This information is sufficient for refinement. For coarsening we need additional information, namely which elements are adjacent to an edge, stored in \matlab{edge2elements}.

The creation of these additional data structures ensures an overall efficient implementation of the coarsening routines.

\subsubsection{Storage of Refined Elements}

Besides the data structures, another important feature is the storage of refined elements. We do not store the refinement history of an adaptively generated mesh. However, in order to coarsen meshes, some information must be implicit in the way refined elements are stored. Therefore we will go into this in more detail.

Depending on the refinement strategy, the way refined elements are stored plays a role. What they all have in common is that new coordinates are appended to the end of \matlab{coordinates}. The smaller the index, the older a node is. New elements are not stored at the end of the array \matlab{elements}. Red refinements are inserted at the position where the parent element was defined -- all other elements are shifted by the number of additionally inserted elements. We emphasize that this type of storage for red, red-green or red-blue strategies is not exploited for our coarsening algorithms. What we do use is that for quadrilaterals, the first node in an element is always the oldest, i.e., it has the minimum index. For triangles, this is not taken into account during storage, but for coarsening, the same sorting is used. Red-green and red-blue triangulations are stored in blocks, e.g., all red elements are stored in one block, and all green respectively blue elements are stored in a subsequent block. To distinguish a red block from a blue or green one, the number of blue or green elements is tracked in a global variable \matlab{nB} or \matlab{nG}. As the elements' order in green and blue blocks is not changed during refinement or coarsening, elements with the same parent element are always stored consecutively, making coarsening easy. Thus, only suitable criteria need to be found to coarsen red patterns. In summary, we only use the following storage properties for coarsening:
\begin{itemize}
\item New coordinates are appended at the end, so the smaller the index, the older the node.
\item Elements are defined in a mathematically positive sense; for quadrilaterals we want to have the nodes with minimum index at the beginning (otherwise an additional sorting is required as for for triangles).
\item Blue and green elements are stored in a block after all red elements.
\end{itemize}

For {\TNVB} or {\TRGB} the situation is different. There, it is of great importance not to lose the successive way of storing elements that have the same parent element over different refinement levels. For more information, see \cite{RGB}.

%
%
%
%
%
%
\section{Mesh Coarsening Strategies}\label{sect:ideas}

In this section, we present the main ideas used to identify the set of admissible-to-coarsen elements. As presented in Section~\ref{sect:refine}, the refining process of red-green refinements is implemented in a three-step-procedure, i.e., we remove green patterns, refine the $1$-irregular grid, and close the grid again with green patterns to ensure the regularity of the grid. We want to use the same three-step-procedure for red-green and red-blue coarsening, replacing the refinement operation on the $1$-irregular grid by a coarsening operation, cf.~Algorithm~\ref{alg:algo2}.

We examine coarsening of $1$-irregular grids first. Un-/Closure of the mesh with green and blue patterns is a straight-forward task and is shortly presented afterwards.

\subsection{Red Refinements}

In adaptive coarsening, different aspects come into play. The task is to find an admissible set of elements that can be coarsened without losing desired properties of the mesh. 

As a basis for the coarsening of elements, we determine quartets of elements that all have the same parent element. This ensures that only elements that belonged together in a previous refinement step are merged. In particular, this guarantees that the shape regularity of a mesh is maintained when it is coarsened. These quartets are collected in the set of $\adm$ elements for coarsening.

To make the procedure adaptive, there is an option to $\mar$ certain elements for coarsening. Not all marked elements can be coarsened, but only those that appear in a quartet. This means that the marking function can reduce the number of admissible quartets determined in the previous step but it cannot enlarge it.

With the so updated admissible set of quartets, properties like the $1$-irregularity of a mesh or the $d$-Neighbor Rule are not taken into account. These are considered in a $\clos$ step. There, the number of admissible quartets is again reduced to maintain the $1$-irregularity of the grid and the $d$-Neighbor Rule. 

The remaining admissible quartets can then be coarsened back to their parent element while maintaining all desired mesh properties. This is done in the step $\upd$ where elements, nodes and irregularity data are updated for the admissible quartets and the rest remain as before.

In summary, the coarsening strategy is carried out using a four-step-procedure that reads
 \[ \adm - \mar - \clos - \upd, \]
cf.~ Algorithm~\ref{alg:algo1}. As this algorithm describes, a set of admissible elements $\mathcal{A}$ is determined, which is updated several times before the information is used to update the mesh.
In the following, we present each step in detail for $1$-irregular meshes. In the presentation, we keep in mind that the coarsening operation on the $1$-irregular grid is also used in a red-green and red-blue context, see~Algorithm~\ref{alg:algo2}. For this reason, characteristics have to be determined that apply equally for those refinement methods.
 
 \subsubsection{$\adm$}

A red refinement of an element creates four new subelements. A coarsening of such a refinement, on the contrary, is intended to merge these four subelements to regain their parent element. So the task of the step $\adm$ is to determine quartets of elements that can be combined to form a larger element. These quartets can not be selected arbitrarily but their determination follows a certain structure. This structure is not explicitly given in our data structure, but we illustrate the idea in Figure~\ref{fig:tree}. The refinement of a single element can be displayed in a tree structure. If this element is red-refined, then four new elements of a finer level are created. The refinement level corresponds to the level of the tree. If adaptive refinement is applied, it is possible that only some elements are refined and others are not. Accordingly, the mesh consists of elements of different levels of refinement. Elements of different levels can not be joined together. To this end, we need to group four elements on the same level that have the same ancestor in the tree - an admissible quartet of elements. The tree structure in Figure~\ref{fig:tree} is displayed for quadrilateral meshes but applies analogously for the triangular case.

\begin{figure}
\centering
\begin{minipage}{0.85\textwidth}
\scalebox{0.82}{
\tikzset{every picture/.style={line width=0.75pt}} 
   
\begin{tikzpicture}[x=0.75pt,y=0.75pt,yscale=-1,xscale=1]
  

\draw   (442,30) -- (542,30) -- (542,130) -- (442,130) -- cycle ;
\draw   (418,159) -- (468,159) -- (468,209) -- (418,209) -- cycle ;
\draw   (618,159) -- (668,159) -- (668,209) -- (618,209) -- cycle ;
\draw   (318,159) -- (368,159) -- (368,209) -- (318,209) -- cycle ;
\draw   (518,159) -- (568,159) -- (568,209) -- (518,209) -- cycle ;

\draw   (270,240) -- (295,240) -- (295,265) -- (270,265) -- cycle ;
\draw   (310,240) -- (335,240) -- (335,265) -- (310,265) -- cycle ;
\draw   (350,240) -- (375,240) -- (375,265) -- (350,265) -- cycle ;
\draw   (390,240) -- (415,240) -- (415,265) -- (390,265) -- cycle ;

\draw   (532,289) -- (544.5,289) -- (544.5,301.5) -- (532,301.5) -- cycle ;
\draw   (552,289) -- (564.5,289) -- (564.5,301.5) -- (552,301.5) -- cycle ;
\draw   (572,289) -- (584.5,289) -- (584.5,301.5) -- (572,301.5) -- cycle ;
\draw   (592,289) -- (604.5,289) -- (604.5,301.5) -- (592,301.5) -- cycle ;

\draw   (472,240) -- (497,240) -- (497,265) -- (472,265) -- cycle ;
\draw   (512,240) -- (537,240) -- (537,265) -- (512,265) -- cycle ;
\draw   (552,240) -- (577,240) -- (577,265) -- (552,265) -- cycle ;
\draw   (592,240) -- (617,240) -- (617,265) -- (592,265) -- cycle ;

\draw    (438,134.5) -- (375.88,156.82) ;
\draw [shift={(374,157.5)}, rotate = 340.23] [color={rgb, 255:red, 0; green, 0; blue, 0 }  ][line width=0.75]    (10.93,-3.29) .. controls (6.95,-1.4) and (3.31,-0.3) .. (0,0) .. controls (3.31,0.3) and (6.95,1.4) .. (10.93,3.29)   ;
\draw    (547,132) -- (610.14,156.77) ;
\draw [shift={(612,157.5)}, rotate = 201.42000000000002] [color={rgb, 255:red, 0; green, 0; blue, 0 }  ][line width=0.75]    (10.93,-3.29) .. controls (6.95,-1.4) and (3.31,-0.3) .. (0,0) .. controls (3.31,0.3) and (6.95,1.4) .. (10.93,3.29)   ;
\draw    (479,136.5) -- (464.29,153.97) ;
\draw [shift={(463,155.5)}, rotate = 310.1] [color={rgb, 255:red, 0; green, 0; blue, 0 }  ][line width=0.75]    (10.93,-3.29) .. controls (6.95,-1.4) and (3.31,-0.3) .. (0,0) .. controls (3.31,0.3) and (6.95,1.4) .. (10.93,3.29)   ;
\draw    (504,136.5) -- (521.51,152.17) ;
\draw [shift={(523,153.5)}, rotate = 221.82] [color={rgb, 255:red, 0; green, 0; blue, 0 }  ][line width=0.75]    (10.93,-3.29) .. controls (6.95,-1.4) and (3.31,-0.3) .. (0,0) .. controls (3.31,0.3) and (6.95,1.4) .. (10.93,3.29)   ;
\draw    (371,213.5) -- (391.55,233.12) ;
\draw [shift={(393,234.5)}, rotate = 223.67000000000002] [color={rgb, 255:red, 0; green, 0; blue, 0 }  ][line width=0.75]    (10.93,-3.29) .. controls (6.95,-1.4) and (3.31,-0.3) .. (0,0) .. controls (3.31,0.3) and (6.95,1.4) .. (10.93,3.29)   ;
\draw    (573,212.5) -- (593.55,232.12) ;
\draw [shift={(595,233.5)}, rotate = 223.67000000000002] [color={rgb, 255:red, 0; green, 0; blue, 0 }  ][line width=0.75]    (10.93,-3.29) .. controls (6.95,-1.4) and (3.31,-0.3) .. (0,0) .. controls (3.31,0.3) and (6.95,1.4) .. (10.93,3.29)   ;
\draw    (581.75,267.13) -- (592.39,283.09) ;
\draw [shift={(593.5,284.75)}, rotate = 236.31] [color={rgb, 255:red, 0; green, 0; blue, 0 }  ][line width=0.75]    (10.93,-3.29) .. controls (6.95,-1.4) and (3.31,-0.3) .. (0,0) .. controls (3.31,0.3) and (6.95,1.4) .. (10.93,3.29)   ;
\draw    (315,213.5) -- (298.3,232.98) ;
\draw [shift={(297,234.5)}, rotate = 310.6] [color={rgb, 255:red, 0; green, 0; blue, 0 }  ][line width=0.75]    (10.93,-3.29) .. controls (6.95,-1.4) and (3.31,-0.3) .. (0,0) .. controls (3.31,0.3) and (6.95,1.4) .. (10.93,3.29)   ;
\draw    (515,214.5) -- (498.34,233.01) ;
\draw [shift={(497,234.5)}, rotate = 311.99] [color={rgb, 255:red, 0; green, 0; blue, 0 }  ][line width=0.75]    (10.93,-3.29) .. controls (6.95,-1.4) and (3.31,-0.3) .. (0,0) .. controls (3.31,0.3) and (6.95,1.4) .. (10.93,3.29)   ;
\draw    (335,214) -- (326.79,233.04) ;
\draw [shift={(326,234.88)}, rotate = 293.32] [color={rgb, 255:red, 0; green, 0; blue, 0 }  ][line width=0.75]    (10.93,-3.29) .. controls (6.95,-1.4) and (3.31,-0.3) .. (0,0) .. controls (3.31,0.3) and (6.95,1.4) .. (10.93,3.29)   ;
\draw    (537,213) -- (526.95,231.74) ;
\draw [shift={(526,233.5)}, rotate = 298.22] [color={rgb, 255:red, 0; green, 0; blue, 0 }  ][line width=0.75]    (10.93,-3.29) .. controls (6.95,-1.4) and (3.31,-0.3) .. (0,0) .. controls (3.31,0.3) and (6.95,1.4) .. (10.93,3.29)   ;
\draw    (352.5,214.38) -- (359.33,233.62) ;
\draw [shift={(360,235.5)}, rotate = 250.45] [color={rgb, 255:red, 0; green, 0; blue, 0 }  ][line width=0.75]    (10.93,-3.29) .. controls (6.95,-1.4) and (3.31,-0.3) .. (0,0) .. controls (3.31,0.3) and (6.95,1.4) .. (10.93,3.29)   ;
\draw    (551,214.5) -- (560.11,232.71) ;
\draw [shift={(561,234.5)}, rotate = 243.43] [color={rgb, 255:red, 0; green, 0; blue, 0 }  ][line width=0.75]    (10.93,-3.29) .. controls (6.95,-1.4) and (3.31,-0.3) .. (0,0) .. controls (3.31,0.3) and (6.95,1.4) .. (10.93,3.29)   ;
\draw    (572,268.38) -- (576.41,282.59) ;
\draw [shift={(577,284.5)}, rotate = 252.76999999999998] [color={rgb, 255:red, 0; green, 0; blue, 0 }  ][line width=0.75]    (10.93,-3.29) .. controls (6.95,-1.4) and (3.31,-0.3) .. (0,0) .. controls (3.31,0.3) and (6.95,1.4) .. (10.93,3.29)   ;
\draw    (549.75,268) -- (539.66,282.24) ;
\draw [shift={(538.5,283.88)}, rotate = 305.32] [color={rgb, 255:red, 0; green, 0; blue, 0 }  ][line width=0.75]    (10.93,-3.29) .. controls (6.95,-1.4) and (3.31,-0.3) .. (0,0) .. controls (3.31,0.3) and (6.95,1.4) .. (10.93,3.29)   ;
\draw    (561,268.63) -- (557.48,282.93) ;
\draw [shift={(557,284.88)}, rotate = 283.83] [color={rgb, 255:red, 0; green, 0; blue, 0 }  ][line width=0.75]    (10.93,-3.29) .. controls (6.95,-1.4) and (3.31,-0.3) .. (0,0) .. controls (3.31,0.3) and (6.95,1.4) .. (10.93,3.29)   ;

\draw   (0,30) -- (100,30) -- (100,130) -- (0,130) -- cycle ;
\draw   (0,210) -- (50,210) -- (50,260) -- (0,260) -- cycle ;
\draw   (50,260) -- (100,260) -- (100,310) -- (50,310) -- cycle ;
\draw   (0,260) -- (50,260) -- (50,310) -- (0,310) -- cycle ;
\draw   (50,210) -- (100,210) -- (100,260) -- (50,260) -- cycle ;

\draw   (0,260) -- (25,260) -- (25,285) -- (0,285) -- cycle ;
\draw   (25,260) -- (50,260) -- (50,285) -- (25,285) -- cycle ;
\draw   (0,285) -- (25,285) -- (25,310) -- (0,310) -- cycle ;
\draw   (25,285) -- (50,285) -- (50,310) -- (25,310) -- cycle ;
\draw   (50,210) -- (75,210) -- (75,235) -- (50,235) -- cycle ;
\draw   (75,210) -- (100,210) -- (100,235) -- (75,235) -- cycle ;
\draw   (50,235) -- (75,235) -- (75,260) -- (50,260) -- cycle ;
\draw   (75,235) -- (100,235) -- (100,260) -- (75,260) -- cycle ;

\draw   (130,30) -- (180,30) -- (180,80) -- (130,80) -- cycle ;
\draw   (180,80) -- (230,80) -- (230,130) -- (180,130) -- cycle ;
\draw   (130,80) -- (180,80) -- (180,130) -- (130,130) -- cycle ;
\draw   (180,30) -- (230,30) -- (230,80) -- (180,80) -- cycle ;

\draw   (130,210) -- (180,210) -- (180,260) -- (130,260) -- cycle ;
\draw   (180,260) -- (230,260) -- (230,310) -- (180,310) -- cycle ;
\draw   (130,260) -- (180,260) -- (180,310) -- (130,310) -- cycle ;
\draw   (180,210) -- (230,210) -- (230,260) -- (180,260) -- cycle ;

\draw   (130,260) -- (155,260) -- (155,285) -- (130,285) -- cycle ;
\draw   (155,260) -- (180,260) -- (180,285) -- (155,285) -- cycle ;
\draw   (130,285) -- (155,285) -- (155,310) -- (130,310) -- cycle ;
\draw   (155,285) -- (180,285) -- (180,310) -- (155,310) -- cycle ;
\draw   (180,210) -- (205,210) -- (205,235) -- (180,235) -- cycle ;
\draw   (205,210) -- (230,210) -- (230,235) -- (205,235) -- cycle ;
\draw   (180,235) -- (205,235) -- (205,260) -- (180,260) -- cycle ;
\draw   (205,235) -- (230,235) -- (230,260) -- (205,260) -- cycle ;

\draw   (205,210) -- (217.5,210) -- (217.5,222.5) -- (205,222.5) -- cycle ;
\draw   (217.5,210) -- (230,210) -- (230,222.5) -- (217.5,222.5) -- cycle ;
\draw   (205,222.5) -- (217.5,222.5) -- (217.5,235) -- (205,235) -- cycle ;
\draw   (217.5,222.5) -- (230,222.5) -- (230,235) -- (217.5,235) -- cycle ;

\draw (37,3) node [anchor=north west][inner sep=0.75pt]    {$\mathcal{T}_{0}$};
\draw (37,183) node [anchor=north west][inner sep=0.75pt]    {$\mathcal{T}_{2}$};
\draw (476,69) node [anchor=north west][inner sep=0.75pt]    {$T_{0,1}$};
\draw (330,173) node [anchor=north west][inner sep=0.75pt]    {$T_{1,1}$};
\draw (430,173) node [anchor=north west][inner sep=0.75pt]    {$T_{1,2}$};
\draw (530,173) node [anchor=north west][inner sep=0.75pt]    {$T_{1,3}$};
\draw (630,173) node [anchor=north west][inner sep=0.75pt]    {$T_{1,4}$};
\draw (273,248) node [anchor=north west][inner sep=0.75pt]  [font=\scriptsize]  {$T_{2,1}$};
\draw (313,248) node [anchor=north west][inner sep=0.75pt]  [font=\scriptsize]  {$T_{2,2}$};
\draw (353,248) node [anchor=north west][inner sep=0.75pt]  [font=\scriptsize]  {$T_{2,3}$};
\draw (393,248) node [anchor=north west][inner sep=0.75pt]  [font=\scriptsize]  {$T_{2,4}$};
\draw (474,248) node [anchor=north west][inner sep=0.75pt]  [font=\scriptsize]  {$T_{2,5}$};
\draw (514,248) node [anchor=north west][inner sep=0.75pt]  [font=\scriptsize]  {$T_{2,6}$};
\draw (554,248) node [anchor=north west][inner sep=0.75pt]  [font=\scriptsize]  {$T_{2,7}$};
\draw (594,248) node [anchor=north west][inner sep=0.75pt]  [font=\scriptsize]  {$T_{2,8}$};
\draw (530,303) node [anchor=north west][inner sep=0.75pt]  [font=\scriptsize]  {$T_{3,1}$};
\draw (590,303) node [anchor=north west][inner sep=0.75pt]  [font=\scriptsize]  {$T_{3,4}$};
\draw (570,303) node [anchor=north west][inner sep=0.75pt]  [font=\scriptsize]  {$T_{3,3}$};
\draw (550,303) node [anchor=north west][inner sep=0.75pt]  [font=\scriptsize]  {$T_{3,2}$};
\draw (167,3) node [anchor=north west][inner sep=0.75pt]    {$\mathcal{T}_{1}$};
\draw (167,183) node [anchor=north west][inner sep=0.75pt]    {$\mathcal{T}_{3}$};

\end{tikzpicture}
}
\end{minipage}
\caption{Initial triangulation $\mathcal{T}_0$, adaptively red-refined triangulations $\mathcal{T}_1,\mathcal{T}_2, \mathcal{T}_3$, and the corresponding tree structure. We only merge elements on the same level that have the same parent element. If all edges starting from the parent element lead to leaves, i.e., there are no further successors on a lower level, then the elements are coarsened back to their parent element. Otherwise, the lower level needs to be coarsened first. Here, $T_{2,1},\ldots, T_{2,4}$ can be coarsened to $T_{1,1}$ but $T_{2,5},\ldots,T_{2,8}$ can not be coarsened to $T_{1,3}$. For this to work, $T_{3,1},\ldots,T_{3,4}$ is first coarsened back to $T_{2,7}$. Then, the successors of $T_{1,3}$ do not have any outgoing edges to elements on a smaller level anymore and can therefore be coarsened. }
\label{fig:tree}
\end{figure}
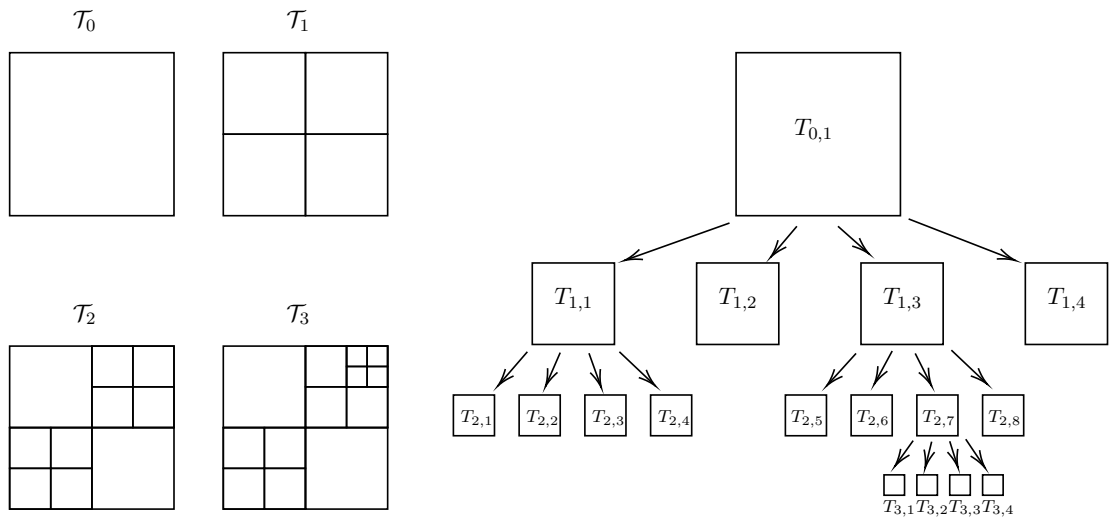

Keeping in mind that this tree structure is not explicitly given in our data structure, information about the parent element and the level of elements is a priori unknown. In a tree structure this can be easily read. In our case, it is necessary to determine an easy-to-verify criterion that checks for admissible quartets. This forms the basis for our coarsening procedures.

To distinguish such an admissible quartet of elements from other quartets, we first take a look at the refinement of an element. For quadrilateral meshes, a red refinement creates five new nodes. One of these nodes is the middle node and the other four nodes are the midpoints of the element's four edges. They are connected to the middle node by an edge. In the following, we want to speak of a five-point stencil when we talk about the five newly created nodes and their connections via edges. The four vertices of the parent element stem from an older generation than the five new nodes, see Figure~\ref{fig:father}.

One idea to determine an admissible quartet of elements is to find out whether or not a middle node belongs to the latest generation in each element adjacent to that node. Usually, out of the five-point stencil, the middle node is the node that is lastly added from these five points. This indicates to determine the newest nodes per element. However, to apply this method in the context of red-blue meshes, a determination via the newest nodes no longer works because with earlier blue refinements the middle node is created before some other nodes of the five-point stencil and is therefore considered older.

In contrast, the oldest nodes in the element can be determined robustly. The oldest nodes always originate from the respective parent element and therefore serve as optimal orientation features. We are now exploiting this as follows: The elements are numbered in mathematically positive order. If we now know which node is the oldest, we can conclude that the second following node within the element must be the corresponding middle node of a red refinement. If this middle node is determined to be the middle node for \emph{all} four adjacent elements, then these four elements form an admissible quartet of elements. Algorithm~\ref{alg:Qadm} describes the just indicated procedure in more detail. A supportive illustration is given in Figure~\ref{fig:middlenode}.

\begin{figure}
\centering
\hspace*{5mm}
\begin{minipage}{0.15\textwidth}
\begin{tikzpicture}
\coordinate (1) at (-0.5,-0.5);
\coordinate (2) at (1,-0.5);
\coordinate (3) at (1,1);
\coordinate (4) at (-0.5,1);
\coordinate (5) at (0.25,-0.5);
\coordinate (6) at (1,0.25);
\coordinate (7) at (0.25,1);
\coordinate (8) at (-0.5,0.25);
\coordinate (9) at ($(1)!0.5!(3)$);
\draw[red] (5)--(7);
\draw[red] (6)--(8);
\draw (1)--(2)--(3)--(4)--(1);
\fill (1) circle (2pt);
\fill (2) circle (2pt);
\fill (3) circle (2pt);
\fill (4) circle (2pt);
\fill[pastelred] (5) circle (2pt);
\fill[pastelred]  (6) circle (2pt);
\fill[pastelred]  (7) circle (2pt);
\fill[pastelred]  (8) circle (2pt);
\fill[pastelred]  (9) circle (2pt);
\end{tikzpicture}
\end{minipage}
\begin{minipage}{0.15\textwidth}
\begin{tikzpicture}
\coordinate (1) at (-0.5,-0.5);
\coordinate (2) at (1,-0.5);
\coordinate (3) at (1,1);
\coordinate (4) at (-0.5,1);
\fill (1) circle (2pt);
\fill (2) circle (2pt);
\fill (3) circle (2pt);
\fill (4) circle (2pt);
\draw (1)--(2)--(3)--(4)--(1);
\end{tikzpicture}
\end{minipage}\hspace*{1cm}
\begin{minipage}{0.15\textwidth}
\begin{tikzpicture}
\coordinate (1) at (-1,0);
\coordinate (2) at (1,0);
\coordinate (3) at (0,1.5);
\coordinate (4) at (0,0);
\coordinate (5) at ($(2)!0.5!(3)$);
\coordinate (6) at ($(3)!0.5!(1)$);
\draw (1)--(2)--(3)--(1);
\draw[pastelred] (4) -- (5)-- (6)--(4);
\fill (1) circle (2pt);
\fill (2) circle (2pt);
\fill (3) circle (2pt);
\fill[pastelred] (4) circle (2pt);
\fill[pastelred] (5) circle (2pt);
\fill[pastelred] (6) circle (2pt);
\end{tikzpicture}
\end{minipage}
\begin{minipage}{0.15\textwidth}
\begin{tikzpicture}
\coordinate (1) at (-1,0);
\coordinate (2) at (1,0);
\coordinate (3) at (0,1.5);
\draw (1)--(2)--(3)--(1);
\fill (1) circle (2pt);
\fill (2) circle (2pt);
\fill (3) circle (2pt);
\end{tikzpicture}
\end{minipage}
\caption{Left: A red refinement of a quadrilateral and its parent element. A five-point stencil (in red) is introduced in a red refinement.  Right: Red-refined triangular element and its parent element. A middle element (in red) is introduced in a red refinement. For both holds that the nodes in red stem from a newer generation than the nodes in black. }
\label{fig:father}
\end{figure}
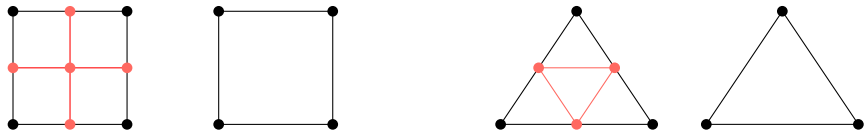

\begin{algorithm}[h!]
\caption{$\adm$ (Quadrilateral)}
\label{alg:Qadm}
\begin{algorithmic}[1]
\State \textbf{Input:} Quadrilateral mesh $\mathcal{T}$ obtained by red-refinement. $\mathcal{T}=\left\{T_1,\ldots,T_m\right\}$ with $T_i = \conv(v_{i1},v_{i2},v_{i3},v_{i4})$ for all $i \in \left\{1,\ldots,m\right\}$. The vertices $v_{ij}$ are cyclically permuted such that $v_{i1}$ is an older vertex than $v_{ij}$ for $j \in \left\{2,3,4\right\}$.
\State \textbf{Output:} Admissible set $\mathcal{A}$ of quartets $Q$ for coarsening.

\Procedure{Admissible}{$\mathcal{T}$}
\State $ \mathcal{A} \coloneqq \emptyset$
\ForAll{$i \in \left\{1,\ldots,m\right\}$} 
\State find index set $I \subset \left\{1,\ldots,m\right\}$ such that $v_{i3}=v_{j3}$ for all $j \in I$ 
\If{$\#I=4$ and $v_{i3} \not\in \mathcal{N}_0$}
\State $Q \gets \left\{T \in \mathcal{T}~|~v_{i3} \in T\right\}$
\State $\mathcal{A} \gets \mathcal{A} \cup \left\{Q\right\}$
\EndIf
         \EndFor
                  \State \textbf{return} ${\mathcal{A}}$
\EndProcedure
\end{algorithmic}
\end{algorithm}

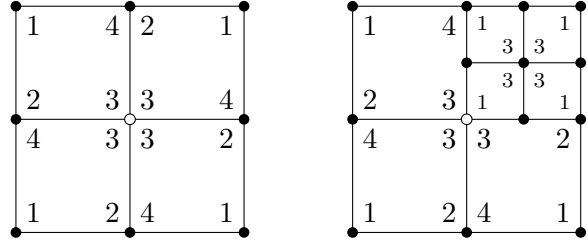
\begin{figure}
\centering
\hspace*{15mm}
\begin{minipage}{0.25\textwidth}
\begin{tikzpicture}
\coordinate (1) at (-1,-1);
\coordinate (2) at (2,-1);
\coordinate (3) at (2,2);
\coordinate (4) at (-1,2);
\coordinate (5) at (0.5,-1);
\coordinate (6) at (2,0.5);
\coordinate (7) at (0.5,2);
\coordinate (8) at (-1,0.5);
\coordinate (9) at ($(1)!0.5!(3)$);
\draw (5)--(7);
\draw (6)--(8);
\draw (1)--(2)--(3)--(4)--(1);
\fill (1) circle (2pt);
\fill (2) circle (2pt);
\fill (3) circle (2pt);
\fill (4) circle (2pt);
\fill (5) circle (2pt);
\fill (6) circle (2pt);
\fill (7) circle (2pt);
\fill (8) circle (2pt);
\node[circle,draw=black, fill=white,inner sep=0pt,minimum size=4pt] at (9) {};
\node at (9) [below right] {3};
\node at (9) [below left] {3};
\node at (9) [above right] {3};
\node at (9) [above left] {3};
\node at (4) [below right] {1};
\node at (3) [below left] {1};
\node at (1) [above right] {1};
\node at (2) [above left] {1};
\node at (7) [below right] {2};
\node at (6) [below left] {2};
\node at (5) [above left] {2};
\node at (8) [above right] {2};
\node at (8) [below right] {4};
\node at (7) [below left] {4};
\node at (6) [above left] {4};
\node at (5) [above right] {4};
\end{tikzpicture}\end{minipage}
\begin{minipage}{0.25\textwidth}
\begin{tikzpicture}
\coordinate (1) at (-1,-1);
\coordinate (2) at (2,-1);
\coordinate (3) at (2,2);
\coordinate (4) at (-1,2);
\coordinate (5) at (0.5,-1);
\coordinate (6) at (2,0.5);
\coordinate (7) at (0.5,2);
\coordinate (8) at (-1,0.5);
\coordinate (9) at ($(1)!0.5!(3)$);
\coordinate (10) at ($(9)!0.5!(6)$);
\coordinate (11) at ($(9)!0.5!(7)$);
\coordinate (12) at ($(7)!0.5!(3)$);
\coordinate (13) at ($(3)!0.5!(6)$);
\coordinate (14) at ($(9)!0.5!(3)$);
\draw (5)--(7);
\draw (6)--(8);
\draw (1)--(2)--(3)--(4)--(1);
\draw (10)--(12);
\draw (11)--(13);
\fill (1) circle (2pt);
\fill (2) circle (2pt);
\fill (3) circle (2pt);
\fill (4) circle (2pt);
\fill (5) circle (2pt);
\fill (6) circle (2pt);
\fill (7) circle (2pt);
\fill (8) circle (2pt);
\node[circle,draw=black, fill=white,inner sep=0pt,minimum size=4pt] at (9) {};
\fill (11) circle (2pt);
\fill (12) circle (2pt);
\fill (13) circle (2pt);
\fill  (10) circle (2pt);
\fill  (14) circle (2pt);
\node at (9) [below right] {3};
\node at (9) [below left] {3};
\node at (9) [above right] {\scriptsize 1};
\node at (9) [above left] {3};
\node at (14) [below right] {\scriptsize 3};
\node at (14) [below left] {\scriptsize 3};
\node at (14) [above right] {\scriptsize 3};
\node at (14) [above left] {\scriptsize 3};
\node at (4) [below right] {1};
\node at (3) [below left] {\scriptsize 1};
\node at (1) [above right] {1};
\node at (2) [above left] {1};
\node at (7) [below right] {\scriptsize 1};
\node at (6) [below left] {2};
\node at (5) [above left] {2};
\node at (8) [above right] {2};
\node at (8) [below right] {4};
\node at (7) [below left] {4};
\node at (6) [above left] {\scriptsize 1};
\node at (5) [above right] {4};
\end{tikzpicture}
\end{minipage}
\caption{Element numbering of red refinements. The storage position within an element is indicated by the numbers. The oldest node of an element is stored at position one. As elements are numbered counterclockwise, the node at the third position is a candidate for a middle node. To determine if it is a middle node of an admissible quartet, this node must be shared by four elements having this node on position three in common. Left: The middle node (in white) is shared by all elements on position three, i.e., the four elements form an admissible quartet. Right: The same middle node (in white) is now blocked because of the refined upper right corner. Here, the position three is not shared by all elements, so we are not on the last level. The elements adjacent to this middle node are thus not an admissible quartet.}
\label{fig:middlenode}
\end{figure}

For triangular meshes, a red refinement creates three new nodes. The newly created nodes are the midpoints of the element's three edges that are connected with each other to form a new triangle. In the following, we want to speak of a middle element when we talk about the three newly created nodes and their connections via edges. As for quadrilateral meshes, the three vertices of the parent element stem from an older generation than the three new nodes.

Again, we want to find admissible quartets of elements. Unlike in the quadrilateral case, we can not work with a middle node. However, we can determine the middle elements. A middle element and its three direct neighbors thus form an admissible quartet of elements. So the first task is to find middle elements. A middle element consists only of nodes of a newer generation. To determine these, we make use of the oldest nodes.

\begin{algorithm}[h!]
\caption{$\adm$ (Triangular)}\label{alg:Tadm}
\begin{algorithmic}[1]
\State \textbf{Input:} Triangular mesh $\mathcal{T}$ obtained by red-refinement. $\mathcal{T}=\left\{T_1,\ldots,T_m\right\}$ with $T_i = \conv(v_{i1},v_{i2},v_{i3})$ for all $i \in \left\{1,\ldots,m\right\}$. The vertices $v_{ij}$ are cyclically permuted such that $v_{i1}$ is an older vertex than $v_{ij}$ for $j \in \left\{2,3\right\}$.
\State \textbf{Output:} Admissible set $\mathcal{A}$ of quartets $Q$ for coarsening.
\Procedure{Admissible}{$\mathcal{T}$}
\ForAll{edges $e \in \mathcal{E}(\mathcal{T})$}
\State find set $I_e= \left\{ i \in \left\{1,\ldots,m\right\}~|~ e \text{ is edge of } T_i\right\}$
\State $v_{\mathrm{old},e}=  \old\limits_{i \in I_e}  v_{i1}$ \Comment{$\old$ returns the oldest node.} 
\EndFor
\State $\mathcal{O} = \left\{ T \in \mathcal{T}~|~v_{\mathrm{old},e_1}=v_{\mathrm{old},e_2}=v_{\mathrm{old},e_3}, \text{ where } e_1,e_2,e_3 \text{ are edges of } T\right\}$ \State $\mathcal{M} = \mathcal{T}\setminus \mathcal{O}$ \Comment{Middle elements.}
\State $\mathcal{A} \coloneqq \emptyset$
\ForAll{ $T \in \mathcal{M}$}
\If{$T$ has $3$ direct neighbors}  \Comment{Quartet having the same parent element.}
\State $Q = \big\{T\big\} \cup \left\{\hat{T} \in \mathcal{T}~|~\hat{T}  \text{ is direct neighbor of } T \right\}$
\State $\mathcal{A} \gets \mathcal{A} \cup \left\{Q\right\}$
 \EndIf
\EndFor
\State \textbf{return} ${\mathcal{A}}$\EndProcedure
\end{algorithmic}
\end{algorithm}

\begin{figure}
\tikzset{every picture/.style={line width=0.75pt}} 

\begin{tikzpicture}[x=0.75pt,y=0.75pt,yscale=-1,xscale=1]

\draw   (182.73,42.4) -- (16.14,200.65) -- (16.14,42.4) -- cycle ;
\draw   (16.14,200.65) -- (182.73,42.4) -- (182.73,200.65) -- cycle ;
\draw   (16.14,122.35) -- (99.44,43.37) -- (99.44,122.35) -- cycle ;
\draw   (182.73,122.35) -- (99.44,201.33) -- (99.44,122.35) -- cycle ;
\draw   (391.64,41.07) -- (225.05,199.32) -- (225.05,41.07) -- cycle ;
\draw   (225.05,199.32) -- (391.64,41.07) -- (391.64,199.32) -- cycle ;
\draw   (225.05,121.02) -- (308.35,41.9) -- (308.35,121.02) -- cycle ;
\draw   (391.64,120.19) -- (308.35,199.32) -- (308.35,120.19) -- cycle ;

\draw   (598,40.24) -- (431.41,198.49) -- (431.41,40.24) -- cycle ;
\draw   (431.41,198.49) -- (598,40.24) -- (598,198.49) -- cycle ;
\draw  [fill=pastelred  ,fill opacity=1 ] (431.41,120.19) -- (514.7,41.07) -- (514.7,120.19) -- cycle ;
\draw  [fill=pastelred  ,fill opacity=1 ] (598,119.36) -- (514.7,198.49) -- (514.7,119.36) -- cycle ;
\draw   (182,82) -- (142,122) -- (142,82) -- cycle ;
\draw   (390.6,79.8) -- (350.6,119.8) -- (350.6,79.8) -- cycle ;
\draw  [fill=pastelred  ,fill opacity=1 ] (598,79) -- (558,119) -- (558,79) -- cycle ;
\node at (570,95) {$\ast$};
\node at (490,95) {\Large$\ast$};

\draw (11.02,202.2) node [anchor=north west][inner sep=0.75pt]  [font=\small]  {$1$};
\draw (175.98,201.28) node [anchor=north west][inner sep=0.75pt]  [font=\small]  {$2$};
\draw (174.85,25.3) node [anchor=north west][inner sep=0.75pt]  [font=\small]  {$3$};
\draw (13.65,27.3) node [anchor=north west][inner sep=0.75pt]  [font=\small]  {$4$};
\draw (94.25,201.36) node [anchor=north west][inner sep=0.75pt]  [font=\small]  {$5$};
\draw (100.45,122.99) node [anchor=north west][inner sep=0.75pt]  [font=\small]  {$6$};
\draw (3.13,112.5) node [anchor=north west][inner sep=0.75pt]  [font=\small]  {$7$};
\draw (183.61,113.33) node [anchor=north west][inner sep=0.75pt]  [font=\small]  {$8$};
\draw (96.31,26.42) node [anchor=north west][inner sep=0.75pt]  [font=\small]  {$9$};
\draw (67.09,167.32) node [anchor=north west][inner sep=0.75pt]  [color={rgb, 255:red, 208; green, 2; blue, 27 }  ,opacity=1 ]  {$1$};
\draw (32.04,138.25) node [anchor=north west][inner sep=0.75pt]  [color={rgb, 255:red, 208; green, 2; blue, 27 }  ,opacity=1 ]  {$1$};
\draw (119.65,138.25) node [anchor=north west][inner sep=0.75pt]  [color={rgb, 255:red, 208; green, 2; blue, 27 }  ,opacity=1 ]  {$5$};
\draw (150.32,167.32) node [anchor=north west][inner sep=0.75pt]  [color={rgb, 255:red, 208; green, 2; blue, 27 }  ,opacity=1 ]  {$2$};
\draw (119.65,59.35) node [anchor=north west][inner sep=0.75pt]  [color={rgb, 255:red, 208; green, 2; blue, 27 }  ,opacity=1 ]  {$3$};
\draw (67.09,88.42) node [anchor=north west][inner sep=0.75pt]  [color={rgb, 255:red, 208; green, 2; blue, 27 }  ,opacity=1 ]  {$6$};
\draw (32.04,59.35) node [anchor=north west][inner sep=0.75pt]  [color={rgb, 255:red, 208; green, 2; blue, 27 }  ,opacity=1 ]  {$4$};
\draw (268.59,183.09) node [anchor=north west][inner sep=0.75pt]  [color={rgb, 255:red, 208; green, 2; blue, 27 }  ,opacity=1 ]  {$1$};
\draw (268.59,105.86) node [anchor=north west][inner sep=0.75pt]  [color={rgb, 255:red, 208; green, 2; blue, 27 }  ,opacity=1 ]  {$1$};
\draw (225.66,142.4) node [anchor=north west][inner sep=0.75pt]  [color={rgb, 255:red, 208; green, 2; blue, 27 }  ,opacity=1 ]  {$1$};
\draw (268.59,154.86) node [anchor=north west][inner sep=0.75pt]  [color={rgb, 255:red, 208; green, 2; blue, 27 }  ,opacity=1 ]  {$1$};
\draw (309.76,142.4) node [anchor=north west][inner sep=0.75pt]  [color={rgb, 255:red, 208; green, 2; blue, 27 }  ,opacity=1 ]  {$1$};
\draw (351.82,183.09) node [anchor=north west][inner sep=0.75pt]  [color={rgb, 255:red, 208; green, 2; blue, 27 }  ,opacity=1 ]  {$2$};
\draw (380.73,150.71) node [anchor=north west][inner sep=0.75pt]  [color={rgb, 255:red, 208; green, 2; blue, 27 }  ,opacity=1 ]  {$2$};
\draw (351.82,153.2) node [anchor=north west][inner sep=0.75pt]  [color={rgb, 255:red, 208; green, 2; blue, 27 }  ,opacity=1 ]  {$2$};
\draw (384.17,59) node [anchor=north west][inner sep=0.75pt]  [font=\scriptsize,color={rgb, 255:red, 208; green, 2; blue, 27 }  ,opacity=1 ]  {$3$};
\draw (268.59,74.3) node [anchor=north west][inner sep=0.75pt]  [color={rgb, 255:red, 208; green, 2; blue, 27 }  ,opacity=1 ]  {$4$};
\draw (309.76,71.81) node [anchor=north west][inner sep=0.75pt]  [color={rgb, 255:red, 208; green, 2; blue, 27 }  ,opacity=1 ]  {$3$};
\draw (351.82,38.59) node [anchor=north west][inner sep=0.75pt]  [color={rgb, 255:red, 208; green, 2; blue, 27 }  ,opacity=1 ]  {$3$};
\draw (342.34,63.33) node [anchor=north west][inner sep=0.75pt]  [color={rgb, 255:red, 208; green, 2; blue, 27 }  ,opacity=1 ]  {$3$};
\draw (370.53,59) node [anchor=north west][inner sep=0.75pt]  [font=\scriptsize,color={rgb, 255:red, 208; green, 2; blue, 27 }  ,opacity=1 ]  {$3$};
\draw (226.53,66) node [anchor=north west][inner sep=0.75pt]  [color={rgb, 255:red, 208; green, 2; blue, 27 }  ,opacity=1 ]  {$4$};
\draw (257.2,38.59) node [anchor=north west][inner sep=0.75pt]  [color={rgb, 255:red, 208; green, 2; blue, 27 }  ,opacity=1 ]  {$4$};
\draw (128.98,64.28) node [anchor=north west][inner sep=0.75pt]  [font=\small]  {$10$};
\draw (131.98,124.28) node [anchor=north west][inner sep=0.75pt]  [font=\small]  {$12$};
\draw (184.98,72.28) node [anchor=north west][inner sep=0.75pt]  [font=\small]  {$11$};
\draw (124.32,101.32) node [anchor=north west][inner sep=0.75pt]  [color={rgb, 255:red, 208; green, 2; blue, 27 }  ,opacity=1 ]  {$6$};
\draw (144,85) node [anchor=north west][inner sep=0.75pt]  [color={rgb, 255:red, 208; green, 2; blue, 27 }  ,opacity=1 ]  {$10$};
\draw (167.32,60.32) node [anchor=north west][inner sep=0.75pt]  [color={rgb, 255:red, 208; green, 2; blue, 27 }  ,opacity=1 ]  {$3$};
\draw (167,100) node [anchor=north west][inner sep=0.75pt]  [color={rgb, 255:red, 208; green, 2; blue, 27 }  ,opacity=1 ]  {$8$};
\draw (364.73,79.13) node [anchor=north west][inner sep=0.75pt]  [font=\scriptsize,color={rgb, 255:red, 208; green, 2; blue, 27 }  ,opacity=1 ]  {$3$};
\draw (364.33,90) node [anchor=north west][inner sep=0.75pt]  [font=\scriptsize,color={rgb, 255:red, 208; green, 2; blue, 27 }  ,opacity=1 ]  {$8$};
\draw (373.4,108) node [anchor=north west][inner sep=0.75pt]  [font=\scriptsize,color={rgb, 255:red, 208; green, 2; blue, 27 }  ,opacity=1 ]  {$8$};
\draw (383.8,95.2) node [anchor=north west][inner sep=0.75pt]  [font=\scriptsize,color={rgb, 255:red, 208; green, 2; blue, 27 }  ,opacity=1 ]  {$8$};
\draw (350.6,90.8) node [anchor=north west][inner sep=0.75pt]  [font=\scriptsize,color={rgb, 255:red, 208; green, 2; blue, 27 }  ,opacity=1 ]  {$6$};
\draw (331.8,95.2) node [anchor=north west][inner sep=0.75pt]  [font=\scriptsize,color={rgb, 255:red, 208; green, 2; blue, 27 }  ,opacity=1 ]  {$6$};
\draw (331.8,108.2) node [anchor=north west][inner sep=0.75pt]  [font=\scriptsize,color={rgb, 255:red, 208; green, 2; blue, 27 }  ,opacity=1 ]  {$6$};
\draw (351.8,119.45) node [anchor=north west][inner sep=0.75pt]  [color={rgb, 255:red, 208; green, 2; blue, 27 }  ,opacity=1 ]  {$5$};

\end{tikzpicture}

\caption{Determination of middle elements via an easy-to-verify criterion. Left: The numbers written in black are the indices of the nodes. The nodes 1--4 are part of the initial triangulation, 5--11 were added via refinement. To each element we assign the index of the oldest node per element (in red). Here, the oldest node is represented by the smallest index. Middle: We now compare the indices of the oldest nodes per element with the neighboring element and assign the older value (the smaller index) to the edge that is shared by these two elements. Boundary edges do not have a direct neighbor -- no comparison is needed. Right: We consider an element a middle element if it does not have the same oldest node assigned on all three edges (in red). The asterisks indicate that this middle element and its direct neighbors is considered an admissible quartet, whereas a middle element without an asterisk does not have three direct neighbors and is thus not in the set of admissible quartets. }
\label{fig:middleelement}
\end{figure}
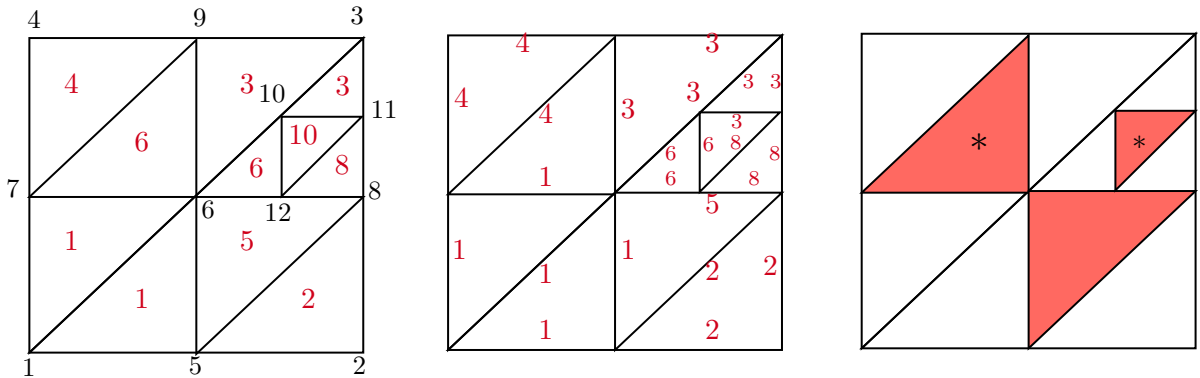

The determination of the admissible quartets of elements in the case of triangular meshes is described in Algorithm~\ref{alg:Tadm}. An illustration of the algorithm is shown in Figure~\ref{fig:middleelement}. We will now briefly introduce the main ideas that lead to this procedure. The algorithm takes advantage of the property that a middle element consists only of newer nodes and thus differs from the other elements - since these always contain a node from an older generation. The procedure therefore provides for the oldest node to be determined for each element. If the element $T_i = \conv(v_{i1},v_{i2},v_{i3})$ is given, we assume that $v_{i1}$ is the oldest node of this element. If not, we will cyclically permute the vertices of the element until $v_{i1}$ is the oldest node of this element. The next step is to compare the oldest node of an element with the oldest nodes of the direct neighbor element. We assign the oldest nodes of this comparison to the edge shared by these two elements. We find that the three elements of a red refinement that are not the middle element are assigned the same oldest node on each edge. So if this is not the case, it is considered the middle element. 

Finally, in order to obtain the admissible quartets of elements, we examine whether a middle element thus determined has three direct neighbours. If this is the case, an admissible quartet is given by this middle element and its three direct neighbours.

\subsubsection{$\mar$}  In the step $\adm$, we determined all admissible quartets that can be considered for elimination. With the step $\mar$, adaptivity is incorporated into a coarsening step. This includes that the determined set $\mathcal{A}$ is reduced by a marking operation. We consider the following marking operation: If an element of an admissible quartet is marked for coarsening, all elements of this quartet are considered for coarsening. If no element of an admissible quartet is marked, this quartet is not coarsened. This means that $\mathcal{A}$ is updated and the number of quartets for coarsening is reduced. This step is identical for both triangular and quadrilateral meshes, and is described in Algorithm~\ref{alg:mark}.

\begin{algorithm}[h!]
\caption{$\mar$}\label{alg:mark}
\begin{algorithmic}[1]
\State \textbf{Input:} Admissible set $\mathcal{A}$ of quartets and a set of marked elements $\mathcal{T}_\mathrm{mark}$.
\State \textbf{Output:} Admissible and marked set $\mathcal{A}$ of quartets $Q$ for coarsening.
\Procedure{Mark}{$\mathcal{A},\mathcal{T}_\mathrm{mark}$}
\State $\bar{\mathcal{A}}\coloneqq\emptyset$
\ForAll{$T \in \mathcal{T}_\mathrm{mark}$} \Comment{Consider all elements in a quartet for coarsening or none.}
\If{$\exists~ Q \in \mathcal{A}$ such that $T \in Q$ }
\State $\bar{\mathcal{A}} \gets \bar{\mathcal{A}}\cup \left\{Q\right\}$
\EndIf
\EndFor
\State $\mathcal{A} \gets \bar{\mathcal{A}}$
\State \textbf{return} ${\mathcal{A}}$ 
\EndProcedure
\end{algorithmic}
\end{algorithm}

Note that the convention of at least one marked element per quartet is not restrictive. By a small modification, we can also require that all elements of an admissible quartet need to be marked to be considered for coarsening. 

This marking operation can lead to k-irregular triangulations with $k>1,~ k \in \mathbb{N}$ and non-compliance with the 3-Neighbor Rule. To this end, an additional step $\clos$ is needed.

\subsubsection{$\clos$}

After a $\clos$ step, two properties shall be satisfied: the $1$-irregularity and the $d$-Neighbor Rule with $d=2$ for triangular and $d=3$ for quadrilateral meshes. To this end, we again want to find easy-to-verify criteria to ensure both rules for triangular and quadrilateral meshes.

The $1$-irregularity of a mesh means that there can be at most one hanging node per edge. Without taking neighbor information into account, coarsening can lead to meshes with more than one hanging node per edge, cf.~Figure~\ref{fig:irregular}. To prevent this, coarsening back to the parent element is blocked if one of the neighboring elements has a higher refinement level, i.e., one of the elements in an admissible quartet has an irregular edge. By blocking, we mean that we reduce the set of admissible quartets $\mathcal{A}$ by those quartets that have at least one irregular edge. Note, that this blocking is not always necessary. For example, if there is another admissible quartet for coarsening at a higher level of refinement that causes this irregular edge of an element and this quartet is coarsened back to the parent element, blocking is not necessary. However, to distinguish between those two cases, which are both specified in Figure~\ref{fig:irregular}, one would not only need the immediate neighbor information but the entire neighborhood. Since the refinement also produced patterns of different levels in the subsequent refinement steps, a coarsening without this distinction does not slow down the coarsening process relative to the refinement.

\begin{figure}
\centering
\begin{minipage}{0.22\textwidth}
\begin{tikzpicture}
\coordinate (1) at (0,0);
\coordinate (2) at (1.5,0);
\coordinate (3) at (3,0);
\coordinate (4) at (0,1.5);
\coordinate (5) at (1.5,1.5);
\coordinate (6) at (3,1.5);
\coordinate (7) at ($(1)!0.5!(2)$);
\coordinate (8) at ($(3)!0.5!(2)$);
\coordinate (9) at ($(3)!0.5!(6)$);
\coordinate (10) at ($(5)!0.5!(6)$);
\coordinate (11) at ($(5)!0.5!(4)$);
\coordinate (12) at ($(1)!0.5!(4)$);
\coordinate (13) at ($(12)!0.5!(9)$);
\coordinate (14) at ($(13)!0.5!(12)$);
\coordinate (15) at ($(13)!0.5!(9)$);
\coordinate (16) at ($(15)!0.5!(9)$);
\coordinate (17) at ($(9)!0.5!(6)$);
\coordinate (18) at ($(10)!0.5!(6)$);
\coordinate (19) at ($(10)!0.5!(15)$);
\coordinate (20) at ($(5)!0.5!(13)$);
\coordinate (21) at ($(13)!0.5!(15)$);
\coordinate (22) at ($(20)!0.5!(19)$);
\coordinate (23) at ($(16)!0.5!(18)$);
\coordinate (24) at ($(19)!0.5!(10)$);
\coordinate (25) at ($(5)!0.5!(10)$);
\coordinate (26) at ($(25)!0.5!(10)$);
\coordinate (27) at ($(25)!0.5!(22)$);
\coordinate (28) at ($(19)!0.5!(22)$);
\coordinate (29) at ($(28)!0.5!(26)$);
\draw (1)--(3)--(6)--(4)--(1);
\draw (12)--(9);
\draw (11)--(7);
\draw (5)--(2);
\draw (10)--(8);
\draw (18)--(16);
\draw (25)--(21);
\draw (20)--(17);
\draw (26)--(28);
\draw (27)--(24);
\node at (14) [below right=-1mm] {\textcolor{red}{$\ast$}};
\node at (14) [below left=-1mm] {\textcolor{red}{$\ast$}};
\node at (14) [above right=-1mm] {\textcolor{red}{$\ast$}};
\node at (14) [above left=-1mm] {\textcolor{red}{$\ast$}};
\node at (29) [below right=-1mm] {\tiny $\ast$};
\node at (29) [below left=-1mm] {\tiny$\ast$};
\node at (29) [above right=-1mm] {\tiny$\ast$};
\node at (29) [above left=-1mm] {\tiny$\ast$};
\node at (23) [below right=-1mm] {\textcolor{red}{$\ast$}};
\node at (23) [below left=-1mm] {\textcolor{red}{$\ast$}};
\node at (23) [above right=-1mm] {\textcolor{red}{$\ast$}};
\node at (23) [above left=-1mm] {\textcolor{red}{$\ast$}};
\end{tikzpicture}\end{minipage}
\hspace*{3.5mm}
\begin{minipage}{0.22\textwidth}
\begin{tikzpicture}
\coordinate (1) at (0,0);
\coordinate (2) at (1.5,0);
\coordinate (3) at (3,0);
\coordinate (4) at (0,1.5);
\coordinate (5) at (1.5,1.5);
\coordinate (6) at (3,1.5);
\coordinate (7) at ($(1)!0.5!(2)$);
\coordinate (8) at ($(3)!0.5!(2)$);
\coordinate (9) at ($(3)!0.5!(6)$);
\coordinate (10) at ($(5)!0.5!(6)$);
\coordinate (11) at ($(5)!0.5!(4)$);
\coordinate (12) at ($(1)!0.5!(4)$);
\coordinate (13) at ($(12)!0.5!(9)$);
\coordinate (14) at ($(13)!0.5!(12)$);
\coordinate (15) at ($(13)!0.5!(9)$);
\coordinate (16) at ($(15)!0.5!(9)$);
\coordinate (17) at ($(9)!0.5!(6)$);
\coordinate (18) at ($(10)!0.5!(6)$);
\coordinate (19) at ($(10)!0.5!(15)$);
\coordinate (20) at ($(5)!0.5!(13)$);
\coordinate (21) at ($(13)!0.5!(15)$);
\coordinate (22) at ($(20)!0.5!(19)$);
\coordinate (23) at ($(16)!0.5!(18)$);
\coordinate (24) at ($(19)!0.5!(10)$);
\coordinate (25) at ($(5)!0.5!(10)$);
\coordinate (26) at ($(25)!0.5!(10)$);
\coordinate (27) at ($(25)!0.5!(22)$);
\coordinate (28) at ($(19)!0.5!(22)$);
\coordinate (29) at ($(28)!0.5!(26)$);
\draw (1)--(3)--(6)--(4)--(1);
\draw (13)--(9);
\draw (5)--(2);
\draw (10)--(8);
\draw (25)--(21);
\draw (20)--(19);
\fill (20) circle (2pt);
\fill (13) circle (2pt);
\fill (19) circle (2pt);
\end{tikzpicture}
\end{minipage}
\caption{Left: Admissible quartets for coarsening are marked by asterisks. Right: Coarsening of all marked elements does not result in a $1$-irregular mesh, i.e., coarsening may introduce more than one hanging node per edge. To prevent this case, we have the convention to block a quartet as soon as one of its elements has an irregular edge. Here, quartets marked in red are blocked for coarsening and only the black ones are coarsened. This ensures the $1$-irregularity of the grid. It is not always necessary to block elements marked in red, as can be seen on the right of the mesh. Here, coarsening this quartet would introduce no more than one hanging node per edge. However, during the refinement process, elements of different levels were also created in different steps. Thus, the two shown cases are therefore not distinguished and are handled in the same way.  }
\label{fig:irregular}
\end{figure}
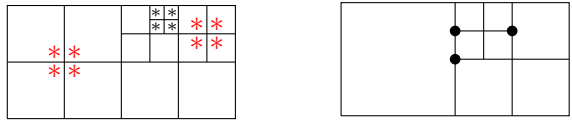

The last desired property is to comply with the $3$-Neighbor Rule which states that any element with three or more neighbors that have been refined must also be refined. This rule shall be effective in a reverse way for coarsening, i.e., an admissible quartet is not coarsened to the parent element if this operation would result in the parent element having at most one direct neighbor. Only the most recently added nodes are decisive for whether coarsening is allowed or not. That means we look at the five-point stencils of the admissible quartets contained in $\mathcal{A}$. By a clever assignment of values for the nodes in the stencils, it is easy to determine if a stencil can be removed without violating the $3$-Neighbor Rule. If it cannot be removed, the set of admissible quartets $\mathcal{A}$ is reduced by the quartet that contains this stencil. To pass this new information to its neighborhood, the process is repeated until no further changes to $\mathcal{A}$ occur. The $\clos$ step for quadrilateral meshes is described in Algorithm~\ref{alg:Qclos}. In particular, the value assignment for the nodes in the stencil is specified. An illustration of the criterion for the $3$-Neighbor Rule in Algorithm~\ref{alg:Qclos} is given in Figure~\ref{fig:numbering}.

\begin{algorithm}[h!]
\caption{$\clos$ (Quadrilateral)}\label{alg:Qclos}
\begin{algorithmic}[1]
\State \textbf{Input:} Quadrilateral mesh $\mathcal{T}$ obtained by red-refinement and admissible set $\mathcal{A}$ of quartets $Q$ for coarsening.
\State \textbf{Output:} Updated admissible set $\mathcal{A}$ of quartets $Q$ for coarsening.
\Procedure{Closure}{$\mathcal{T},\mathcal{A}$}
\ForAll{$Q \in \mathcal{A}$} \Comment{Ensure $1$-irregularity.}
\If{for any $T \in Q$ holds that $T$ has an irregular edge}
\State $\mathcal{A} \gets \mathcal{A}\setminus \left\{Q\right\}$
\EndIf
\EndFor
\While{first run or $\mathcal{A}$ changes} \Comment{Ensure 3-Neighbor Rule.}
\State $\hat{\mathcal{A}}\coloneqq \mathcal{A}$
\ForAll{$Q \in \mathcal{A}$} 
\State $\mathcal{N}_\mathrm{st}\gets \left\{ v \in \mathcal{N}~|~v \text{ is node of the 5-point stencil in } Q\right\}$\Comment{See Figure~\ref{fig:father}.}
\ForAll{$v \in \mathcal{N}_\mathrm{st}$} 
\If{$v$ is a middle node}
\State $w_v \gets 0$
\ElsIf{$v$ is a hanging node, a boundary node or $v$ is shared by two stencils of admissible \hspace*{3.3cm}quartets $Q$ contained in $\mathcal{A}$ }
\State $w_v \gets 2$
\Else
\State $w_v \gets 1$
\EndIf
\EndFor
\If{$\sum\limits_{v \in \mathcal{N}_\mathrm{st}} w_v \leq 5$}
\State $\hat{\mathcal{A}} \gets \hat{\mathcal{A}}\setminus \left\{Q\right\}$
\EndIf
\EndFor
\State $\mathcal{A} \gets \hat{\mathcal{A}}$
\EndWhile
\State \textbf{return} ${\mathcal{A}}$
\EndProcedure
\end{algorithmic}
\end{algorithm}

\begin{figure}
\centering
\begin{minipage}{0.32\textwidth}
\begin{tikzpicture}
\coordinate (1) at (0,0);
\coordinate (2) at (1.5,0);
\coordinate (3) at (3,0);
\coordinate (4) at (4.5,0);
\coordinate (5) at (4.5,1.5);
\coordinate (6) at (4.5,3);
\coordinate (7) at (3,3);
\coordinate (8) at (1.5,3);
\coordinate (9) at (0,3);
\coordinate (10) at (0,1.5);
\coordinate (11) at (1.5,1.5);
\coordinate (12) at (3,1.5);
\coordinate (13) at ($(1)!0.5!(2)$);
\coordinate (14) at ($(2)!0.5!(3)$);
\coordinate (15) at ($(3)!0.5!(4)$);
\coordinate (16) at ($(4)!0.5!(5)$);
\coordinate (17) at ($(5)!0.5!(6)$);
\coordinate (18) at ($(7)!0.5!(6)$);
\coordinate (19) at ($(8)!0.5!(7)$);
\coordinate (20) at ($(9)!0.5!(8)$);
\coordinate (21) at ($(9)!0.5!(10)$);
\coordinate (22) at ($(10)!0.5!(1)$);
\coordinate (23) at ($(2)!0.5!(11)$);
\coordinate (24) at ($(11)!0.5!(10)$);
\coordinate (25) at ($(3)!0.5!(12)$);
\coordinate (26) at ($(12)!0.5!(11)$);
\coordinate (27) at ($(12)!0.5!(5)$);
\coordinate (28) at ($(7)!0.5!(12)$);
\coordinate (29) at ($(8)!0.5!(11)$);
\coordinate (30) at ($(13)!0.5!(24)$);
\coordinate (31) at ($(14)!0.5!(26)$);
\coordinate (32) at ($(27)!0.5!(15)$);
\coordinate (33) at ($(27)!0.5!(18)$);
\coordinate (34) at ($(19)!0.5!(26)$);
\coordinate (35) at ($(20)!0.5!(24)$);
\coordinate (36) at ($(21)!0.5!(35)$);
\coordinate (37) at ($(20)!0.5!(35)$);
\coordinate (38) at ($(9)!0.5!(20)$);
\coordinate (39) at ($(9)!0.5!(21)$);
\coordinate (40) at ($(37)!0.5!(39)$);
\coordinate (41) at ($(1)!0.5!(13)$);
\coordinate (42) at ($(30)!0.5!(13)$);
\coordinate (43) at ($(30)!0.5!(22)$);
\coordinate (44) at ($(22)!0.5!(1)$);
\coordinate (45) at ($(44)!0.5!(42)$);
\coordinate (46) at ($(14)!0.5!(3)$);
\coordinate (47) at ($(3)!0.5!(25)$);
\coordinate (48) at ($(31)!0.5!(25)$);
\coordinate (49) at ($(31)!0.5!(14)$);
\coordinate (50) at ($(47)!0.5!(49)$);
\coordinate (51) at ($(15)!0.5!(4)$);
\coordinate (52) at ($(16)!0.5!(4)$);
\coordinate (53) at ($(16)!0.5!(32)$);
\coordinate (54) at ($(15)!0.5!(32)$);
\coordinate (55) at ($(52)!0.5!(54)$);
\coordinate (56) at ($(33)!0.5!(17)$);
\coordinate (57) at ($(17)!0.5!(6)$);
\coordinate (58) at ($(18)!0.5!(6)$);
\coordinate (59) at ($(18)!0.5!(33)$);
\coordinate (60) at ($(57)!0.5!(59)$);
\draw (1)--(4)--(6)--(9)--(1);
\draw (21)--(17);
\draw (10)--(5);
\draw (22)--(16);
\draw (20)--(13);
\draw (8)--(2);
\draw (19)--(14); 
\draw (7)--(3);
\draw (18)--(15);
\draw(38)--(36);
\draw (39)--(37);
\draw (43)--(41);
\draw (44)--(42);
\draw (49)--(47);
\draw (48)--(46);
\draw (54)--(52);
\draw (51)--(53);
\draw (59)--(57);
\draw (58)--(56);
\node at (40) [below right=-1mm] {\tiny $\ast$};
\node at (40) [below left=-1mm] {\tiny$\ast$};
\node at (40) [above right=-1mm] {\tiny$\ast$};
\node at (40) [above left=-1mm] {\tiny$\ast$};
\node at (45) [below right=-1mm] {\tiny $\ast$};
\node at (45) [below left=-1mm] {\tiny$\ast$};
\node at (45) [above right=-1mm] {\tiny$\ast$};
\node at (45) [above left=-1mm] {\tiny$\ast$};
\node at (50) [below right=-1mm] {\tiny $\ast$};
\node at (50) [below left=-1mm] {\tiny$\ast$};
\node at (50) [above right=-1mm] {\tiny$\ast$};
\node at (50) [above left=-1mm] {\tiny$\ast$};
\node at (55) [below right=-1mm] {\tiny $\ast$};
\node at (55) [below left=-1mm] {\tiny$\ast$};
\node at (55) [above right=-1mm] {\tiny$\ast$};
\node at (55) [above left=-1mm] {\tiny$\ast$};
\node at (60) [below right=-1mm] {\tiny $\ast$};
\node at (60) [below left=-1mm] {\tiny$\ast$};
\node at (60) [above right=-1mm] {\tiny$\ast$};
\node at (60) [above left=-1mm] {\tiny$\ast$};
\node at (34) [below right=-1mm] {\textcolor{red}{$\ast$}};
\node at (34) [below left=-1mm] {\textcolor{red}{$\ast$}};
\node at (34) [above right=-1mm] {\textcolor{red}{$\ast$}};
\node at (34) [above left=-1mm] {\textcolor{red}{$\ast$}};
\node at (38) [above] {\scriptsize 2};
\node at (43) [above] {\scriptsize 2};
\node at (48) [above] {\scriptsize 2};
\node at (53) [above] {\scriptsize 2};
\node at (58) [above] {\scriptsize 2};
\node at (19) [above] {\scriptsize 2};
\node at (36) [below] {\scriptsize 2};
\node at (41) [below] {\scriptsize 2};
\node at (56) [below] {\scriptsize 2};
\node at (51) [below] {\scriptsize 2};
\node at (46) [below] {\scriptsize 2};
\node at (26) [below left] {\scriptsize 1};
\node at (39) [left] {\scriptsize 2};
\node at (44) [left] {\scriptsize 2};
\node at (49) [left] {\scriptsize 2};
\node at (54) [left] {\scriptsize 2};
\node at (59) [left] {\scriptsize 2};
\node at (29) [below left] {\scriptsize 1};
\node at (37) [right] {\scriptsize 2};
\node at (42) [right] {\scriptsize 2};
\node at (47) [right] {\scriptsize 2};
\node at (52) [right] {\scriptsize 2};
\node at (57) [right] {\scriptsize 2};
\node at (28) [below right] {\scriptsize 1};
\fill (39) circle (2pt);
\fill (38) circle (2pt);
\fill (36) circle (2pt);
\fill (37) circle (2pt);
\fill (44) circle (2pt);
\fill (41) circle (2pt);
\fill (42) circle (2pt);
\fill (43) circle (2pt);
\fill (46) circle (2pt);
\fill (49) circle (2pt);
\fill (47) circle (2pt);
\fill (48) circle (2pt);
\fill (51) circle (2pt);
\fill (52) circle (2pt);
\fill (54) circle (2pt);
\fill (53) circle (2pt);
\fill (57) circle (2pt);
\fill (58) circle (2pt);
\fill (59) circle (2pt);
\fill (56) circle (2pt);
\fill (19) circle (2pt);
\fill (29) circle (2pt);
\fill (26) circle (2pt);
\fill (28) circle (2pt);
\end{tikzpicture}
\end{minipage}\hspace*{15mm}
\begin{minipage}{0.32\textwidth}
\begin{tikzpicture}
\coordinate (1) at (0,0);
\coordinate (2) at (1.5,0);
\coordinate (3) at (3,0);
\coordinate (4) at (4.5,0);
\coordinate (5) at (4.5,1.5);
\coordinate (6) at (4.5,3);
\coordinate (7) at (3,3);
\coordinate (8) at (1.5,3);
\coordinate (9) at (0,3);
\coordinate (10) at (0,1.5);
\coordinate (11) at (1.5,1.5);
\coordinate (12) at (3,1.5);
\coordinate (13) at ($(1)!0.5!(2)$);
\coordinate (14) at ($(2)!0.5!(3)$);
\coordinate (15) at ($(3)!0.5!(4)$);
\coordinate (16) at ($(4)!0.5!(5)$);
\coordinate (17) at ($(5)!0.5!(6)$);
\coordinate (18) at ($(7)!0.5!(6)$);
\coordinate (19) at ($(8)!0.5!(7)$);
\coordinate (20) at ($(9)!0.5!(8)$);
\coordinate (21) at ($(9)!0.5!(10)$);
\coordinate (22) at ($(10)!0.5!(1)$);
\coordinate (23) at ($(2)!0.5!(11)$);
\coordinate (24) at ($(11)!0.5!(10)$);
\coordinate (25) at ($(3)!0.5!(12)$);
\coordinate (26) at ($(12)!0.5!(11)$);
\coordinate (27) at ($(12)!0.5!(5)$);
\coordinate (28) at ($(7)!0.5!(12)$);
\coordinate (29) at ($(8)!0.5!(11)$);
\coordinate (30) at ($(13)!0.5!(24)$);
\coordinate (31) at ($(14)!0.5!(26)$);
\coordinate (32) at ($(27)!0.5!(15)$);
\coordinate (33) at ($(27)!0.5!(18)$);
\coordinate (34) at ($(19)!0.5!(26)$);
\coordinate (35) at ($(20)!0.5!(24)$);
\coordinate (36) at ($(21)!0.5!(35)$);
\coordinate (37) at ($(20)!0.5!(35)$);
\coordinate (38) at ($(9)!0.5!(20)$);
\coordinate (39) at ($(9)!0.5!(21)$);
\coordinate (40) at ($(37)!0.5!(39)$);
\coordinate (41) at ($(1)!0.5!(13)$);
\coordinate (42) at ($(30)!0.5!(13)$);
\coordinate (43) at ($(30)!0.5!(22)$);
\coordinate (44) at ($(22)!0.5!(1)$);
\coordinate (45) at ($(44)!0.5!(42)$);
\coordinate (46) at ($(14)!0.5!(3)$);
\coordinate (47) at ($(3)!0.5!(25)$);
\coordinate (48) at ($(31)!0.5!(25)$);
\coordinate (49) at ($(31)!0.5!(14)$);
\coordinate (50) at ($(47)!0.5!(49)$);
\coordinate (51) at ($(15)!0.5!(4)$);
\coordinate (52) at ($(16)!0.5!(4)$);
\coordinate (53) at ($(16)!0.5!(32)$);
\coordinate (54) at ($(15)!0.5!(32)$);
\coordinate (55) at ($(52)!0.5!(54)$);
\coordinate (56) at ($(33)!0.5!(17)$);
\coordinate (57) at ($(17)!0.5!(6)$);
\coordinate (58) at ($(18)!0.5!(6)$);
\coordinate (59) at ($(18)!0.5!(33)$);
\coordinate (60) at ($(57)!0.5!(59)$);
\draw (1)--(4)--(6)--(9)--(1);
\draw (21)--(17);
\draw (10)--(5);
\draw (22)--(16);
\draw (20)--(13);
\draw (8)--(2);
\draw (19)--(14); 
\draw (7)--(3);
\draw (18)--(15);
\draw(38)--(36);
\draw (39)--(37);
\draw (43)--(41);
\draw (44)--(42);
\draw (49)--(47);
\draw (48)--(46);
\draw (54)--(52);
\draw (51)--(53);
\node at (40) [below right=-1mm] {\tiny $\ast$};
\node at (40) [below left=-1mm] {\tiny$\ast$};
\node at (40) [above right=-1mm] {\tiny$\ast$};
\node at (40) [above left=-1mm] {\tiny$\ast$};
\node at (45) [below right=-1mm] {\tiny $\ast$};
\node at (45) [below left=-1mm] {\tiny$\ast$};
\node at (45) [above right=-1mm] {\tiny$\ast$};
\node at (45) [above left=-1mm] {\tiny$\ast$};
\node at (50) [below right=-1mm] {\tiny $\ast$};
\node at (50) [below left=-1mm] {\tiny$\ast$};
\node at (50) [above right=-1mm] {\tiny$\ast$};
\node at (50) [above left=-1mm] {\tiny$\ast$};
\node at (55) [below right=-1mm] {\tiny $\ast$};
\node at (55) [below left=-1mm] {\tiny$\ast$};
\node at (55) [above right=-1mm] {\tiny$\ast$};
\node at (55) [above left=-1mm] {\tiny$\ast$};
\node at (34) [below right=-1mm] {$\ast$};
\node at (34) [below left=-1mm] {$\ast$};
\node at (34) [above right=-1mm] {$\ast$};
\node at (34) [above left=-1mm] {$\ast$};
\node at (33) [below right=-1mm] {$\ast$};
\node at (33) [below left=-1mm] {$\ast$};
\node at (33) [above right=-1mm] {$\ast$};
\node at (33) [above left=-1mm] {$\ast$};
\node at (38) [above] {\scriptsize 2};
\node at (43) [above] {\scriptsize 2};
\node at (48) [above] {\scriptsize 2};
\node at (53) [above] {\scriptsize 2};
\node at (19) [above] {\scriptsize 2};
\node at (36) [below] {\scriptsize 2};
\node at (41) [below] {\scriptsize 2};
\node at (51) [below] {\scriptsize 2};
\node at (46) [below] {\scriptsize 2};
\node at (26) [below left] {\scriptsize 1};
\node at (39) [left] {\scriptsize 2};
\node at (44) [left] {\scriptsize 2};
\node at (49) [left] {\scriptsize 2};
\node at (54) [left] {\scriptsize 2};
\node at (29) [below left] {\scriptsize 1};
\node at (37) [right] {\scriptsize 2};
\node at (42) [right] {\scriptsize 2};
\node at (47) [right] {\scriptsize 2};
\node at (52) [right] {\scriptsize 2};
\node at (28) [below right] {\scriptsize 2};
\node at (18) [above] {\scriptsize 2};
\node at (17) [right] {\scriptsize 2};
\node at (27) [below left] {\scriptsize 1};
\fill (39) circle (2pt);
\fill (38) circle (2pt);
\fill (36) circle (2pt);
\fill (37) circle (2pt);
\fill (44) circle (2pt);
\fill (41) circle (2pt);
\fill (42) circle (2pt);
\fill (43) circle (2pt);
\fill (46) circle (2pt);
\fill (49) circle (2pt);
\fill (47) circle (2pt);
\fill (48) circle (2pt);
\fill (51) circle (2pt);
\fill (52) circle (2pt);
\fill (54) circle (2pt);
\fill (53) circle (2pt);
\fill (19) circle (2pt);
\fill (29) circle (2pt);
\fill (26) circle (2pt);
\fill (28) circle (2pt);
\fill (17) circle (2pt);
\fill (18) circle (2pt);
\fill (27) circle (2pt);
\end{tikzpicture}
\end{minipage}
\caption{Two exemplary meshes with the assigned weights $w_v$ for all 5-point stencil nodes $v$ contained in $\mathcal{A}$ (black dots).  As a 5-node stencil, we consider the 5 newest nodes of an admissible quartet of elements for coarsening. Hanging nodes, boundary nodes and nodes shared by two stencils of an admissible quartet of elements receive a weight of two. Middle nodes receive a weight of zero which is omitted in the presentation. Every other node of a stencil gets the weight one. If the sum of weights per stencil is greater than 5 coarsening is performed (black asterisks), if it is less than or equal to 5 the quartet of elements containing this stencil is blocked for coarsening (red asterisks). This criterion applied in a loop ensures that the 3-Neighbor Rule is maintained.}
\label{fig:numbering}
\end{figure}
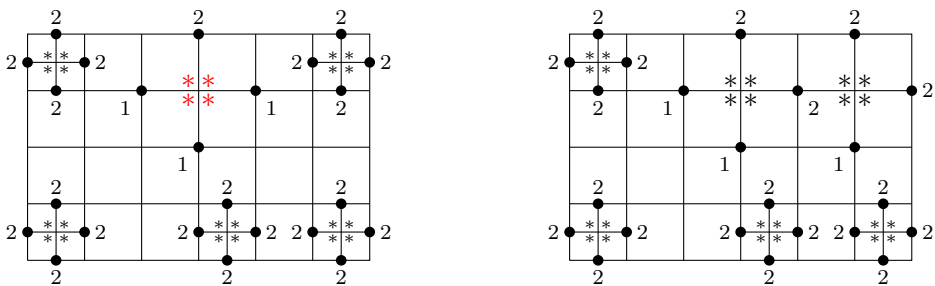

For triangular meshes, the situation is analogous. Here, an admissible quartet of elements is also eliminated from the set of admissible quartets $\mathcal{A}$ as soon as one of the elements in an admissible quartet has an irregular edge. Note, that an admissible quartet consists of one middle element and its three neighboring elements. The middle element cannot have an irregular edge if it is in an admissible quartet. Thus, only the other three elements sharing their edges with neighboring elements may be affected. This ensures the $1$-irregularity.

For triangular irregular meshes, the 2-Neighbor Rule is not used. However, in case of a red-green refinement a 2-Neighbor Rule is needed. 

In this case, the same ideas can be used with small modifications. Here, the nodes of the middle elements of admissible quartets are considered. The assignment of weights $w_v$ for all middle element nodes $v$ is performed analogously to the quadrilateral case, except that there are no middle nodes. As a further change, the sum must be less or equal than four instead of five. Algorithm~\ref{alg:Tclos} explains the step $\clos$ for triangular meshes in more detail. 

\begin{algorithm}[h!]
\caption{ $\clos$ (Triangular)}\label{alg:Tclos}
\begin{algorithmic}[1]
\State \textbf{Input:} Triangular mesh $\mathcal{T}$ obtained by red-refinement and admissible set $\mathcal{A}$ of quartets $Q$ for coarsening.
\State \textbf{Output:} Updated admissible set $\mathcal{A}$ of quartets $Q$ for coarsening.
\Procedure{Closure}{$\mathcal{T},\mathcal{A}$}
\ForAll{$Q \in \mathcal{A}$} \Comment{Ensure $1$-irregularity.}
\If{for any $T \in Q$ holds that $T$ has an irregular edge}
\State $\mathcal{A} \gets \mathcal{A}\setminus \left\{Q\right\}$
\EndIf
\EndFor
\If{2-Neighbor Rule is in use}  \Comment{Optional: Ensure 2-Neighbor Rule.}
\While{first run or $\mathcal{A}$ changes} 
\State $\hat{\mathcal{A}}\coloneqq \mathcal{A}$
\ForAll{$Q \in \mathcal{A}$} 
\State $\mathcal{N}_{M} \gets \left\{ v \in \mathcal{N}(T)~|~T \text{ is middle element in } Q\right\}$
\ForAll{$v \in \mathcal{N}_{M}$} 
\If{$v$ is a hanging node, a boundary node or is shared by two middle elements of \hspace*{3.2cm} admissible quartets $Q$ contained in $\mathcal{A}$ }
\State $w_v \gets 2$
\Else
\State $w_v \gets1$
\EndIf
\EndFor
\If{$\sum\limits_{v \in \mathcal{N}_M} w_v \leq 4$}
\State $\hat{\mathcal{A}} \gets \hat{\mathcal{A}}\setminus \left\{Q\right\}$
\EndIf
\EndFor
\State $\mathcal{A} \gets \hat{\mathcal{A}}$
\EndWhile
\EndIf
\State \textbf{return} ${\mathcal{A}}$
\EndProcedure
\end{algorithmic}
\end{algorithm}

With all these easy-to-verify criteria, we have now narrowed down the admissible quartets for coarsening to include adaptivity and adherence of rules. These quartets are then finally updated in the next step.

\subsubsection{$\upd$}

The step $\upd$ is the last step performed within the coarsening algorithm. The elements to be coarsened are specified in quartets in $\mathcal{A}$ and only need to be joined to their parent element, see Figure~\ref{fig:father}. To ensure that the vertices are arranged counterclockwise within the parent element, the elements in an admissible quartet must be sorted accordingly. After sorting, the quartet can be coarsened to their parent element. All other elements are not changed. Coarsening may create new irregular edges and remove old ones. Therefore the step $\upd$ also includes the update of the irregularity data. Algorithm~\ref{alg:upd} outlines this last step.

\begin{algorithm}[h!]
\caption{ $\upd$ }\label{alg:upd}
\begin{algorithmic}[1]
\State \textbf{Input:} Mesh $\mathcal{T}$ obtained by red-refinement and admissible set $\mathcal{A}$ of quartets $Q$ for coarsening.
\State \textbf{Output:} Coarsened mesh $\hat{\mathcal{T}}$.
\Procedure{Update Mesh}{$\mathcal{T},\mathcal{A}$}
\ForAll{$Q \in \mathcal{A}$}
\State Coarsen $T \in Q$ to their parent element according to Figure~\ref{fig:father}.
\EndFor
\State Update irregularity data.
\State \textbf{return} $\hat{\mathcal{T}}$
\EndProcedure
\end{algorithmic}
\end{algorithm}

\medskip

We have now seen the basic concept of the coarsening algorithm and its single components for red meshes. In the further course of the work, the data structures and the exact implementation of the presented points will be dealt with much more concretely. We move their discussion to Section~\ref{sect:implement}. For now, we focus on some important results for \textsc{CoarsenR} from the algorithmic framework in Algorithm~\ref{alg:algo1}.

\subsubsection{Properties of \textsc{CoarsenR}}

We would now like to discuss important properties of the meshes that the \textsc{CoarsenR} algorithm outputs. It is irrelevant whether we are talking about triangular or quadrilateral meshes: the results apply equally to both shapes. We want to note that Algorithm~\ref{alg:algo1} in the specific case of \textsc{QcoarsenR} calls Algorithm~\ref{alg:Qadm},\ref{alg:mark},\ref{alg:Qclos} and \ref{alg:upd} and \textsc{TcoarsenR} calls Algorithm~\ref{alg:Tadm},\ref{alg:mark},\ref{alg:Tclos} and \ref{alg:upd}.

We have constructed the algorithms in such a way that it retains all the properties that were obtained during the refinement. By construction, it holds

\begin{theorem}[Output \textsc{CoarsenR}]\label{th:1}
Let $\mathcal{T}$ be a $1$-irregular triangulation obtained by red refinement of an initial triangulation $\mathcal{T}_0$ and $\mathcal{T}_\mathrm{mark} \subset \mathcal{T}$. Then \textsc{CoarsenR}$(\mathcal{T},\mathcal{T}_\mathrm{mark})$ from Algorithm~\ref{alg:algo1} generates a $1$-irregular and shape regular triangulation. Further, the $d$-Neighbor Rule holds with $d=3$ for quadrilaterals and, if the 2-Neighbor Rule is in use for triangular meshes, with $d$=2 for triangles.
\end{theorem}

We also want our algorithm to be able to recover the initial triangulation. In our case this can be done without any further conditions to the triangulation.

\begin{theorem}[Coarsening]\label{th:coarsen}
Let $\mathcal{T}$ be an arbitrary red refinement of an initial triangulation $\mathcal{T}_0$. Let $(\mathcal{T}^{(i)})_{i=0,1,\ldots}$ be a sequence of triangulations obtained by Algorithm~\ref{alg:algo1}, i.e.,
\[ \mathcal{T}^{(0)}\coloneqq \mathcal{T} \quad\text{ and }\quad \mathcal{T}^{(i+1)} = \textsc{CoarsenR}(\mathcal{T}^{(i)},\mathcal{T}^{(i)}).\]
Then after a finite number of steps $M \in \mathbb{N}_0$, we obtain 
\[ \mathcal{T}^{(M)} = \mathcal{T}_0.\]
\end{theorem}

For this to be true, the input set $\mathcal{A}$ in $\upd$ has to be non-empty for an arbitrary red triangulation $\mathcal{T}\neq\mathcal{T}_0$. This way, $\upd$ always coarsens elements to their parent element and after a finite number of coarsening steps, the initial mesh is regained. That coarsening cannot extend beyond the initial mesh is given in the functions $\adm$. For quadrilaterals, the condition $v_{i3} \not\in \mathcal{N}_0$ in Algorithm~\ref{alg:Qadm} ensures that only middle nodes are considered that were not already included in the initial grid. In the case of triangles, the function $\old$ in Algorithm~\ref{alg:Tadm} is responsible for not coarsening further. If the function $\old$ is applied to an initial grid, all nodes are of the same age. This means that no middle elements are detected that belong to the initial mesh, and therefore no further coarsening beyond the initial mesh is performed.
We thus prove

\begin{lemma}[Non-emptyness of $\mathcal{A}$]
Let $\mathcal{T}$ be an arbitrary $1$-irregular triangulation consisting of red patterns and $\mathcal{T}\not= \mathcal{T}_0$ applies. Then, the input set $\mathcal{A}$ of $\upd$ is non-empty when \textsc{CoarsenR}$(\mathcal{T},\mathcal{T})$ is called.
\end{lemma}

\begin{proof}
The set of admissible quartets $\mathcal{A}$ is first generated in $\adm$. During the procedure, this initial set $\mathcal{A}$ is only reduced but not enlarged. For this purpose, we first prove that the function $\adm$ outputs a non-empty set $\mathcal{A}$.

Due to the hierarchical structure of the mesh, there are always quartets of elements that have the same parent element. These elements correspond to the last level in a tree, cf.~Figure~\ref{fig:tree}. The only exception is the roots in the tree, but this case is excluded with the condition $\mathcal{T}\not= \mathcal{T}_0$. Thus, the output $\mathcal{A}$ of $\adm$ is non-empty. All elements of the triangulation are marked and therefore the step $\mar$ does not change the set $\mathcal{A}$. It remains to show that after the step $\clos$ there are still quartets in $\mathcal{A}$. 

To ensure the $1$-irregularity, quartets of elements are blocked, if they have an irregular edge. If an element in the quartet has an irregular edge, this means that there are elements at a finer level. If the quartet is at the finest level, it cannot be blocked by irregular edges. So there are always quartets that are not blocked because they are at the finest level. Thus, the set $\mathcal{A}$ remains non-empty.

 In addition, the $d$-Neighbor Rule blocks a quartet from being coarsened if neighboring elements on the same level are blocked. This blocking does not result from adaptive marking, because all elements are marked for coarsening. If such a blocking is done, there must therefore be elements on a finer level. Otherwise, if this quartet is at the finest level, it cannot be blocked. We conclude $\mathcal{A}\neq \emptyset$.
\end{proof}

\begin{bem}
The hierarchical structure of the mesh even indicates that the number of coarsening steps is given by the level of refinement, see also numerical experiments in Section~\ref{sect:experiments}.
\end{bem}

\subsection{Red-Green and Red-Blue Refinements}

We now discuss Algorithm~\ref{alg:algo2} for red-green and red-blue meshes. This algorithm calls the presented Algorithm \ref{alg:algo1}. To this end, we only focus on the other two steps $\rec$ and $\reg$. 

\subsubsection{$\rec$}

The $\rec$ operation can be easily explained using Figure~\ref{fig:regularize_green}. There, a mesh is depicted with different green patterns. By recoarsen we mean to delete the green patterns to obtain a $1$-irregular mesh with only red patterns.

To coarsen a regular mesh obtained by red-green or red-blue refinement back to an irregular mesh consisting of red patterns, it is necessary to determine the green and blue patterns. In our data structure we have stored red, green and blue patterns in blocks. This means that first all red elements are stored and then either a block of green or blue elements. To distinguish these blocks from each other, the number of green or blue elements is tracked in a global variable. This information helps to distinguish all relevant green or blue patterns to be deleted from the red ones. The next step is to determine the relationship between the elements. In the case of red-green for triangular meshes, there is only one green pattern that can exist in three different orientations, cf.~left pattern in Figure~\ref{fig:green}. A green pattern consists of two elements that are stored consecutively. To define their parent element, their common newest node is deleted and the remaining nodes are stored counterclockwise.

For quadrilateral meshes, the red-green strategy has three different green patterns that can be oriented in each direction, cf.~second left to right in Figure~\ref{fig:green}. One of the three patterns (right) consists of quadrilaterals and is thus stored in a block after all red quadrilaterals. The number of green quadrilaterals is tracked in a global variable. As in the triangular case, two consecutive elements have the same parent element. The definition of the parent element is easily made by deletion of the newest nodes and storage of the old ones in a mathematical positive sense. In addition, triangular patterns are used to regularize the mesh. These are stored in an additional variable and are thus already classified as green patterns. There are two types of green patterns consisting of triangles. The different types of green patterns are also stored in blocks. In this case, the number of elements belonging to a type is not stored and must be determined via characteristics. This is done using the node numbering within a green pattern. The number determined in this way tells us how many green patterns of which type exist. We also know how many elements per type have the same parent element. These are stored one after the other. With this information and the determination of the newest nodes, the parent element can be redefined.

For red-blue meshes, the number of blue patterns is tracked and can therefore be eliminated analogously. 

The $\rec$ step involves not only coarsening green and blue patterns to their corresponding parent elements, but also defining irregular edges resulting from this process and updating the marked elements $\mathcal{T}_\mathrm{mark}$. In the case of green patterns, deleting the newest node within a green pattern results in an irregular edge with this newest node as a hanging node. In the case of blue patterns, a hanging node is usually also introduced. However, there are cases where they do not become a hanging node. For example, if a blue pattern shares its newest node with another blue pattern, this newest node is completely eliminated from the mesh and does not become a hanging node. Further, the markings of green or blue elements in $\mathcal{T}_\mathrm{mark}$ are eliminated and the updated set $\mathcal{T}_\mathrm{mark}^R$ is output.

Overall, we have received a $1$-irregular mesh $\mathcal{T}^R$ out of the red-green or red-blue mesh $\mathcal{T}$ with an updated set of marked elements $\mathcal{T}_\mathrm{mark}^R$ that can now be processed via Algorithm~\ref{alg:algo1}.

\subsubsection{$\reg$}

The output of Algorithm~\ref{alg:algo1} is a coarse $1$-irregular mesh to which the $d$-Neighbor Rule applies with $d=3$ for quadrilaterals and $d=2$ for triangles. This ensures, that the specified green patterns have the shapes required to regularize the mesh. In other words, the patterns match the hanging nodes. Regularization is the elimination of hanging nodes by adding more elements to the mesh. This is the reverse operation of $\rec$. So we make sure to store the elements in blocks and track the number of green and blue patterns accordingly. This helps to coarsen the mesh again in a subsequent coarsening step to $\rec$. For the red-blue strategy a further note has to be made. Often, closing one hanging node by a blue pattern introduces another hanging node. This means that the blue patterns do not readily match the hanging nodes. To this end, the step $\reg$ for red-blue includes a further closure step that makes sure to remove all hanging nodes, compare the red-blue refinement strategy in \cite{funkenschmidt}.

\subsubsection{Properties of \textsc{CoarsenRg} and \textsc{CoarsenRb}}

For red-green and red-blue strategies, the shape regularity is fulfilled. Again, only elements that previously belonged together are merged. In addition, coarsening cannot result in hanging nodes, since $\rec$ creates a $1$-irregular grid in compliance with the $d$-Neighbor Rule that can then be coarsened by \textsc{CoarsenR}. From Theorem~\ref{th:1} it is known that \textsc{CoarsenR} outputs a $1$-irregular mesh in accordance with the $d$-Neighbor Rule. This ensures that no other green patterns then the predefined are needed. In $\reg$, all hanging nodes can be eliminated by adding these matching green patterns. For red-blue strategies, $\reg$ guarantees the regularity of the mesh by a further closure step, which is not specified here. In total, the following applies:

\begin{theorem}[Output \textsc{CoarsenRg}/\textsc{CoarsenRb}]\label{th:2}
Let $\mathcal{T}$ be a conforming triangulation obtained by red-green or red-blue refinement of an initial triangulation $\mathcal{T}_0$ and $\mathcal{T}_\mathrm{mark} \subset \mathcal{T}$. Then \textsc{CoarsenRg/} \textsc{CoarsenRb}$ (\mathcal{T},\mathcal{T}_\mathrm{mark})$ from Algorithm~\ref{alg:algo2} generates a conforming and shape regular triangulation.
\end{theorem}

With the same arguments, Theorem~\ref{th:coarsen} applies for a triangulation $\mathcal{T}$ obtained by red-green or red-blue refinement and successive applications of \textsc{CoarsenRG} and \textsc{CoarsenRB} as described in Algorithm~\ref{alg:algo2}. This is apparent as we operate on the $1$-irregular grid and the green respectively blue refinement is only an un-/closing operation.

\subsection{Newest Vertex Bisection and Red-Green-Blue Refinements}

Other works already discuss coarsening algorithms for the newest vertex bisection \cite{chenzhang,p1afem} and the red-green-blue refinement \cite{RGB}. At this point we only give a brief insight into the ideas.  We would like to draw particular attention to the difficulties encountered in these refinement strategies in terms of coarsening. Let us remember that we want to use the same algorithmic framework as for red refinements, namely Algorithm~\ref{alg:algo1}. A red refinement quarters an element. A green refinement or a bisection halves an element. This means that we can imagine the refinement again analogously with a tree -- here even a binary tree. So we have to find a criterion with which we can go back in the binary tree. This was shown in the work of Chen and Zhang \cite{chenzhang}. Now we have seen in Figure~\ref{fig:NVBRGB} that the red-green-blue refinement differs only by the last pattern. This is where a red refinement is applied instead of bisec(3). This makes it possible to have a different tree structure. Some elements are halved and others are quartered. This presents certain challenges, which are described in detail in the work of Funken and Schmidt \cite{RGB}. We introduce the ideas of the individual functions briefly, but refer to \cite{chenzhang,RGB} for a detailed description.

\subsubsection{$\adm$}
In contrast to the red refinement, we do not determine quartets of elements for coarsening. Instead, we find admissible-to-coarsening nodes via the determination of the cardinality of adjacent elements to the newest nodes.
The newest node of an element can be eliminated if there are two or four adjacent elements to this node. Otherwise, the elimination is blocked and can be resolved in a preceding step. 

A similar determination of admissible-to-coarsening nodes can be made in the red-green-blue refinement. Here, the red pattern plays a different role. To this end, we treat the middle element of a red pattern differently and count the adjacent elements to a newest node excluding these middle elements. The same criterion as for {\TNVB} can then be used to classify the newest nodes, i.e., for the quantity two or four adjacent elements excluding middle elements, the node can be eliminated. We refer to \cite{chenzhang, RGB} for a more thorough presentation.

\subsubsection{$\mar$}

The marking operation marks all nodes of a marked element.

\subsubsection{$\clos$}

The newest vertex bisection does not need a $\clos$ step as the admissible-to-coarsen nodes are already selected so that no further processing steps are necessary.

The red pattern in the red-green-blue refinement however plays a special role and thus an additional $\clos$ step is needed to ensure the shape regularity. This includes to take care of the assigned reference edges. The idea here is to go back as if the mesh was refined by the newest vertex bisection. This still guarantees the shape regularity but also ensures that coarsening is not blocked, see \cite{RGB}.

\subsubsection{$\upd$}
After determining the relevant nodes, the adjacent elements to these nodes are returned to their parent element by coarsening green or red patterns. The return of the green pattern is analogous to the return of a bisection. Blue patterns are eliminated via a two-step coarsening of green patterns but are not considered separately. Red patterns are returned to their parent element or a green or blue intermediate pattern, see \cite{RGB}.

We want to note that Theorem~\ref{th:2} holds analogously for {\TNVB} and {\TRGB}. A similar result as in Theorem~\ref{th:coarsen} can also be obtained, with restrictions to the assignment of reference edges in the initial triangulation $\mathcal{T}_0$. A more thorough discussion of these results can be found in \cite{chenzhang,RGB}.

\section{Implementation and Overview of Toolbox}\label{sect:implement}

In this section, we want to focus on how we can implement the ideas in \textsc{Matlab}. In particular, we have seen that coarsening red refined meshes plays a central role in our work. Coarsening green and blue patterns is an easy task that we will not go into further in this section. The implementation of {\TNVB} has already been discussed in detail in \cite{p1afem} and of {\TRGB} in \cite{RGB}. We therefore omit their presentation here and focus exclusively on the implementation of red coarsening in \textsc{Matlab} according to Algorithm~\ref{alg:algo1}.

\subsection{\textsc{MATLAB} Implementation: Coarsening of Red Refinements}

An implementation of the coarsening is based on the presented criteria and the utilization of data structures. We will first discuss the implementation of coarsening of adaptively refined quadrilateral meshes using the red refinement. Since the implementation for triangles differs only slightly, we present the differences in the step $\adm$ and otherwise refer to the codes provided in \cite{ameshcoars}. The reader will see that the ideas can be translated one-to-one with minor adjustments.

\subsubsection{QcoarsenR}

To get a better overview of the individual blocks, we will structure this section analogous to the individual steps $\adm$, $\mar$, $\clos$ and $\upd$ and an additional preparation step. 

\paragraph{$\adm$}

In this part, we determine admissible quartets of elements as described in Algorithm~\ref{alg:Qadm}.
\begin{itemize}
\item Lines 1--5: The function is usually called by \\
 \matlab{[coordinates,elements4,irregular,boundary] ...}\\
 \indent\matlab{ = QcoarsenR(N0,coordinates,elements4,irregular,boundary,marked)}. \\
 The variables \matlab{coordinates}, \matlab{elements4}, \matlab{irregular} and \matlab{boundary} are clear from the definition in Section~\ref{sect:datastructure}. Note, that the input \matlab{elements4} is already sorted as required in Algorithm~\ref{alg:Qadm}, i.e., the smallest index is at position one. \matlab{N0} is the number of nodes in the initial triangulation $\mathcal{T}_0$. Boundary data is an optional argument. \matlab{marked} determines the elements for coarsening and is considered to be the last entry of \matlab{varargin} (Line 5). The number of nodes \matlab{nC} and element \matlab{nE} is determined.
 \item Lines 6--13: Here, admissible middle nodes are determined. Due to the sorting within an element, the possible middle nodes are the ones at position three. To this end, we sort the elements such that all elements with the same third node are listed after each other. If there are four in a row, they are considered admissible middle nodes, see Figure~\ref{fig:middlenode}.
 \item Lines 14--18: Based on the admissible middle nodes, we determine the four elements adjacent to this node. To this end, we define a variable with the indices of the coordinates and change the sign for admissible middle nodes. Applied to the elements, we can determine for which elements such a negative index is present. Then, \matlab{node2elem} represents the admissible quartets of elements. 
\end{itemize}

\begin{lstlisting}[xleftmargin=0.8cm,language=matlab,frame=tb,linerange ={1-18},label=lst:QcoarsenR1,caption= {The step $\adm$ in QcoarsenR.m}]
function [coordinates,elements,irregular,varargout] = ...
    QcoarsenR(N0,coordinates,elements,irregular,varargin)
nC = size(coordinates,1);
nE = size(elements,1);
marked = varargin{end};
%*** get admissible middle nodes
temp = sortrows(elements,3);
adm = temp(temp(1:end-3,3) == temp(4:end,3),3);
adm = adm(adm>N0); % do not delete nodes of initial triangulation
if isempty(adm)
    varargout = varargin;
    return
end
%*** determine elements belonging to admissible middle nodes
i2fi = 1:nC; i2fi(adm) = -i2fi(adm);
idx = -min(i2fi(elements),[],2); kdx = find(idx > 0);
temp = sortrows([idx(kdx),kdx]);
node2elem = reshape(temp(:,2),4,[])';
\end{lstlisting}

\paragraph{$\mar$}

In this part, we reduce the admissible quartets of elements as described in Algorithm~\ref{alg:mark}.
\begin{itemize}
\item Lines 19--23: We mark elements for coarsening. If at least one element in a quartet is marked, it is considered admissible. The condition in Line 21 can be changed to \matlab{==4} if it is preferable to require that all elements of an quartet shall be marked that this quartet is considered for coarsening.
\end{itemize}

\begin{lstlisting}[xleftmargin=0.8cm,language=matlab,frame=tb,linerange ={19-23},firstnumber = 19,label=lst:QcoarsenR2,caption= {The step $\mar$ in QcoarsenR.m}]
%*** consider only marked elements
marked_elements = zeros(1,nE); marked_elements(marked) = 1;
idx = sum(marked_elements(node2elem),2)>1; % set >1 for coarsen if at least one ...
% element is marked; set == 4 for coarsen if all elements are marked
node2elem = node2elem(idx,:);
\end{lstlisting}

\paragraph{Preparation Step}

Until now, it was not necessary to use neighboring properties other than those already stored in the data structures. However, in order to implement the $\clos$ and the $\upd$ step, further geometric information is required, which is now generated. In addition, the existing data is being prepared for further use.

\begin{itemize}
\item Lines 24--32: The elements in a quartet are not sorted in \matlab{node2elem}. To ensure storing the coordinates of the parent element in counterclockwise order, we need to ensure that they are stored counterclockwise. To this end, a sorting is made until the four elements are sorted counterclockwise.
 \item Lines 33--35: With the help of the auxiliary function \matlab{provideGeometricData}, we generate an edge numbering. 
 \item Lines 36--39: For the $1$-irregularity, it is of importance if any of the elements in an admissible quartet has an irregular edge. To this end, in \matlab{adm2edges} the outer edges of an admissible quartet is stored.
 \item 40--41: We mark an edge if it has a hanging node.
\end{itemize}
\begin{lstlisting}[xleftmargin=0.8cm,language=matlab,frame=tb,linerange ={24-41},firstnumber = 24,label=lst:QcoarsenR3,caption= {Some preparation for subsequent steps in QcoarsenR.m}]
%*** adjust order
i1 = [3,4,4]; i2 = [2,2,3];
for k = 1:3
    idx = find(elements(node2elem(:,i1(k)),4) == ...
               elements(node2elem(:,i2(k)-1),2));
    tmp = node2elem(idx,i2(k));
    node2elem(idx,i2(k)) = node2elem(idx,i1(k));
    node2elem(idx,i1(k)) = tmp;
end
%*** obtain geometric information on edges
[edge2nodes,irregular2edges,element2edges,boundary2edges{1:nargin-5}] ...
    = provideGeometricData(irregular,elements,varargin{1:end-1});
adm2edges = [element2edges(node2elem(:,1),[4,1]), ...
    element2edges(node2elem(:,2),[4,1]), ...
    element2edges(node2elem(:,3),[4,1]), ...
    element2edges(node2elem(:,4),[4,1])];
markedEdges = zeros(size(edge2nodes,1),1);
markedEdges(irregular2edges(:,1)) = 1;
\end{lstlisting}

\paragraph{$\clos$}

In this part, we reduce the admissible quartets of elements as described in Algorithm~\ref{alg:Qclos}.

\begin{itemize}
\item Lines 43: We first determine boundary edges since this information is not always directly available. \matlab{boundary} is an optional argument. The key here is that irregular and inner edges occur twice, i.e., single occuring edges must lie on the boundary. This is important for implementing the assignment of node values specified in the criterion for the 3-Neighbor Rule.
\item Lines 45--49: If any of the elements in an admissible quartet has an irregular edge it is eliminated from the set of admissible middle nodes \matlab{adm} and admissible quartets \matlab{node2elem}, see Figure~\ref{fig:irregular}.
\item Lines 47--59: The values for the nodes in the five-point stencil are assigned. Nodes from boundary or irregular edges are counted twice. In Lines 53--54, it is determined whether this node is shared by two stencils of admissible quartets or not.  The value assignment is accordingly. See Figure~\ref{fig:numbering} for reference. Line 57 determines whether a quartet needs to be eliminated due to the 3-Neighbor Rule. The whole process is repeated until no further changes occur.
\end{itemize}

\begin{lstlisting}[xleftmargin=0.8cm,language=matlab,frame=tb,linerange ={42-59},firstnumber = 42,label=lst:QcoarsenR4,caption= {The step $\clos$ in QcoarsenR.m}]
%*** follow 3-Neighbor Rule
boundinfo = edge2nodes(accumarray([element2edges(:); irregular2edges(:)],...
                1,[length(edge2nodes) 1])==1,:);
idx = find(sum(reshape(markedEdges(adm2edges),[],8),2)); % follow 1-irregularity
foo = 1;
while ~isempty(idx) || foo
    adm(idx)=[];
    node2elem(idx,:)=[];
    %*** mark nodes to be eliminated
    markedNodes = zeros(nC,1);
    markedNodes(adm) = 1;
    markedNodes = markedNodes + accumarray(elements(node2elem,2),... 
                        ones(size(elements(node2elem,2))),[nC 1]);
    markedNodes(boundinfo(:,2)) = 2*markedNodes(boundinfo(:,2)); % boundary edges 
    markedNodes(irregular(:,3)) = 2*markedNodes(irregular(:,3)); % irregular edges
    idx = find(sum(reshape(markedNodes(elements(node2elem,4)),[],4),2)<6);
    foo = 0;
end
\end{lstlisting}

\paragraph{$\upd$}

In this part, we update the mesh as described in Algorithm~\ref{alg:upd}.
\begin{itemize}
\item Lines 60--62: Elements that are not in any of the admissible quartets are not changed. The parent elements of the quartets are appended to the array \matlab{elements}. The parent element is defined by the four oldest nodes in the quartet. These are in the first position in the respective child elements. The previous sorting guarantees that the nodes are mathematically positively defined within this element.
\item Lines 63--77: Irregularity data is updated, too. To this end, we store all irregular edges and additionally all edges that we eliminated. If any edges appear twice, they are deleted (Lines 70--71). If any of the remaining edges is a boundary edge, they are also deleted (Lines 72-77). What remains are the updated irregular edges. 
\item Lines 78--84: Due to the element updates, some nodes are no longer in the mesh. For this purpose the corresponding nodes are deleted in \matlab{coordinates}. The indexing to the coordinates in the arrays \matlab{elements} and \matlab{irregular} is updated accordingly.
\item 85--99: If additionally boundary data is provided, these edges are also updated.
\item Lines 100--106: Finally, we ensure that all elements are sorted such that the minimal index is at position one.
\end{itemize}

\begin{lstlisting}[xleftmargin=0.8cm,language=matlab,frame=tb,linerange ={60-107},firstnumber = 60,label=lst:QcoarsenR5,caption= {The step $\upd$ in QcoarsenR.m}]
%*** create new elements
idx = setdiff(1:nE,unique(node2elem));
newelements = [elements(idx,:);reshape(elements(node2elem(:),1),[],4)];
%*** update constrains, allow redundancies
irregular = [irregular; ...
[elements(node2elem(:,1),1),elements(node2elem(:,2),1),elements(node2elem(:,1),2)];...
[elements(node2elem(:,2),1),elements(node2elem(:,3),1),elements(node2elem(:,2),2)];...
[elements(node2elem(:,3),1),elements(node2elem(:,4),1),elements(node2elem(:,3),2)];...
[elements(node2elem(:,4),1),elements(node2elem(:,1),1),elements(node2elem(:,4),2)]];
irregular = sortrows(irregular,3);
idx = find(irregular(1:end-1,3) == irregular(2:end,3));
irregular([idx(:);idx(:)+1],:) = [];
node2bound = zeros(nC,1);
node2bound(unique(boundinfo(:)),:) = 1;
idx = (markedNodes & node2bound);
node2constr = zeros(nC,1);
node2constr(irregular(:,3),1) = 1:size(irregular,1);
irregular(node2constr(idx),:) = [];
%*** update coordinates and element numbering
i2fi = zeros(nC,1);
i2fi(newelements) = 1;
coordinates = coordinates(find(i2fi),:);
i2fi = cumsum(i2fi);
elements = reshape(i2fi(newelements),[],4);
irregular = reshape(i2fi(irregular),[],3);
%*** correct boundaries
for j = 1:nargout-3
    boundary = varargin{j};
    if ~isempty(boundary)
        node2boundary = zeros(nC,2);
        node2boundary(boundary(:,1),1) = 1:size(boundary,1);
        node2boundary(boundary(:,2),2) = 1:size(boundary,1);
        idx = ( markedNodes & node2boundary(:,2) );
        boundary(node2boundary(idx,2),2) = boundary(node2boundary(idx,1),2);
        boundary(node2boundary(idx,1),2) = 0;
        varargout{j} = i2fi(boundary(find(boundary(:,2)),:));
    else
        varargout{j} = [];
    end
end
%*** sort such that oldest node is first
[~,jdx] = min(elements,[],2);
map = [1,2,3,4; 2,3,4,1; 3,4,1,2; 4,1,2,3];
for k=1:4
    idx = jdx == k;
    elements(idx,:) = elements(idx,map(k,:));
end
end
\end{lstlisting}

\subsubsection{TcoarsenR}

For triangular meshes, the most ideas can be adopted. What distinguishes the triangular case from the quadrilateral is mainly the step $\adm$. All other steps are done with minor changes and are thus not addressed separately. To get to the $\adm$ step, some preparations are necessary which we discuss first.

\paragraph{Preparation Step}
In this part, we prepare the input arrays for our algorithm and generate further geometric information on the mesh.
\begin{itemize}
\item Lines 1--5: This is analogous to the quadrilateral case.
\item Lines 6--12: In contrast to our quadrilateral mesh, the triangular mesh is not stored in a way that the minimal index is stored at position one. To this end a preliminary sorting is performed.
\item Lines 13--18: With the auxiliary functions \matlab{provideGeometricData} and \matlab{createEdge2Elements} we generate edge information and their correspondence to elements.
\end{itemize}

\begin{lstlisting}[xleftmargin=0.8cm,language=matlab,frame=tb,linerange ={1-18},firstnumber = 1,label=lst:TcoarsenR1,caption= {Preparation step in TcoarsenR.m}]
function [coordinates,elements,irregular,varargout] = ...
    TcoarsenR(N0,coordinates,elements,irregular,varargin)
nC = size(coordinates,1);
nE = size(elements,1);
marked = varargin{end};
%*** sort such that oldest node is first
[~,jdx] = min(elements,[],2);
map = [1,2,3;2,3,1;3,1,2];
for k=1:3
    idx = jdx == k;
    elements(idx,:) = elements(idx,map(k,:));
end
%*** obtain geometric information
[edge2nodes,element3edges,~] = provideGeometricData([elements;...
    irregular],zeros(0,4),varargin{1:end-1});
irregular2edges = element3edges(nE+1:end,:);
element3edges = element3edges(1:nE,:);
edge2element = createEdge2Elements(element3edges);
\end{lstlisting}

\paragraph{$\adm$}

In this part, we determine admissible quartets of elements as described in Algorithm~\ref{alg:Tadm}, cf.~Figure~\ref{fig:middleelement}.

\begin{itemize}
\item Line 20: The oldest node in an element is the node with the minimal index. Due to the pre-sorting this node is stored at position one in \matlab{elements}.
\item Lines 21--22: We assign, as described in Algorithm~\ref{alg:Tadm} Line 6, the value $v_\mathrm{old,e}$ to each edge $e \in \mathcal{E}(\mathcal{T})$.
\item Lines 23--24: We check the criterion that all values assigned to the edges of an element must be equal. If this is not the case, we consider it to be a middle element.
\item Lines 25--26: We exclude the middle elements that contain a node from the initial triangulation $\mathcal{T}_0$ ensuring that coarsening does not occur beyond the initial triangulation.
\item Lines 27--28: For each edge of the middle elements the adjacent elements are determined and stored in \matlab{elem2coarse}
\item Lines 29--30: If any of the adjacent elements is $0$, i.e., there does not exist a direct neighbor, this middle element is not considered for an admissible quartet.
\item Lines 31--33: So far, \matlab{elem2coarse} has six entries, with the middle element listed three times. However, we would like each element to occur only once and we would like to have a quartet in the end. For this purpose we determine the entries that are not equal to the middle element. These are then stored in an updated \matlab{elem2coarse} with the middle elements \matlab{midElements}. This completes the determination of the admissible quartets of elements.
\item Lines 34--38: As in the quadrilateral case, we want to have the elements in such an order that we store the coordinates of the parent element counterclockwise. For this purpose, the outer three elements are stored counterclockwise and the middle element at the end.
\end{itemize}

\begin{lstlisting}[xleftmargin=0.8cm,language=matlab,frame=tb,linerange ={19-38},firstnumber = 19,label=lst:TcoarsenR2,caption= {The step $\adm$ in TcoarsenR.m}]
%*** determine elements to coarsen
minNode = elements(:,1);
value2edges = minNode.*ones(nE,3);
minEdgeVal = accumarray(element3edges(:),value2edges(:),[size(edge2nodes,1) 1],@min);
midElements = find(~((minEdgeVal(element3edges(:,1))==minEdgeVal(element3edges(:,2)...
  )) .* (minEdgeVal(element3edges(:,2)) == minEdgeVal(element3edges(:,3)))));
idx = sum(elements(midElements,:)>N0,2)<3; % do not delete nodes of initial triang.
midElements(idx) = [];
midEdges = element3edges(midElements,:);
elem2coarse = reshape(edge2element(midEdges(:),:),[],6);
idx = any(elem2coarse == 0,2);
elem2coarse(idx,:)= []; midElements(idx) = [];
elem2coarse_tmp = elem2coarse';
idx = (reshape(midElements',1,[]) .* ones(6,size(midElements,1)));
elem2coarse = [reshape(elem2coarse_tmp(elem2coarse_tmp ~= idx),3,[])', midElements];
%*** adjust order of elements
idx = find(elements(elem2coarse(:,1),2) == elements(elem2coarse(:,3),3));
tmp = elem2coarse(idx,3);
elem2coarse(idx,3) = elem2coarse(idx,2);
elem2coarse(idx,2) = tmp;
\end{lstlisting}

The steps $\mar$, $\clos$ and $\upd$ are equivalent to the ones presented for quadrilaterals with some minor changes. The full code is provided in \cite{ameshcoars}.

\subsection{Overview of Functions Provided by \texttt{ameshcoars}-Package}\label{sect:overview}

Our \texttt{ameshcoars} toolbox consists of several modules and is available on \cite{ameshcoars}. We list the functions and state the function call for each function. For the different input parameters, we use the abbreviations \matlab{C(coordinates)}, \matlab{E3(elements3)}, \matlab{E4(elements4)}, \matlab{I(irregular)}, and \matlab{b(boundary)}. Boundary data is an optional argument that can be left out. \matlab{N0} is the number of coordinates in the initial mesh $\mathcal{T}_0$, i.e., the initial mesh before any refinement. With this information, we make sure to not coarsen further than to the initial triangulation.

In this toolbox, we stick to the naming from the \texttt{ameshref}-package. For the coarsening routines, we simply replace the \emph{refine} by a \emph{coarsen}. The rest of the naming remains the same. As our toolbox is based on \texttt{ameshref} there is the folder \texttt{refinement/} that includes all the refinement and auxiliary functions used in the refinement. What is interesting for us are the functions in the folder \texttt{coarsening/} and their call.

\subsubsection{Local Mesh Coarsening}

The functions for adaptive mesh coarsening are called by

\begin{itemize}
\item \matlab{[C,E3,B] = TcoarsenNVB(N0,C,E3,B,marked)}
\item \matlab{[C,E3,B] = TcoarsenRGB(N0,C,E3,B,marked)}
\item \matlab{[C,E3,B] = TcoarsenRG(N0,C,E3,B,marked)}
\item \matlab{[C,E3,I,B] = TcoarsenR(N0,C,E3,I,B,marked)}
\item \matlab{[C,E4,I,B] = QcoarsenR(N0,C,E4,I,B,marked)}
\item \matlab{[C,E4,B] = QcoarsenRB(N0,C,E4,B,marked)}
\item \matlab{[C,E3,E4,B] = QcoarsenRG(N0,C,E3,E4,B,marked)}
\end{itemize}
We want to remark that the functions \matlab{TcoarsenRG}, \matlab{QcoarsenRB} and \matlab{QcoarsenRG} are called in a three-step-procedure calling the functions \matlab{TcoarsenR_2neighbor}, which is a function operating on a $1$-irregular triangular grid as \matlab{TcoarsenR} with the addition to stick to the 2-Neighbor Rule, and \matlab{QcoarsenR}.  Additional pre- and post-processing is done 
in
\begin{itemize}
\item \matlab{coarse_greenelements} and \matlab{regularize_Tgreen} for TcoarsenRG, \item \matlab{coarse_blueelements} and \matlab{regularize_Qblue} for QcoarsenRB and  \item\matlab{coarse_Qgreenelements} and \matlab{regularizeedges_tri} for QcoarsenRG. 
\end{itemize}

Additional auxiliary functions are needed to provide more information on the geometric data, see an explanation in Section~\ref{sect:datastructure}. These include the functions 

\begin{itemize}
\item \matlab{[edge2nodes,element3edges,element4edges,boundary2edges] =} \\
\indent  \matlab{provideGeometricData(E3,E4,B)} and 
\item \matlab{edge2elements = createEdge2Elements(element2edges)}.
\end{itemize} 
For the latter there is also an extended version \matlab{createEdge2Elements_adv}, which additionally provides the information about the numbering of the edges within this element. Thus, \matlab{createEdge2Elements} outputs the elements that share an edge, and \matlab{createEdge2Elements_adv} additionally outputs the information about which edge within the element it is.

\subsubsection{A Minimal Example}\label{sect:minimalexample}

Listing~\ref{lst:minimalexample} shows an exemplary code of how to embed the coarsening routine into a framework. We start with defining an initial mesh $\mathcal{T}_0$ (Lines 1--5). The number of coordinates in the initial mesh is stored in \matlab{N0}. A refined mesh $\tilde{\mathcal{T}}$ is created via \texttt{TrefineR} (Lines 8--16). For a given triangulation $\mathcal{T}$ and a given discrete point set $\mathcal{P}$, the function \texttt{point2element} determines the elements of $\mathcal{T}$ that include $p$ for some $p \in \mathcal{P}$. Thus, for the defined discrete point set in Lines 17--23, elements in $\tilde{\mathcal{T}}$ are marked according to \texttt{point2element}. We coarsen the mesh via the function call \texttt{TcoarsenR} (Line 27--28) until no further change is made (Line 29). Lines 34--36 plot the locally coarsened mesh. 

\begin{lstlisting}[xleftmargin=0.8cm,language=matlab,frame=tb,label=lst:minimalexample,caption= {A minimal example}]
%% Define initial mesh
coordinates = [0,0;1,0;1,1;0,1;2,0;2,1];
elements = [3,1,2;1,3,4;2,6,3;6,2,5];
irregular = zeros(0,3);
boundary = [1,2;2,5;5,6;6,3;3,4;4,1];
N0 = size(coordinates,1);
c_old = 0;
%% Refine uniformly
while 1
    marked = 1:size(elements,1);
    [coordinates,elements,irregular,boundary] ...
        = TrefineR(coordinates,elements,irregular,boundary,marked);
    if isempty(marked)|| (size(coordinates,1)>1e3)
        break
    end
end
% define discrete points (here a disc)
phi = -pi:pi/50:pi;
r = 0.2:1/50:0.4;
[r,phi] = meshgrid(r,phi);
s = r.*cos(phi);
t = r.*sin(phi);
points = [s(:)+1,t(:)+0.5];
%% Coarsen at discrete points
while 1
    mark3 = point2element(coordinates,elements,points);
    [coordinates,elements,irregular,boundary] = ...
                TcoarsenR(N0,coordinates,elements,irregular,boundary,mark3);
    if size(coordinates,1) == c_old
        break
    end
    c_old = size(coordinates,1);
end
% plot results
clf, patch('Faces',elements,'Vertices',coordinates,'Facecolor','none')
axis equal, axis off

\end{lstlisting}

This minimal example can be called in an adapted way for all implemented coarsening routines. This includes the definition of the mesh and the calls of the refinement and coarsening strategy. Especially, for red-green or red-blue strategies it is important to first \matlab{clear all} variables, as in these functions global variables \matlab{nG} or \matlab{nB} are defined. If they are not cleared before a new script is started, the global variable may already have a predefined value and an error may be displayed when the script is executed.

\subsubsection{Examples and Demo Files}
Numerical examples and demo files based on the interplay of refinement and coarsening are provided in subdirectories of the \texttt{ameshcoars}--toolbox:
\begin{itemize}
\item \texttt{example1/}: refinement along a moving circle,
\item \texttt{example2/}: adaptive finite element implementation following \cite{p1afem} for a quasi-stationary partial differential equation,
\item \texttt{example3/}: triangulation of a GIF, 
\item \texttt{example4/}: local coarsening of a uniformly refined triangulation.
\end{itemize}

\section{Numerical Experiments}\label{sect:experiments}
To demonstrate the efficiency of the developed \textsc{Matlab} code, we provide some numerical experiments performed on an Apple MacBook Air with a 1.6 GHz Intel Core i5, a RAM of 8 GB 1600MHz DDR3 on Mac OS High Sierra, version 10.13.6. Throughout, Matlab version 9.2.0 (R2017a) is used.

For the first experiment, we consider a refinement along a circle, which is implemented as a part in \texttt{example1/}, and coarsen this refinement by marking all elements for coarsening until the initial triangulation is restored. We measure with \textsc{Matlab}'s \matlab{tic/toc} 20 times the computational time needed to coarsen the refined circle until the initial triangulation is restored, and take the average of the measured times. The results are shown in Figure~\ref{fig:scale}. There the number of nodes and the computational time in seconds show a nearly linear behavior. We would like to emphasize that the time displayed for a number of nodes is the computational time required to perform a coarsening step with all elements marked. To determine the total computational time required to coarsen a mesh with a given number of nodes back to the initial triangulation, the times must be cumulated. One can see that the computational times for a coarsening step scale linearly with the number of nodes in that mesh.

%
%
\definecolor{mycolor1}{rgb}{0.00000,0.44700,0.74100}%
\definecolor{mycolor2}{rgb}{0.85000,0.32500,0.09800}%
\definecolor{mycolor3}{rgb}{0.92900,0.69400,0.12500}%
\definecolor{mycolor4}{rgb}{0.49400,0.18400,0.55600}%
\definecolor{mycolor5}{rgb}{0.46600,0.67400,0.18800}%
\definecolor{mycolor6}{rgb}{0.30100,0.74500,0.93300}%
\definecolor{mycolor7}{rgb}{0.63500,0.07800,0.18400}%
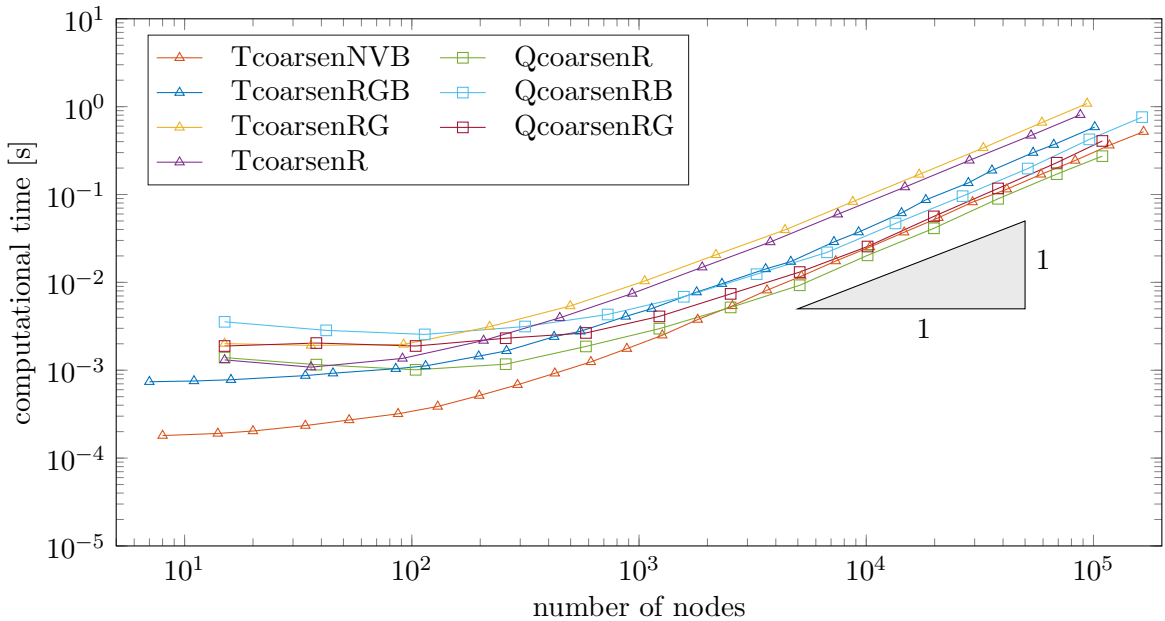
\begin{figure}
\begin{tikzpicture}

\begin{axis}[%
width=0.8\textwidth,
height=0.27\textheight,
at={(3.467in,1.367in)},
scale only axis,
xmode=log,
xmin=5,
xmax=200000,
xminorticks=true,
xlabel={number of nodes},
ymode=log,
ymin=1e-05,
ymax=10,
yminorticks=true,
ylabel={computational time [s]},
axis background/.style={fill=white},
legend columns=2,
legend style={column sep=0.3cm,legend pos = north west, legend cell align=left, align=left, draw=white!15!black}
]
\addplot [color=mycolor2, mark=triangle, mark options={solid, mycolor2}]
  table[row sep=crcr]{%
166190	0.5195997892\\
117548	0.36315741705\\
83104	0.2438660294\\
58780	0.1695051417\\
41554	0.114864975\\
29390	0.08203458315\\
20772	0.05435086115\\
14690	0.0371924118\\
10380	0.0252625229\\
7338	0.01748132315\\
5182	0.01176272185\\
3658	0.00814031715\\
2576	0.0053987895\\
1812	0.00377355695\\
1266	0.0025010359\\
884	0.00175936695\\
614	0.0012445871\\
426	0.0009229767\\
292	0.00068318685\\
198	0.0005145303\\
130	0.0003868203\\
87	0.00031890985\\
53	0.0002714325\\
34	0.0002344383\\
20	0.00020336695\\
14	0.00019078055\\
8	0.00018039365\\
};
\addlegendentry{TcoarsenNVB}

\addplot [color=mycolor5, mark=square, mark options={solid, mycolor5}]
  table[row sep=crcr]{%
109412	0.2734670956\\
68923	0.1705318288\\
38031	0.08938165115\\
19851	0.0412500901\\
10136	0.0203037156\\
5095	0.00926711665\\
2530	0.00518779745\\
1229	0.00297052115\\
584	0.0018655824\\
259	0.0011718106\\
104	0.0010114091\\
38	0.00115820785\\
15	0.00139229135\\
0	0\\
0	0\\
0	0\\
0	0\\
0	0\\
0	0\\
0	0\\
0	0\\
0	0\\
0	0\\
0	0\\
0	0\\
0	0\\
0	0\\
0	0\\
0	0\\
0	0\\
0	0\\
0	0\\
0	0\\
0	0\\
0	0\\
0	0\\
0	0\\
0	0\\
0	0\\
0	0\\
0	0\\
0	0\\
0	0\\
0	0\\
0	0\\
0	0\\
0	0\\
0	0\\
0	0\\
0	0\\
0	0\\
0	0\\
0	0\\
0	0\\
0	0\\
0	0\\
0	0\\
0	0\\
0	0\\
0	0\\
0	0\\
0	0\\
0	0\\
0	0\\
0	0\\
0	0\\
0	0\\
0	0\\
0	0\\
0	0\\
0	0\\
0	0\\
0	0\\
0	0\\
0	0\\
0	0\\
0	0\\
0	0\\
0	0\\
0	0\\
0	0\\
0	0\\
0	0\\
0	0\\
0	0\\
0	0\\
0	0\\
0	0\\
0	0\\
0	0\\
0	0\\
0	0\\
0	0\\
0	0\\
0	0\\
0	0\\
0	0\\
0	0\\
0	0\\
0	0\\
};
\addlegendentry{QcoarsenR}

\addplot [color=mycolor1, mark=triangle, mark options={solid, mycolor1}]
  table[row sep=crcr]{%
101330	0.58690345875\\
66883	0.3695228924\\
54243	0.2997602269\\
35756	0.18838079315\\
28141	0.13568310595\\
18289	0.08694636485\\
14297	0.0616365271\\
9256	0.03733339775\\
7207	0.02889003465\\
4655	0.0172263497\\
3608	0.0142030244\\
2317	0.009652451\\
1793	0.0077470757\\
1136	0.0050251457\\
874	0.0040918044\\
552	0.0027750994\\
423	0.0024071503\\
261	0.00166195165\\
197	0.00144440645\\
115	0.00111956135\\
85	0.0010406393\\
45	0.00092496635\\
34	0.0008684498\\
16	0.0007794369\\
11	0.0007544094\\
7	0.00073952635\\
0	0\\
0	0\\
0	0\\
0	0\\
0	0\\
0	0\\
0	0\\
0	0\\
0	0\\
0	0\\
0	0\\
0	0\\
0	0\\
0	0\\
0	0\\
0	0\\
0	0\\
0	0\\
0	0\\
0	0\\
0	0\\
0	0\\
0	0\\
0	0\\
0	0\\
0	0\\
0	0\\
0	0\\
0	0\\
0	0\\
0	0\\
0	0\\
0	0\\
0	0\\
0	0\\
0	0\\
0	0\\
0	0\\
0	0\\
0	0\\
0	0\\
0	0\\
0	0\\
0	0\\
0	0\\
0	0\\
0	0\\
0	0\\
0	0\\
0	0\\
0	0\\
0	0\\
0	0\\
0	0\\
0	0\\
0	0\\
0	0\\
0	0\\
0	0\\
0	0\\
0	0\\
0	0\\
0	0\\
0	0\\
0	0\\
0	0\\
0	0\\
0	0\\
0	0\\
0	0\\
0	0\\
0	0\\
0	0\\
0	0\\
};
\addlegendentry{TcoarsenRGB}

\addplot [color=mycolor6, mark=square, mark options={solid, mycolor6}]
  table[row sep=crcr]{%
163758	0.7599595943\\
95928	0.42433579375\\
51406	0.19734273115\\
26544	0.09579400305\\
13448	0.04685640555\\
6724	0.02196455\\
3292	0.01239303005\\
1574	0.00685701505\\
728	0.0043086685\\
315	0.00314602265\\
114	0.00255598505\\
42	0.0028430379\\
15	0.0035570589\\
0	0\\
0	0\\
0	0\\
0	0\\
0	0\\
0	0\\
0	0\\
0	0\\
0	0\\
0	0\\
0	0\\
0	0\\
0	0\\
0	0\\
0	0\\
0	0\\
0	0\\
0	0\\
0	0\\
0	0\\
0	0\\
0	0\\
0	0\\
0	0\\
0	0\\
0	0\\
0	0\\
0	0\\
0	0\\
0	0\\
0	0\\
0	0\\
0	0\\
0	0\\
0	0\\
0	0\\
0	0\\
0	0\\
0	0\\
0	0\\
0	0\\
0	0\\
0	0\\
0	0\\
0	0\\
0	0\\
0	0\\
0	0\\
0	0\\
0	0\\
0	0\\
0	0\\
0	0\\
0	0\\
0	0\\
0	0\\
0	0\\
0	0\\
0	0\\
0	0\\
0	0\\
0	0\\
0	0\\
0	0\\
0	0\\
0	0\\
0	0\\
0	0\\
0	0\\
0	0\\
0	0\\
0	0\\
0	0\\
0	0\\
0	0\\
0	0\\
0	0\\
0	0\\
0	0\\
0	0\\
0	0\\
0	0\\
0	0\\
0	0\\
0	0\\
0	0\\
0	0\\
};
\addlegendentry{QcoarsenRB}

\addplot [color=mycolor3, mark=triangle, mark options={solid, mycolor3}]
  table[row sep=crcr]{%
93755	1.0856479235\\
59308	0.65952968515\\
32756	0.33928194525\\
17097	0.16920570725\\
8729	0.0828804507\\
4389	0.039562535\\
2181	0.0205995019\\
1059	0.01036070185\\
498	0.00538420755\\
220	0.0031336146\\
92	0.00196487975\\
36	0.00191436405\\
15	0.00201055855\\
0	0\\
0	0\\
0	0\\
0	0\\
0	0\\
0	0\\
0	0\\
0	0\\
0	0\\
0	0\\
0	0\\
0	0\\
0	0\\
0	0\\
0	0\\
0	0\\
0	0\\
0	0\\
0	0\\
0	0\\
0	0\\
0	0\\
0	0\\
0	0\\
0	0\\
0	0\\
0	0\\
0	0\\
0	0\\
0	0\\
0	0\\
0	0\\
0	0\\
0	0\\
0	0\\
0	0\\
0	0\\
0	0\\
0	0\\
0	0\\
0	0\\
0	0\\
0	0\\
0	0\\
0	0\\
0	0\\
0	0\\
0	0\\
0	0\\
0	0\\
0	0\\
0	0\\
0	0\\
0	0\\
0	0\\
0	0\\
0	0\\
0	0\\
0	0\\
0	0\\
0	0\\
0	0\\
0	0\\
0	0\\
0	0\\
0	0\\
0	0\\
0	0\\
0	0\\
0	0\\
0	0\\
0	0\\
0	0\\
0	0\\
0	0\\
0	0\\
0	0\\
0	0\\
0	0\\
0	0\\
0	0\\
0	0\\
0	0\\
0	0\\
0	0\\
0	0\\
0	0\\
};
\addlegendentry{TcoarsenRG}

\addplot [color=mycolor7, mark=square, mark options={solid, mycolor7}]
  table[row sep=crcr]{%
109412	0.4069204412\\
68923	0.22996601805\\
38031	0.1172187833\\
19851	0.05637060355\\
10136	0.0255958446\\
5095	0.013106394\\
2530	0.0073964186\\
1229	0.00409977215\\
584	0.0026565655\\
259	0.0023128485\\
104	0.001887278\\
38	0.00203877595\\
15	0.00188622625\\
0	0\\
0	0\\
0	0\\
0	0\\
0	0\\
0	0\\
0	0\\
0	0\\
0	0\\
0	0\\
0	0\\
0	0\\
0	0\\
0	0\\
0	0\\
0	0\\
0	0\\
0	0\\
0	0\\
0	0\\
0	0\\
0	0\\
0	0\\
0	0\\
0	0\\
0	0\\
0	0\\
0	0\\
0	0\\
0	0\\
0	0\\
0	0\\
0	0\\
0	0\\
0	0\\
0	0\\
0	0\\
0	0\\
0	0\\
0	0\\
0	0\\
0	0\\
0	0\\
0	0\\
0	0\\
0	0\\
0	0\\
0	0\\
0	0\\
0	0\\
0	0\\
0	0\\
0	0\\
0	0\\
0	0\\
0	0\\
0	0\\
0	0\\
0	0\\
0	0\\
0	0\\
0	0\\
0	0\\
0	0\\
0	0\\
0	0\\
0	0\\
0	0\\
0	0\\
0	0\\
0	0\\
0	0\\
0	0\\
0	0\\
0	0\\
0	0\\
0	0\\
0	0\\
0	0\\
0	0\\
0	0\\
0	0\\
0	0\\
0	0\\
0	0\\
0	0\\
0	0\\
};
\addlegendentry{QcoarsenRG}

\addplot [color=mycolor4, mark=triangle, mark options={solid, mycolor4}]
  table[row sep=crcr]{%
87548	0.80796091445\\
53101	0.4704932151\\
28462	0.2447355005\\
14792	0.1214922627\\
7498	0.05958132755\\
3782	0.02880449005\\
1897	0.01487452085\\
933	0.0074716319\\
448	0.00395244555\\
207	0.0021783632\\
91	0.0013623355\\
36	0.00108307005\\
15	0.00131079795\\
0	0\\
0	0\\
0	0\\
0	0\\
0	0\\
0	0\\
0	0\\
0	0\\
0	0\\
0	0\\
0	0\\
0	0\\
0	0\\
0	0\\
0	0\\
0	0\\
0	0\\
0	0\\
0	0\\
0	0\\
0	0\\
0	0\\
0	0\\
0	0\\
0	0\\
0	0\\
0	0\\
0	0\\
0	0\\
0	0\\
0	0\\
0	0\\
0	0\\
0	0\\
0	0\\
0	0\\
0	0\\
0	0\\
0	0\\
0	0\\
0	0\\
0	0\\
0	0\\
0	0\\
0	0\\
0	0\\
0	0\\
0	0\\
0	0\\
0	0\\
0	0\\
0	0\\
0	0\\
0	0\\
0	0\\
0	0\\
0	0\\
0	0\\
0	0\\
0	0\\
0	0\\
0	0\\
0	0\\
0	0\\
0	0\\
0	0\\
0	0\\
0	0\\
0	0\\
0	0\\
0	0\\
0	0\\
0	0\\
0	0\\
0	0\\
0	0\\
0	0\\
0	0\\
0	0\\
0	0\\
0	0\\
0	0\\
0	0\\
0	0\\
0	0\\
0	0\\
0	0\\
};
\addlegendentry{TcoarsenR}

\addplot[area legend, draw=black, fill=white!80!black, fill opacity=0.4]
table[row sep=crcr] {%
x	y\\
5000	 0.005\\
50000	0.005\\
50000	0.05\\
5000	 0.005\\
}--cycle;

\node[below, align=center]
at (axis cs:60000,0.03) {1};
\node[right, align=left]
at (axis cs:15000,0.003) {1};

\end{axis}
\end{tikzpicture}%
\caption{Computational times in seconds vs.~number of nodes. We observe a nearly linear behavior between the number of nodes and the computational time. The newest vertex bisection implemented in \cite{p1afem} is a bit more performant because no further $\clos$ step is necessary.}
\label{fig:scale}
\end{figure}

What is noticeable in Figure~\ref{fig:scale} is that {\TRGB} and {\TNVB} have many more data points than the other strategies. To this end, we have studied the interaction between refinement and coarsening. Specifically, we wanted to know how many refinement steps are required to preserve the adaptive mesh, and on the contrary, how many coarsening steps are required to restore the original mesh with all elements marked for coarsening. We have already indicated that in the case of the red, red-green and red-blue strategies the number of coarsening steps is identical to the number of refinement steps. This is due to the hierarchical structure of the meshes. In a coarsening step, less coarsening is generally possible than was refined in the last refinement step. This is because in coarsening 1-irregularity and the $d$-Neighbor Rule are ensured by blocking the coarsening. With the refinement, on the other hand, this is achieved by further refinements. In general, the number of coarsening steps is not less than the number of refinement steps needed to restore the initial mesh. To see this, Figure~\ref{fig:refvscoars} shows the refinement and coarsening steps as a function of the number of elements. The refinement curve shows the number of elements achieved in each refinement step. For a better comparison, we have reversed the coarsening steps and the corresponding number of elements. Normally for coarsening we would start with a mesh with a large number of elements and decrease with each step. Here we have simply reversed the data points beginning with the smallest to the highest number of elements to get a better comparison of refinement and coarsening.

%
%
\definecolor{mycolor1}{rgb}{0.00000,0.44700,0.74100}%
\definecolor{mycolor2}{rgb}{0.85000,0.32500,0.09800}%
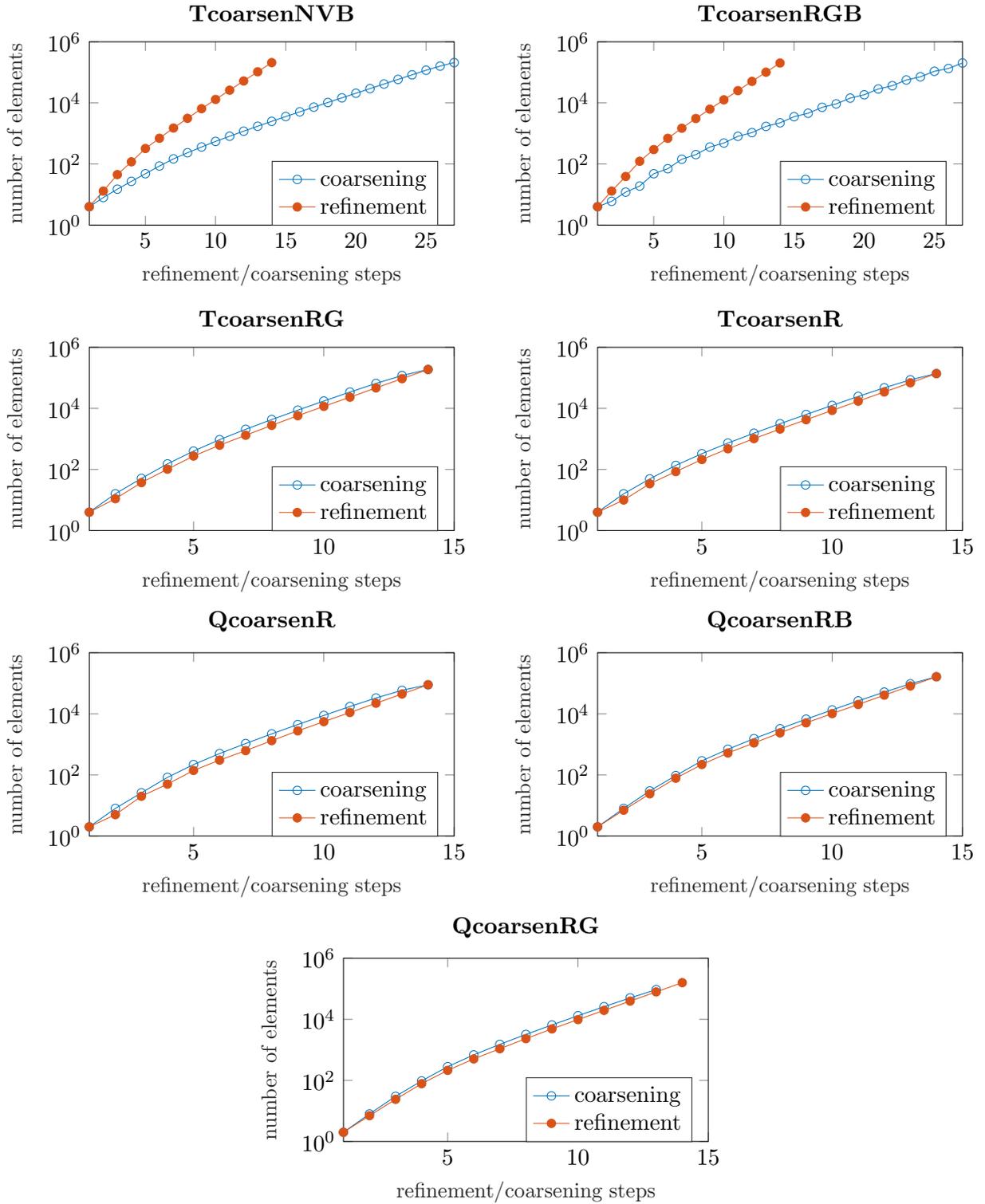
\begin{figure}
\begin{tikzpicture}

\begin{axis}[%
width=0.35\textwidth,
height=1.2in,
at={(0.7in,9.068in)},
scale only axis,
xmin=1,
xmax=27,
xlabel style={font=\color{white!15!black}},
xlabel={\small refinement/coarsening steps},
ymode=log,
ymin=1,
ymax=1000000,
yminorticks=false,
ylabel style={font=\color{white!15!black}},
ylabel={\small number of elements},
axis background/.style={fill=white},
title style={font=\bfseries},
title={TcoarsenNVB},
legend style={at={(0.5,0.35)}, anchor=north west, legend cell align=left, align=left, draw=white!15!black}
]

\addplot [color=mycolor1, mark=o, mark options={solid, mycolor1}]
  table[row sep=crcr]{%
1	4\\
2	8\\
3	15\\
4	27\\
5	48\\
6	86\\
7	147\\
8	233\\
9	360\\
10	548\\
11	810\\
12	1186\\
13	1720\\
14	2484\\
15	3570\\
16	5098\\
17	7258\\
18	10306\\
19	14614\\
20	20698\\
21	29312\\
22	41476\\
23	58706\\
24	83034\\
25	117482\\
26	160094\\
27	208746\\
};
\addlegendentry{coarsening}

\addplot [color=mycolor2, mark=*, mark options={solid, mycolor2}]
  table[row sep=crcr]{%
1	4\\
2	13\\
3	45\\
4	118\\
5	321\\
6	696\\
7	1502\\
8	3132\\
9	6462\\
10	12986\\
11	26042\\
12	52148\\
13	104370\\
14	208746\\
};
\addlegendentry{refinement}

\end{axis}

\begin{axis}[%
width=0.35\textwidth,
height=1.2in,
at={(4in,9.068in)},
scale only axis,
xmin=1,
xmax=27,
xminorticks=true,
xlabel style={font=\color{white!15!black}},
xlabel={\small refinement/coarsening steps},
ymode=log,
ymin=1,
ymax=1000000,
yminorticks=true,
ylabel style={font=\color{white!15!black}},
ylabel={\small number of elements},
axis background/.style={fill=white},
title style={font=\bfseries},
title={TcoarsenRGB},
legend style={at={(0.5,0.35)}, anchor=north west, legend cell align=left, align=left, draw=white!15!black}
]
\addplot [color=mycolor1, mark=o, mark options={solid, mycolor1}]
  table[row sep=crcr]{%
1	4\\
2	6\\
3	12\\
4	19\\
5	48\\
6	70\\
7	143\\
8	203\\
9	359\\
10	486\\
11	804\\
12	1062\\
13	1701\\
14	2224\\
15	3533\\
16	4581\\
17	7159\\
18	9253\\
19	14353\\
20	18451\\
21	28528\\
22	36512\\
23	56212\\
24	71442\\
25	108413\\
26	133691\\
27	202582\\
};
\addlegendentry{coarsening}

\addplot [color=mycolor2, mark=*, mark options={solid, mycolor2}]
  table[row sep=crcr]{%
1	4\\
2	13\\
3	39\\
4	123\\
5	297\\
6	693\\
7	1482\\
8	3085\\
9	6239\\
10	12597\\
11	25221\\
12	50589\\
13	101226\\
14	202582\\
};
\addlegendentry{refinement}

\end{axis}

\begin{axis}[%
width=0.35\textwidth,
height=1.2in,
at={(0.7in,7.068in)},
scale only axis,
xmin=1,
xmax=15,
xminorticks=true,
xlabel style={font=\color{white!15!black}},
xlabel={\small refinement/coarsening steps},
ymode=log,
ymin=1,
ymax=1000000,
yminorticks=true,
ylabel style={font=\color{white!15!black}},
ylabel={\small number of elements},
axis background/.style={fill=white},
title style={font=\bfseries},
title={TcoarsenRG},
legend style={at={(0.5,0.35)}, anchor=north west, legend cell align=left, align=left, draw=white!15!black}
]
\addplot [color=mycolor1, mark=o, mark options={solid, mycolor1}]
  table[row sep=crcr]{%
1	4\\
2	16\\
3	51\\
4	152\\
5	399\\
6	944\\
7	2060\\
8	4299\\
9	8710\\
10	17385\\
11	34116\\
12	65429\\
13	118528\\
14	187419\\
};
\addlegendentry{coarsening}

\addplot [color=mycolor2, mark=*, mark options={solid, mycolor2}]
  table[row sep=crcr]{%
1	4\\
2	11\\
3	37\\
4	102\\
5	274\\
6	622\\
7	1315\\
8	2777\\
9	5736\\
10	11597\\
11	23186\\
12	46655\\
13	93533\\
14	187419\\
};
\addlegendentry{refinement}

\end{axis}

\begin{axis}[%
width=0.35\textwidth,
height=1.2in,
at={(4in,7.068in)},
scale only axis,
xmin=1,
xmax=15,
xminorticks=true,
xlabel style={font=\color{white!15!black}},
xlabel={\small refinement/coarsening steps},
ymode=log,
ymin=1,
ymax=1000000,
yminorticks=true,
ylabel style={font=\color{white!15!black}},
ylabel={\small number of elements},
axis background/.style={fill=white},
title style={font=\bfseries},
title={TcoarsenR},
legend style={at={(0.5,0.35)}, anchor=north west, legend cell align=left, align=left, draw=white!15!black}
]
\addplot [color=mycolor1, mark=o, mark options={solid, mycolor1}]
  table[row sep=crcr]{%
1	4\\
2	16\\
3	49\\
4	136\\
5	328\\
6	730\\
7	1537\\
8	3133\\
9	6289\\
10	12490\\
11	24568\\
12	47089\\
13	85840\\
14	137515\\
};
\addlegendentry{coarsening}

\addplot [color=mycolor2, mark=*, mark options={solid, mycolor2}]
  table[row sep=crcr]{%
1	4\\
2	10\\
3	34\\
4	85\\
5	211\\
6	475\\
7	1021\\
8	2095\\
9	4246\\
10	8551\\
11	17152\\
12	34369\\
13	68737\\
14	137515\\
};
\addlegendentry{refinement}

\end{axis}

\begin{axis}[%
width=0.35\textwidth,
height=1.2in,
at={(0.7in,5.068in)},
scale only axis,
xmin=1,
xmax=15,
xminorticks=true,
xlabel style={font=\color{white!15!black}},
xlabel={\small refinement/coarsening steps},
ymode=log,
ymin=1,
ymax=1000000,
yminorticks=true,
ylabel style={font=\color{white!15!black}},
ylabel={\small number of elements},
axis background/.style={fill=white},
title style={font=\bfseries},
title={QcoarsenR},
legend style={at={(0.5,0.35)}, anchor=north west, legend cell align=left, align=left, draw=white!15!black}
]
\addplot [color=mycolor1, mark=o, mark options={solid, mycolor1}]
  table[row sep=crcr]{%
1	2\\
2	8\\
3	26\\
4	83\\
5	218\\
6	500\\
7	1064\\
8	2204\\
9	4448\\
10	8852\\
11	17300\\
12	32948\\
13	58778\\
14	89144\\
};
\addlegendentry{coarsening}

\addplot [color=mycolor2, mark=*, mark options={solid, mycolor2}]
  table[row sep=crcr]{%
1	2\\
2	5\\
3	20\\
4	50\\
5	140\\
6	305\\
7	623\\
8	1322\\
9	2762\\
10	5534\\
11	11036\\
12	22202\\
13	44486\\
14	89144\\
};
\addlegendentry{refinement}

\end{axis}

\begin{axis}[%
width=0.35\textwidth,
height=1.2in,
at={(4in,5.068in)},
scale only axis,
xmin=1,
xmax=15,
xminorticks=true,
xlabel style={font=\color{white!15!black}},
xlabel={\small refinement/coarsening steps},
ymode=log,
ymin=1,
ymax=1000000,
yminorticks=true,
ylabel style={font=\color{white!15!black}},
ylabel={\small number of elements},
axis background/.style={fill=white},
title style={font=\bfseries},
title={QcoarsenRB},
legend style={at={(0.5,0.35)}, anchor=north west, legend cell align=left, align=left, draw=white!15!black}
]
\addplot [color=mycolor1, mark=o, mark options={solid, mycolor1}]
  table[row sep=crcr]{%
1	2\\
2	8\\
3	30\\
4	95\\
5	289\\
6	693\\
7	1533\\
8	3245\\
9	6673\\
10	13393\\
11	26485\\
12	51343\\
13	95861\\
14	163687\\
};
\addlegendentry{coarsening}

\addplot [color=mycolor2, mark=*, mark options={solid, mycolor2}]
  table[row sep=crcr]{%
1	2\\
2	7\\
3	24\\
4	78\\
5	219\\
6	521\\
7	1111\\
8	2371\\
9	5055\\
10	10145\\
11	20213\\
12	40823\\
13	81595\\
14	163687\\
};
\addlegendentry{refinement}

\end{axis}

\begin{axis}[%
width=0.35\textwidth,
height=1.2in,
at={(2.35in,3.068in)},
scale only axis,
xmin=1,
xmax=15,
xminorticks=true,
xlabel style={font=\color{white!15!black}},
xlabel={\small refinement/coarsening steps},
ymode=log,
ymin=1,
ymax=1000000,
yminorticks=true,
ylabel style={font=\color{white!15!black}},
ylabel={\small number of elements},
axis background/.style={fill=white},
title style={font=\bfseries},
title={QcoarsenRG},
legend style={at={(0.5,0.35)}, anchor=north west, legend cell align=left, align=left, draw=white!15!black}
]
\addplot [color=mycolor1, mark=o, mark options={solid, mycolor1}]
  table[row sep=crcr]{%
1	2\\
2	8\\
3	30\\
4	96\\
5	281\\
6	683\\
7	1511\\
8	3187\\
9	6509\\
10	13067\\
11	25817\\
12	50073\\
13	93121\\
14	0\\
};
\addlegendentry{coarsening}

\addplot [color=mycolor2, mark=*, mark options={solid, mycolor2}]
  table[row sep=crcr]{%
1	2\\
2	7\\
3	24\\
4	78\\
5	213\\
6	504\\
7	1084\\
8	2321\\
9	4829\\
10	9749\\
11	19541\\
12	39329\\
13	78839\\
14	157931\\
};
\addlegendentry{refinement}

\end{axis}

\end{tikzpicture}%
\caption{Refinement/Coarsening steps vs.~number of elements. In coarsening, the data points start from a high number of elements and decrease with the number of coarsening steps, whereas in refinement, the opposite is true. For a better comparison between refinement and coarsening steps, a reversed data set for coarsening is plotted. We see that {\TNVB} and {\TRGB} use more coarsening steps than the refinement. It is typical for these strategies that twice as many coarsening steps are needed as refinement steps. The hierarchical structure of red meshes ensures an identical number of refinement/coarsening steps. The coarsening is nevertheless more restrained in the intermediate steps than in the refinement. This is not surprising, as coarsening can be blocked in the $\clos$ step, whereas further refinements are made in the refinement step to guarantee 1-irregularity and the $d$-Neighbor Rule. Un-/Closure of the meshes with green or blue patterns does not change this behavior. }
\label{fig:refvscoars}
\end{figure}

{\TRGB} and {\TNVB} require more coarsening steps than refinement steps. This can easily be explained by the fact that a blue, bisec(3) and possibly a red refinement is coarsened in a two-step process that reverses green refinements. In the refinement process, these patterns were introduced within one refinement step. This leads to the difference in the number of refinement and coarsening steps. A more detailed discussion is given in \cite{RGB}.

In the toolbox, some examples are implemented that present our coarsening algorithm in different contexts. In this work, we restrict ourselves to present local coarsening implemented in \texttt{example4/} because the presentation of moving circles, singularities or a triangulation of an image sequence is not suitable for a presentation in a paper. The interested reader should call up the sample implementations to convince himself of the functionality of the toolbox. In \texttt{example4/} a mesh is first refined uniformly and with a local marking strategy elements are coarsened only in a specific region, see Figure~\ref{fig:localcoarsening}.

\begin{figure}

\begin{minipage}{0.4\textwidth}
\centering \small \bfseries TcoarsenNVB\\
\vspace*{1ex}
\includegraphics[width=\textwidth]{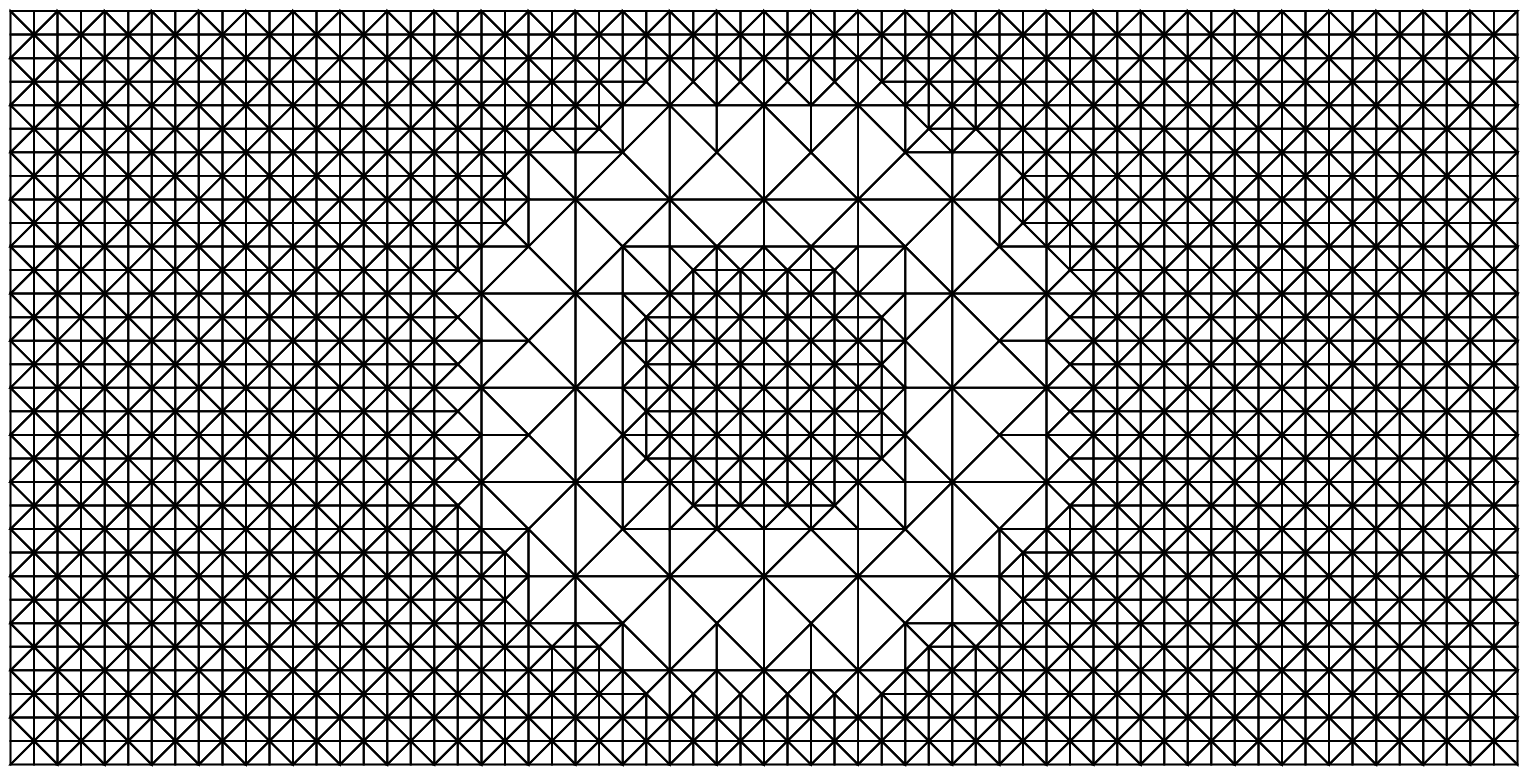}
\end{minipage}\hspace*{2ex}
\begin{minipage}{0.4\textwidth}
\centering \small \bfseries TcoarsenRGB\\
\vspace*{1ex}
\includegraphics[width=\textwidth]{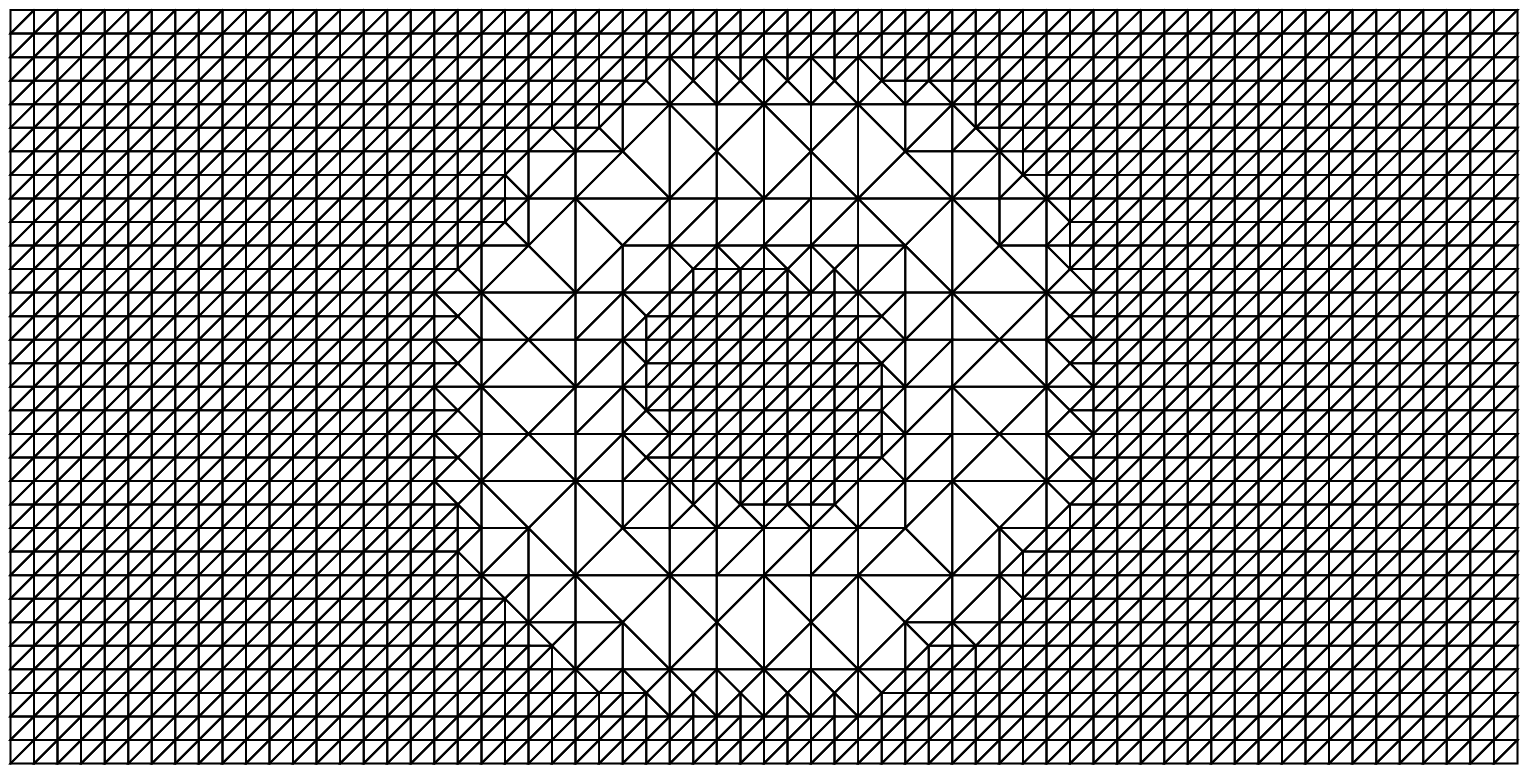}
\end{minipage}\\
\vspace*{1ex}
\begin{minipage}{0.4\textwidth}
\centering \small \bfseries TcoarsenRG\\
\vspace*{1ex}
\includegraphics[width=\textwidth]{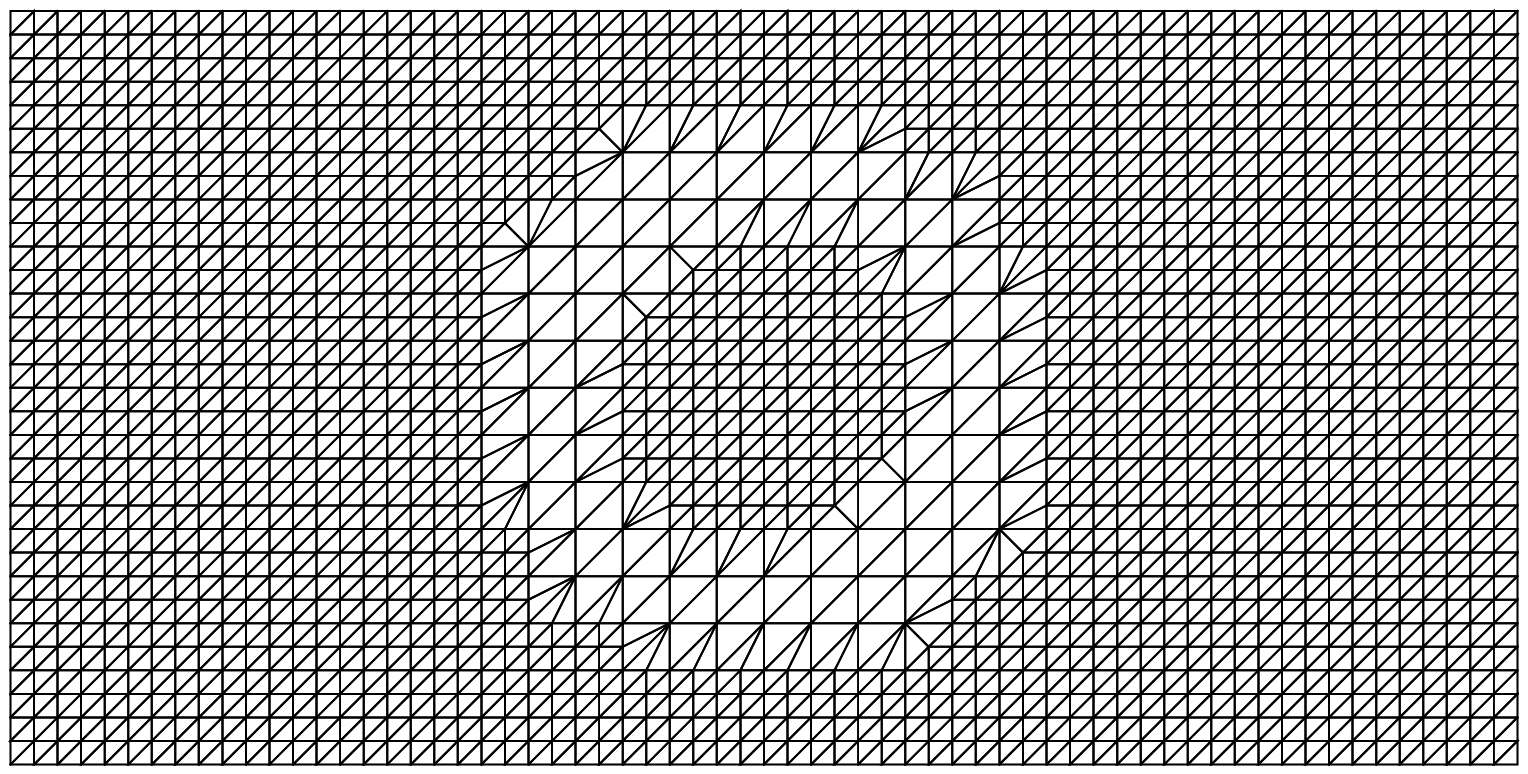}
\end{minipage}\hspace*{2ex}
\begin{minipage}{0.4\textwidth}
\centering \small \bfseries TcoarsenR\\
\vspace*{1ex}
\includegraphics[width=\textwidth]{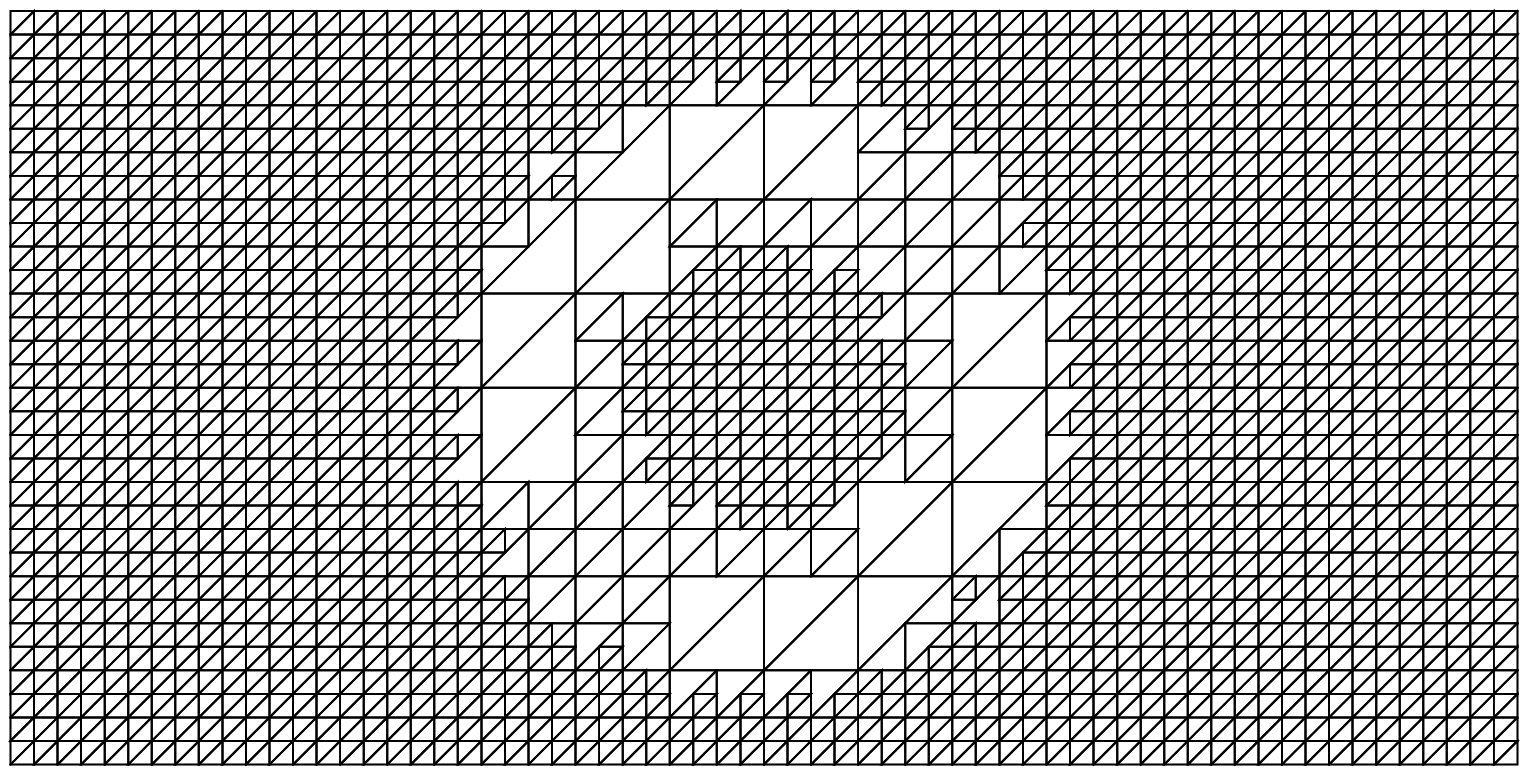}
\end{minipage}\\
\vspace*{1ex}
\begin{minipage}{0.4\textwidth}
\centering \small \bfseries QcoarsenR\\
\vspace*{1ex}
\includegraphics[width=\textwidth]{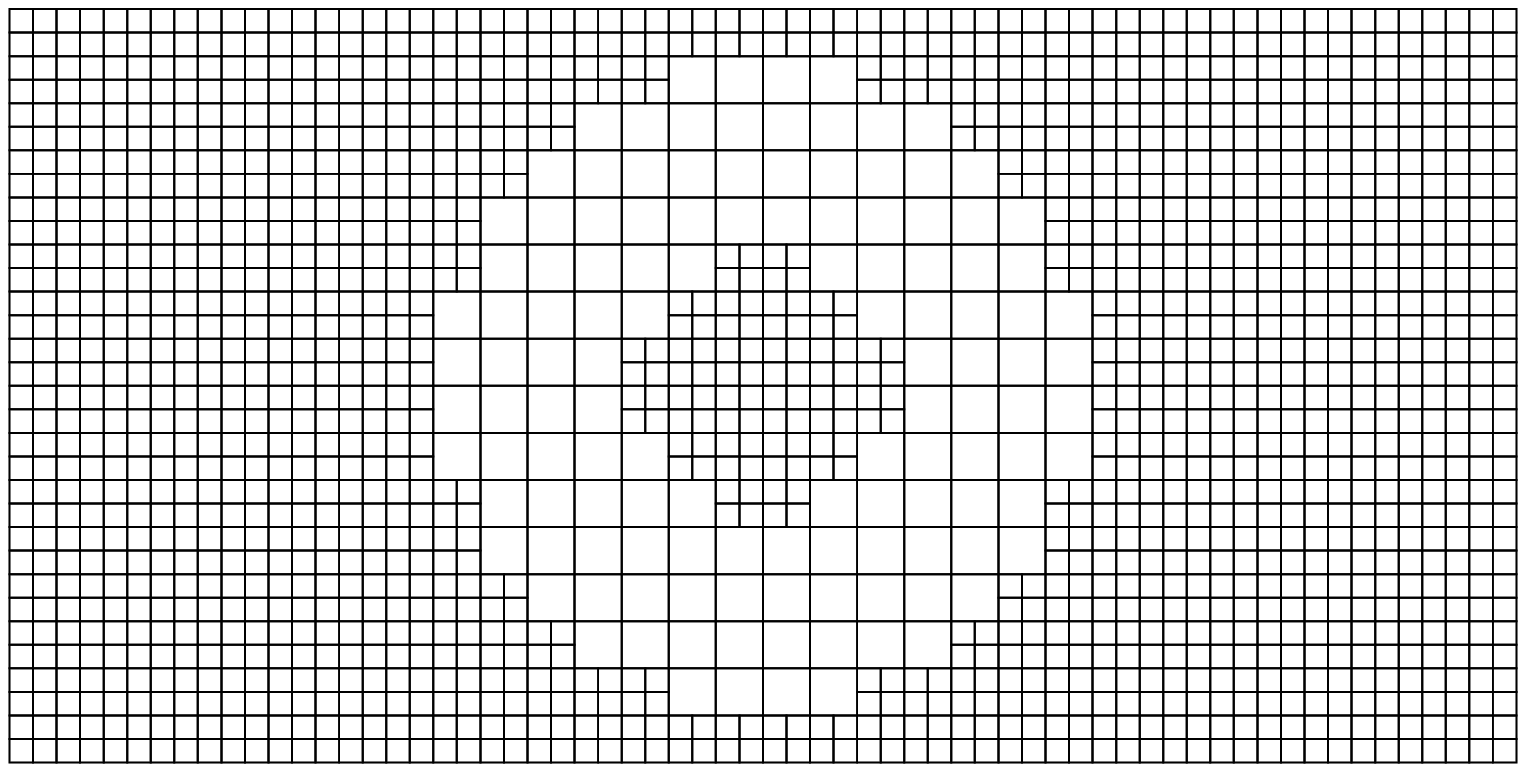}
\end{minipage}\hspace*{2ex}
\begin{minipage}{0.4\textwidth}
\centering \small \bfseries QcoarsenRB\\
\vspace*{1ex}
\includegraphics[width=\textwidth]{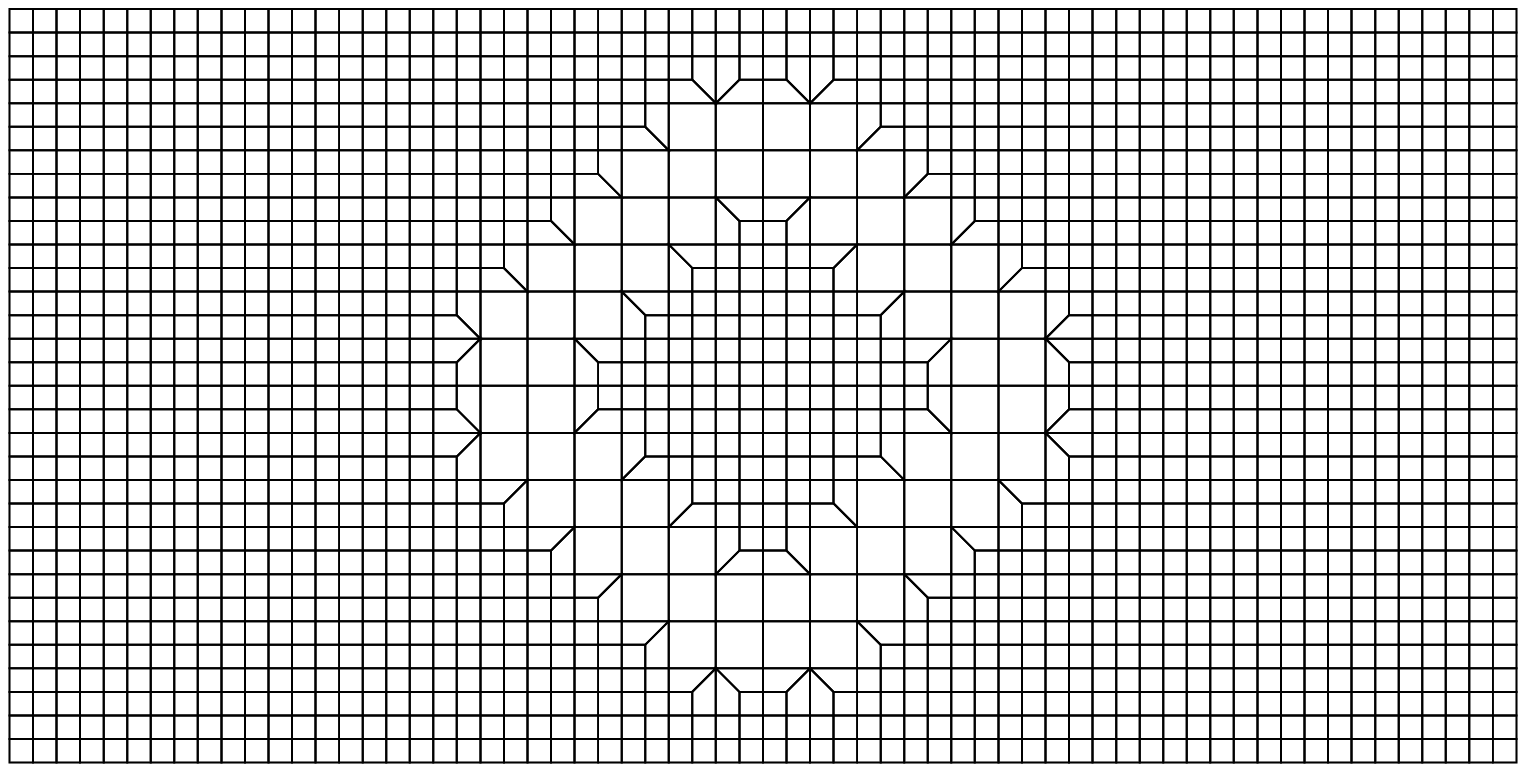}
\end{minipage}\\
\vspace*{1ex}
\centering \begin{minipage}{0.4\textwidth}
\centering \small \bfseries QcoarsenRG\\
\vspace*{1ex}
\includegraphics[width=\textwidth]{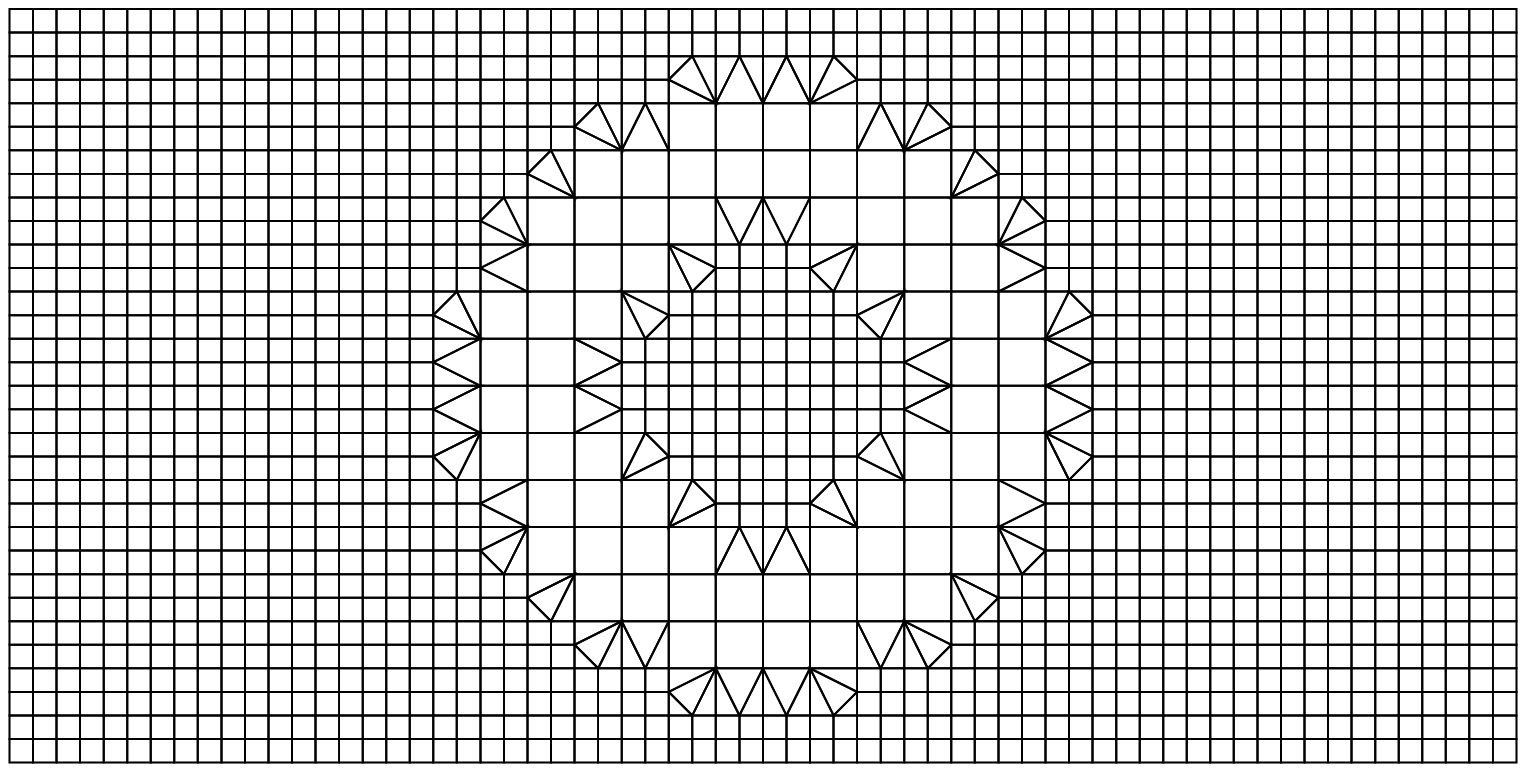}
\end{minipage}
\caption{Local Coarsening. A uniformly refined mesh is coarsened locally in the shape of a circle. The resulting meshes are shown for each mesh coarsening strategy.}
\label{fig:localcoarsening}
\end{figure}

\FloatBarrier
\bibliographystyle{acm}
\bibliography{references}
\end{document}